\documentclass{article}
\usepackage{amsmath,amssymb,amsxtra}
\usepackage{mathrsfs}
\usepackage{float,verbatim}   
\usepackage{geometry}
\usepackage{threeparttable,booktabs}
\usepackage{graphicx,graphics}
\graphicspath{{./figs/}}
\usepackage{epstopdf}
\usepackage{makecell}
\usepackage{mathtools}
\usepackage{subfig}
\usepackage{algorithm,algorithmic}
\usepackage{color}
\usepackage{latexsym,float,tikz}
\usepackage{array}
\usepackage{dsfont}
\usepackage{bm}
\usepackage{hyperref}
\hypersetup{colorlinks=true,linkcolor=blue,citecolor=blue,urlcolor=blue}

\usepackage{amsthm}
\theoremstyle{plain}
\newtheorem{theorem}{Theorem}[section]
\newtheorem{lemma}[theorem]{Lemma}
\newtheorem{proposition}[theorem]{Proposition}

\newtheorem{example}[theorem]{Example}

\theoremstyle{remark}
\newtheorem{remark}[theorem]{Remark}

\renewcommand{\d}{{\rm \,d}}
\def \mcd {{\mathcal D}}
\def \mce {{\mathcal E}}

\def \mbc {{\mathbb C}}

\def \mbr {{\mathbb R}}
\def \mbs {{\mathbb S}}

\def \beqq {\begin{equation}}
\def \eeqq {\end{equation}}
\def \bpf {\begin{proof}}
\def \epf {\end{proof}}
\def \beq {\begin{equation*}}
\def \eeq {\end{equation*}} 
\def \eps {\epsilon}   
   
\def \La {\Lambda}    
   
\def \lap {\Delta}
\def \p {\partial}
\def \ha {\frac{1}{2}}
\numberwithin{equation}{section}
\renewcommand{\d}{\,\mathrm{d}}
\allowdisplaybreaks

\allowdisplaybreaks

\title{{Stability of Electrical Impedance Tomography with Anisotropies and its Application to the Deep Calde\'on Method}\thanks{T. Hu is partly supported by NSFC (Project 125B2022). B. Jin is partly supported by Hong Kong RGC General Research Fund (14306824) and ANR / Hong Kong RGC Joint Research Scheme (A-CUHK402/24), NSFC / RGC Joint Research Scheme (N\_CUHK446/25) and a start-up fund from The Chinese University of Hong Kong. Y. Wang is partly supported by NSF under grant DMS-2508559.}}
\date{}
\author{Tianhao Hu\thanks{Department of Mathematics, The Chinese University of Hong Kong, Shatin, N.T., Hong Kong, P.R. China (\texttt{thhu@link.cuhk.edu.hk}, \texttt{b.jin@cuhk.edu.hk})} \and Bangti Jin\footnotemark[2] \thanks{AI for Science and Engineering Center, Shenzhen Loop Area Institute, Shenzhen, China} 
\and Yiran Wang\thanks{Department of Mathematics, Emory University, 400 Dowman Drive, Atlanta, GA, 30322, USA
(\texttt{yiran.wang@emory.edu})}}

\begin{document}

\maketitle
\begin{abstract}
In this work, we establish new conditional Lipschitz stability results for electrical impedance tomography (EIT) with anisotropies, of recovering the conductivity in a conformal class of a known anisotropic conductivity in both two- and multi-dimensional cases. Then we employ the stability theory to understand the property of the deep Calder\'on method, one deep learning-based technique for image reconstruction in EIT that has shown promising empirical results, but still lacks theoretical underpinnings. Specifically, we relate the stability theory to the robustness of the method with the proper choice of the training data, and present numerical results in two-dimension to complement the theoretical analysis. 
\end{abstract}



\section{Introduction} 
Electrical impedance tomography (EIT) is a noninvasive imaging technique exploiting the electric properties (conductivity and permittivity) of the imaging object. We refer to \cite{FSU} for the physical and medical background of the problem. Typically, one makes the assumption that the conductivity is isotropic. However, there are many important examples of anisotropic conductors, e.g., muscle tissues in the human body and skull structures in head imaging \cite{FosterSchwan:1989,AmmariGarnier:2016}. Mathematically, the anisotropic EIT problem in the two-dimensional case can be described as follows. Let $\Omega \subset\mathbb{R}^2$ be an open bounded domain with a smooth boundary $\p \Omega$. We use $x = (x_1, x_2)$ for the coordinates of $\mbr^2$. Let $\gamma = (\gamma_{ij})_{i, j = 1}^2$ be an anisotropic electrical impedance and $u$ be the electrical potential. Then $u$ satisfies the following conductivity equation 
\begin{equation}\label{eq-eit}
\left\{\begin{aligned}
\nabla \cdot (\gamma(x) \nabla u(x))  = \sum_{i, j = 1}^2 \frac{\p }{\p x_i}\Big(\gamma_{ij}(x)\frac{ \p u} {\p x_j} (x)\Big) &= 0, \quad \text{in } \Omega,\\
u & = f,\quad \mbox{on }\partial\Omega.
\end{aligned}\right.
\end{equation}
Under the standard boundedness and ellipticity conditions on $\gamma$, given a voltage distribution $f$ on the boundary $\p \Omega$,  there exists a unique solution $u$ of problem \eqref{eq-eit}. Then we measure the current density $g$ on $\p \Omega$, given by 
\beq
g(x) = \nu(x)\cdot \gamma(x) \nabla u(x)|_{\p \Omega} = \sum_{i, j = 1}^2 \nu_i(x) \gamma_{ij}(x) \frac{\p u}{\p x_j}(x)|_{\p \Omega},
\eeq
where $\nu(x) = (\nu_1(x), \nu_2(x))\in \mbs^1$ denotes the unit outward normal vector at the point $x\in \partial\Omega$ to the boundary $\p \Omega$.  For the EIT problem, the measurement is the Dirichlet-to-Neumann map  
\beqq\label{eq-dtn}
\La_\gamma: f(x) \rightarrow g(x), 
\eeqq
and the goal is to recover the conductivity $\gamma$ from the Dirichlet-to-Neumann map $\La_\gamma.$
 
In this work, we are interested in numerical methods for reconstructing conductivity images for the EIT. This is particularly challenging because the EIT inverse problem is severely ill-posed. For the isotropic problem, a logarithmic type stability estimate was proved  in \cite{Liu}.  More precisely, for two conductivities $\gamma_j = \sigma_j I$, $j = 1, 2$, where $I\in\mathbb{R}^{2\times 2}$ denotes the identity matrix and $\sigma_j$ are positive and bounded scalar functions, there holds 
\beqq\label{eq-logest} 
\|\sigma_1 - \sigma_2\|_{L^\infty(\Omega)} \leq C \omega(\|\La_{\gamma_1} - \La_{\gamma_2}\|_{*}), 
\eeqq
where $\|\cdot\|_*$ denotes the operator norm from $H^\ha(\p \Omega)$ to $H^{-\ha}(\p \Omega)$ and the function $\omega$ is such that 
$\omega(\tau)\leq |\log \tau|^{-\kappa}$ for some $\kappa>0$ and sufficiently small $\tau>0.$  It is later proved in \cite{Man} that such an estimate is optimal. In practice, reconstruction algorithms for EIT often rely on  variational regularization \cite{EnglHankeNeubauer:1996,ItoJin:2015}, which constructs an approximation by minimizing an objective functional that consists of a data fitting term and a penalty term. The penalty term is designed to tackle the inherent ill-posedness of the inverse problem.

Recently, there is significant interest in deep learning based methods for EIT image reconstruction. This is driven by the significant advances in dedicated computing architecture, the availability of paired training data and the development of novel training algorithms etc. In the last few years, several deep learning based approaches have been proposed, e.g., direct inversion (e.g.\ using fully connected neural networks \cite{chen2020deep}, operator learning \cite{Las}), postprocessing type (e.g., deep D-bar method \cite{Ham} and deep Calder\'on method \cite{Cen}), deep direct sampling method \cite{GuoJiang}, and algorithmic unrolling (e.g., deep gradient method and deep Gauss-Newton method) \cite{ChenZhouYang:2022,colibazzi2023deep}. We refer interested readers to the reviews \cite{Dim,Tanyu} for detailed discussions, including comparative studies with more traditional methods. In particular, postprocessing type methods first employ a traditional reconstruction method to obtain an initial guess and then postprocess the initial reconstruction using convolutional neural networks that are supervisedly trained using paired training dataset. By suitably combining the physical knowledge with a priori knowledge encoded in the training dataset (learned by the training procedure), this class of methods exhibits strong empirical performance and thus has been extensively employed.

When solving inverse problems via deep learning, especially for medical imaging applications, instability phenomena, e.g., hallucination may occur \cite{Ant, Sze}, when the test data deviates from the distribution of training data, even if only very slightly. Recent studies (e.g. \cite{Col, Got, Sca}) indicate that any stable and accurate reconstruction procedure must be ``kernel aware". More precisely, consider an under-determined system of noisy linear equations 
$y = A x + e \in \mbc^m$,
where $A\in \mbc^{m\times n}$ represents a sampling model with $m < n$ and $e$ denotes the measurement error. The task is to recover $x\in \mbc^n$ from $y\in \mbc^m$, given $A\in \mathbb{C}^{m\times n}$. A reconstructor $\Psi: \mbc^m\rightarrow \mbc^n$ is said to lack kernel awareness if it  can approximately recover two vectors $x,x'\in\mathbb{C}^n$ in the sense that \cite[p. 5]{Col}
\beq
\|\Psi(A x) - x\| \leq \eps\quad \text{and}\quad\|\Psi(A x') - x'\| \leq \eps,
\eeq
whose difference $\|x - x'\| \gg 2\eps$ is large, but where the difference $x-x'$ lies close to the null space of $A$ (which is nontrivial due to the condition $m<n$) so that $\| A(x - x')\| < \eps.$ In the absence of kernel awareness, we have
\beq
\|\Psi(Ax) - \Psi(Ax')\| \geq \|x - x'\| - 2\eps,
\eeq
which implies the instability of the reconstructor $\Psi$. Thus, any stable and accurate reconstruction procedure should be ``kernel aware". 
Unfortunately, most deep learning methods do not enforce kernel awareness during the learning procedure. Thus it is concluded in \cite{Got} that ``deep learning reconstruction procedure that over-performs in a certain sense must inevitably succumb to one or more of the main (instability) phenomena".
In the literature, designing deep neural networks to promote kernel awareness is an active research
area (see e.g., \cite{Lip1, Lip2}). However, these studies are  conducted for general deep learning problems. For severely ill-posed problems, e.g., EIT, they will inevitably lead to a loss of accuracy, in light of the philosophy in \cite{Got}. 

In this work, we investigate sufficient conditions on the training data in order to improve stability properties of the deep Calder\'on method \cite{Cen}, one postprocessing type deep learning-based techniques for EIT image reconstruction; see Section \ref{subsec-deep} for details about the method. To this end, we first rigorously establish the Lipschitz stability of the linearized inversion for a suitable admissible set of anisotropic conductivities, and then use the stability result to shed insights into the method. Note that the Lipschitz stability result is sufficient to guarantee the kernel awareness at least on the theoretical level \cite[p. 5]{Col}; see the comments after Theorem \ref{thm-main2d} for further discussions. Our result provides guideline for designing a suitable training dataset in order to ensure the stable recovery of the deep Calder\'{o}n method in the presence of anisotropies, provided that the neural network is well trained.

The rest of the paper is organized as follows. In Section \ref{sec-main}, we present the main stability result in the two-dimensional case and discuss the implication on the stability of the trained neural networks. Then in Section \ref{sec-pf2d}, we give the proof of the main result. In Section \ref{sec-pf3d}, we present an analogous result in the multi-dimensional case. In Section \ref{sec-num}, we present numerical experiments to complement the theoretical findings. Throughout, the notation $(\cdot,\cdot)$ and $|\cdot|$ denote the Euclidean inner product and norm, respectively.
 
%===========================%
\section{The main results}\label{sec-main} 
Consider the operator $\gamma \rightarrow \La_{\gamma}$ and we aim to use deep learning methods to find its (approximate) inverse. In this work, we address the local inversion problem near some fixed $\gamma_0$. Specifically, consider conductivities of the form
\beqq\label{eq-gamma}
\gamma = (1 + \delta) (I + A),
\eeqq
where $\delta \in C^1(\overline{\Omega})$ is a scalar function and $A = (A_{ij})_{i, j = 1}^2$ is a matrix-valued $C^1(\overline{\Omega})$ function. Throughout we assume that the anisotropy matrix $A$ is known. We investigate the recovery of the perturbation $\delta.$ Thus $\gamma$ can be viewed as perturbations of $\gamma_0 = I  + A$. This setting corresponds to the recovery of anisotropic conductivities in a conformal class, which has been studied in several works \cite{Lio, Mur, GaSi}. Note that the reconstruction of anisotropic EIT has also been investigated in the literature (see, e.g., \cite{HenkinSantacesaria:2010,HamiltonLassas:2014}).

%===========================%
\subsection{The deep Calder\'on method}\label{subsec-deep} 
We consider the deep Calder\'on method developed in \cite{Cen}. Originally, it was developed for the isotropic EIT problem; see \cite{SunZhongWang:2023,LiShinZhou:2025} for further its refinements. In this work, we apply the method to the anisotropic EIT problem and analyze its stability properties. Note that in the literature, there are several other deep learning methods based on a similar strategy, e.g., deep D-bar \cite{Ham} and the neural correction scheme  \cite{BCW}. 

The deep Calder\'on method consists of two steps. First, one applies Calder\'on method to obtain an initial reconstruction from $\La_\gamma.$ This step is essentially contained in Calder\'on's original paper \cite{Cal}, and was explored numerically in several works \cite{BikowskiMueller:2008,KnudsenMueller:2013,MullerMuller:2017,ShinMueller:2020,Mur}. Specifically, let $k$ and $\zeta$ be two vectors in $\mbr^2$ such that $|k| = |\zeta|$ and $(k, \zeta) = 0$. Consider the harmonic functions 
$u_1(x) = e^{\frac12({\rm i} k \cdot x +  \zeta \cdot x)}$ and $u_2(x) = e^{\frac12({\rm i} k\cdot x - \zeta \cdot x)}$,
where ${\rm i}$ is the imaginary unit.
Then   for $|k|< R$ and $\|\delta\|_{L^\infty(\Omega)}$ sufficiently small, there holds
\beq
\begin{gathered}
\int_{\p \Omega} u_1(x) (\La_\gamma - \La_{\gamma_0}) u_2(x)\d S 
 \simeq - \frac{|k|^2}{2} \int_{\Omega} \delta(x) e^{ \mathrm{i} k\cdot x} \d x.
\end{gathered}
\eeq  
In Section \ref{sec-pf2d}, we provide the details of the derivation. Thus, we can take 
\beqq\label{eq-F}
\widehat h(-k) = -\frac{2}{|k|^2}\int_{\p \Omega} u_1(x) (\La_\gamma - \La_{\gamma_0}) u_2(x)\d S 
\eeqq
as an approximation of $\widehat \delta(-k)$ in the Fourier domain. By the inverse Fourier transform, we get
\beqq\label{eq-apr}
\delta(x) \simeq \widetilde \delta(x) :=\frac{1}{(2\pi)^2}\int_{|k| < R} e^{\mathrm{i} k \cdot x} \widehat h(-k)\d k.
\eeqq 
Next, one post-processes the initial reconstruction $\widetilde \delta(x)$ using a convolutional neural network (CNN)  $f_\theta$ (more precisely, U-net \cite{Ron15}, which is widely used in medical imaging segmentation). Note that the second step of the deep Calder\'{o}n method can be regarded as a nonlinear map $f: \tilde \delta \rightarrow \delta$.

In the construction of the approximation $\widetilde \delta (x)$ in the first step, we point out that the following information has been ignored: (i) the high frequency Fourier modes of $\delta$; (ii) the nonlinearity of the map $\gamma\mapsto \La_\gamma$; (iii) the anisotropy of the conductivity $\gamma$. For (iii), we regard the anisotropy as a part of the nonlinear map $f$ and expect the anisotropy to be recovered in the second step, given the invertibility of $f$. One may also use the anisotropic Calder\'on method (see, e.g., \cite{Mur}) in the first step, but it still does not completely remove the anisotropy from the second step.

The role of the U-net in the second step is to learn an approximation $f_\theta$ of the nonlinear map $f$ from a given training dataset consisting of pairs of the Calder\'{o}n reconstruction and the corresponding ground truth. For CNNs, the universal approximation property is known (see e.g., \cite{Zho}), although it is often unclear whether the desired approximation property can be realized in practice. In this work, we are interested in finding conditions on the training dataset under which a well trained neural network $f_\theta$ has Lipschitz stability therefore being kernel-aware. It is important to point out that in the first step of the deep Calder\'on method, the high frequency Fourier information of $\delta$ is thrown away, so there is potential loss of information. However, if a Lipschitz stable $f_\theta$ exists (and can be numerically achieved), the map that takes $\widetilde \delta$ to $\chi_R\ast\delta$, where $\chi_R$ denotes the characteristic function of the set $\{k\in\mathbb{R}:|k|<R\}$, should be Lipschitz stable, and this further implies that the inversion from $\La_{\gamma}$ to $\chi_R\ast\delta$ is Lipshitz stable. Our result in Section \ref{subsec-stab} finds the condition on $\delta$ to make this happen. Thus, under the hypothesis that the universal approximation property of the CNN can be  achieved, the stability result provides a way to promote kernel awareness for the deep Calder\'on method by choosing proper training data, rather than changing the U-net architecture, e.g., by explicitly enforcing the Lipschitz constraint. 
 
%==========================%
\subsection{The stability result} \label{subsec-stab}
Let $\chi(\xi)$ be the characteristic function of the unit disc $B_1(0) = \{\xi\in \mbr^2: |\xi| < 1\}$ in $\mbr^2$. We define $\chi(D)$ to be a Fourier multiplier so that $\chi(D) f(x) =  (\chi \widehat f)^\vee(x)$ where $\wedge$ and $ \vee$ denote respectively the Fourier and inverse Fourier transforms on $\mbr^2$.  For $M >0, N>0$, we define
\beqq\label{eq-dn0}
\begin{split}
\mce_{M, N} =  \{&\delta\in C^2(\overline{\Omega}): \mathrm{supp}(\delta)\Subset \Omega, \|\delta\|_{L^\infty(\Omega)} < M\\ &\text{ and } \|\delta\|_{L^\infty(\Omega)} \leq  N\|\chi(D) \lap\delta\|_{L^2(\mbr^2)}\}.
\end{split}
\eeqq
Note that this is a subset (but not a subspace) of $C^2(\overline{\Omega})$. Then we have the following conditional stability result. 
\begin{theorem}\label{thm-main2d} 
Suppose $A\in C^1(\overline{\Omega})$ such that $\nabla\cdot(\gamma_0 \nabla)$ is elliptic, see \eqref{eq-ell}. Then for $M<1$ and  $N >0$ sufficiently small, we have for $\delta \in \mce_{M, N}$ that  
\beqq\label{eq-lipest0} 
\|\delta \|_{L^\infty(\Omega)} \leq  C \|\La_{\gamma} - \La_{\gamma_0}\|_*,
\eeqq
where $C>0$ is  independent of $\delta$, and $\|\bullet\|_*$ denotes the operator norm from $H^{1/2}(\p \Omega)$ is $H^{-1/2}(\p \Omega).$ 
\end{theorem}

In Section \ref{sec-pf3d}, we shall prove an analogous result for higher dimensions.  We have a few remarks on Theorem \ref{thm-main2d}. First, note that the inversion of the map $\gamma\mapsto \La_\gamma$ is inherently nonlinear. The estimate \eqref{eq-lipest0} implies that the inversion of the operator is Lipschitz stable, which implies that the inversion of $\La_{\gamma}$ is close to Lipschitz stable for small $\delta$. Note that Lipschitz stability is sufficient to guarantee the kernel awareness required for stable and accurate neural networks \cite{Col, Got, Sca}. Thus we focus on the local inversion near some fixed conductivity. In contrast, for linear inverse problems, this is not an issue; see,  e.g., \cite{WaZh}. 

Second, we remark that conditions similar to \eqref{eq-dn0} have been used to obtain uniqueness results for the Calder\'on problem in the geometric setting \cite{MSS, UhWa}.  We have assumed $\delta$ to be compactly supported in the domain $\Omega$. From the proof, one can see that it is possible to remove the compactness assumption by introducing suitable cut-offs in \eqref{eq-dn0}. Also, one can replace $\chi(\xi)$ by any compactly supported cut-off function in \eqref{eq-dn0} and obtain the same result with different values of $C$ and $N$. In particular, this applies to the spectral cut-off $\chi_R$ in the first step of the deep Calder\'on method. The stability constant $C$ in \eqref{eq-lipest0} can be found more explicitly in the proof. For example, the constant is small for $N>0$ small. See also Remark \ref{rmk:constant} for more discussions. %In general, one would expect the constant to grow exponentially for $M, N$ large.  
 
Third, we consider using deep neural networks to find an approximation of the inverse operator of $\gamma \rightarrow \La_\gamma$ on the set $\mce_{M, N}$. Suppose that an accurate and stable neural network can be trained on $\mce_{M, N}$ (which is highly nontrivial due to the comnplex loss landscape). Then the reconstruction on $\mce_{M, N}$ will not encounter any ill-posedness. The ill-posedness will only appear when the neural network is applied to data outside $\mce_{M, N}$. From this perspective, deep learning methods for EIT could be potentially superior to traditional methods. 

Finally, it is worth mentioning that there is a large body of literature on the Lipschitz stability for the isotropic EIT problem; see, e.g., \cite{AlessandriniVessella:2005,BerettadeHoop:2013,ADGS,AAS,AlbertiSantacesaria:2019,Harrach:2019,AlbertiSantacesaria:2022,AspriBBeretta:2022,BerettaFrancini:2021,GardeHyvonen:2025} and the references therein. Most of these existing works require the parameter to belong to a finite-dimensional space or manifold with the notable exception of \cite{GardeHyvonen:2025, GaHy}. The works \cite{GardeHyvonen:2025,GaHy} establish the Lipschitz stability for an infinite-dimensional space of highly smooth conductivities satisfying some elliptic PDE. Also, the works \cite{AAS, AlbertiSantacesaria:2019,Harrach:2019} require only a finite number of measurements. For the result in Theorem \ref{thm-main2d}, we first note that the set $\mce_{M, N}$ in \eqref{eq-dn0} is not finite dimensional. The mechanism of the Lipschitz type stability is closer to that in \cite{UhWa, MSS} and \cite{GardeHyvonen:2025, GaHy}, which holds on compact sets. This can be seen more clearly from the proof in Section \ref{sec-pf2d} and the fact that there are compact subsets of $\mce_{M, N}$, e.g.,  
$\{\delta\in C^2(\overline{\Omega}): \mathrm{supp}(\delta)\Subset\Omega, \|\delta\|_{L^\infty(\Omega)} < M \text{ and } \|\delta\|_{C^3(\Omega)} \leq  N\|\chi(D) \lap\delta\|_{L^2(\mbr^2)}\}$.
Second, we note that the stability constants in the literature are usually implicit, while for the constant $C$ in the estimate \eqref{eq-lipest0}, we know some dependency on the data set as explained above. Numerically, a Lipshitz stability with a large stability constant is not effective. In Section \ref{sec:trainingset} below, we can actually use the dependency to find training data with a small stability constant. For the anisotropic EIT problem, the Lipschitz stability is less studied, see \cite{GaSi, UhWa, MSS,FosGabSin25}.

%%%%%%%%%%%%%%%%%%%%%%%%%%
\subsection{Implications on training sets}\label{sec:trainingset}
Theorem \ref{thm-main2d} allows us to understand the training of neural networks on specific training sets. First, we consider conductivities $\gamma = 1 + \delta$ where $\delta$ is a Gaussian that belongs to $\mce_{M, N}$. This type of data is used, e.g., in \cite[Section 3.4.1]{FaYi}.  Let $a>0$, $b>0$ and 
$\delta(x) = a e^{-\frac{b |x|^2}{2}}$.
If the domain $\Omega$ contains the origin then $\|\delta\|_{L^\infty(\Omega)} = a.$ Note that Gaussians are not compactly supported. Nonetheless, by taking $\Omega$ large, one can treat Gaussians as sufficiently good approximations of compactly supported functions. It is well-known that  \cite[Theorem 7.6.1]{Ho1}
$
\widehat \delta(\xi) = \frac{2\pi a}{b} e^{-\frac{|\xi|^2}{2b}}$.
Direct computation gives 
\begin{align*}
\|\chi(D)\lap \delta\|_{L^2(\mathbb{R}^2)}^2 & =(2\pi)^{-2} \int_{|\xi|\leq 1}  \frac{4\pi^2 a^2}{b^2}|\xi|^2 e^{-\frac{|\xi|^2}{b}}\d\xi \\
&=  (2\pi)^{-2}\int_{0}^1  \frac{4\pi^2 a^2}{b^2} 2\pi e^{-\frac{r^2}{b}}r^3\d r =  2\pi a^2 \int_{0}^{\frac{1}{\sqrt b}} e^{-s^2}s^3\d s. 
\end{align*}
Thus, for $\delta \in \mce_{M, N}$, there holds
\beqq\label{eq-N1}
N \geq \frac{\|\delta\|_{L^\infty(\Omega)}}{\|\chi(D)\lap \delta\|_{L^2(\mathbb{R}^2)}} = \frac{1}{(2\pi \int_{0}^{\frac{1}{\sqrt b}} e^{-s^2}s^3\d s)^\ha}:=\underline{N}. 
\eeqq
We plot the graph of the lower bound $\underline{N}$ of $N$ as a function of $b$ in Fig.\ \ref{fig:N-b}(a). 
Since Theorem \ref{thm-main2d} requires $N$ to be small, we should take  $b$ small. So, good stability can be achieved for Gaussians concentrated in regions that are not too small. In other words, high resolutions might be difficult to achieve in a stable way. The numerical results in \cite[Section 3.4.1]{FaYi} seem to agree well with the theory. In Example \ref{exam1} of Section \ref{sec-num}, we also perform relevant numerical experiments to validate the observation. 

\begin{figure}[t]
\centering
\begin{tabular}{cc}
\includegraphics[scale=0.5]{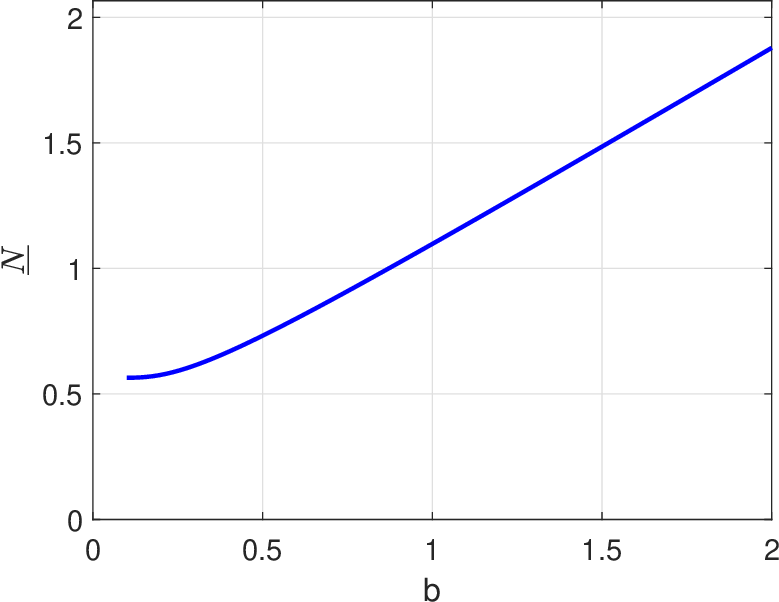}  &  \includegraphics[scale=0.5]{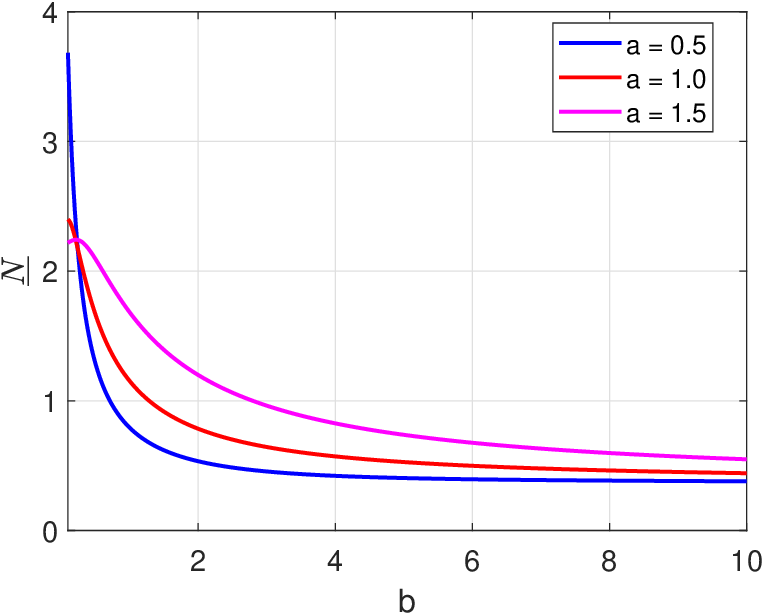}\\
 (a) Gaussian bump  & (b) characteristic function
\end{tabular}
\caption{The relation between the lower bound $\underline{N}$ of $N$, and $b$ for the perturbation by a Gaussian bump (given by \eqref{eq-N1}) and perturbation by a characteristic function (given by \eqref{eq-N2}) with various values of $a$.\label{fig:N-b} }
\end{figure}

Second, we consider the popular choice of piecewise constant functions as training data used in, e.g.,\ \cite{Cen}. Since Theorem \ref{thm-main2d} applies to $C^2$ functions, we take the convolution with Gaussians as approximations. Let $f_a(x)$ be the characteristic function of $B_a(0)$ in $\mbr^2$ for $a>0$. For $b>0$, let $g_b(x) = \frac{1}{2\pi b}e^{-\frac{|x|^2}{2b}}$ be a Gaussian. Note that 
$g_b(x) $ converges to the Dirac delta function on $\mbr^2$ supported at $x = 0$ in the sense of distributions as $b$ tends to zero.   
We set $\delta = g_b \ast f_a$. Then $\delta$ is smooth. Note that both $g_b$ and $f_a$ are non-negative. Due to symmetry, we have
\beq
\begin{gathered}
\|\delta\|_{L^\infty(\Omega)}= \delta(0) = \int_{|x|\leq a}  \frac{1}{2\pi b}e^{-\frac{|x|^2}{2b}} \d x
  =  2(1 - e^{-\frac{a^2}{2b}}). 
\end{gathered}
\eeq
Next, there holds
$$\widehat f_a(\xi) = 2\pi a \frac{J_1(\pi a |\xi|)}{|\xi|},$$
where $J_1$ is the Bessel function of the first kind.  We know $\widehat g_b(\xi) = e^{-\frac{b|\xi|^2}{2}}$. Then we derive 
\beq
\begin{split}
\|\chi(D)\lap \delta\|_{L^2(\mathbb{R}^2)}^2  
& = (2\pi)^{-2} \int_{|\xi|\leq 1}  |\xi|^2 e^{-b|\xi|^2} a^2 |J_1(\pi a |\xi|)|^2 |\xi|^{-2} \d\xi \\
 &  =2\pi a^2  \int_0^1   e^{-b r^2} |J_1(\pi a r)|^2  r\d r  
   = 2\pi \int_0^a   e^{-\frac{b s^2}{a^2}} |J_1(\pi  s)|^2  s\d s. 
\end{split}
\eeq
Therefore, we get 
\beqq\label{eq-N2}
N \geq \frac{\|\delta\|_{L^\infty(\Omega)}}{\|\chi(D)\lap \delta\|_{L^2(\mathbb{R}^2)}} = \frac{ 2(1 - e^{-\frac{a^2}{2b}})}{(2 \pi \int_0^a   e^{-\frac{b s^2}{a^2}} |J_1(\pi  s)|^2  s\d s)^\ha }:=\underline{N}. 
\eeqq

The estimate \eqref{eq-N2} implies that when $\delta$ belongs to $\mce_{M, N}$ in Theorem \ref{thm-main2d} for $N$ relatively small. 
For fixed $a$, we plot the graph of the lower bound $\underline{N}$ of $N$ as a function of $b$ in Fig.\ \ref{fig:N-b}(b). We observe that if $a$ is small, one can choose $b$ relatively small in order to keep $N$ small. However, if $a$ is large, then one needs to take $b$ large in order to make $N$ small. This suggests that a stable network can be achieved better for small objects. It is somewhat interesting that even for traditional reconstruction methods based on optimization, recent numerical study in \cite{APP} also indicates that the reconstruction algorithm works better for small objects. We perform related numerical experiments in Example \ref{exam2} of Section \ref{sec-num}. 

% 
 %%%%%%%%%%%%%%%%%%
\section{Proof of Theorem \ref{thm-main2d}}\label{sec-pf2d}
Our proof is based on Calder\'on's original approach in \cite{Cal}. Let $k, \zeta$ be vectors in $\mbr^2$  such that $|k| = |\zeta|$ and $(k, \zeta) = 0$. Consider the harmonic functions
$u_1(x) = e^{({\rm i} k \cdot x +  \zeta \cdot x)/2}$ and $u_2(x) = e^{({\rm i} k\cdot x  -\zeta \cdot x)/2}$.
We construct  solutions of problem \eqref{eq-eit} which are perturbations of $u_i.$ Below we use $C$ as a generic constant that can change line by line.  
\begin{lemma}\label{lm-per1}
Suppose $A\in C^1(\overline{\Omega})$ and that $\nabla \cdot((I + A) \nabla)$ is strictly elliptic in the sense that there is a constant $c>0$ such that 
\beqq\label{eq-ell}
\sum_{i, j = 1}^2  (I_{ij} + A_{ij}(x)) \xi_i\xi_j\geq c|\xi|^2
\eeqq
for all $x\in \Omega$ and $\xi\in \mbr^2$. Then for $i = 1, 2$, there are unique solutions $\tilde u_i\in C^2(\overline\Omega)$ of 
\beqq\label{eq-tu}
\left\{\begin{aligned}
\nabla \cdot ((I + A) \nabla \tilde u_i) &= 0,\quad \text{in } \Omega, \\
\tilde u_i &= u_i,\quad \text{on } \p \Omega.
\end{aligned}\right.
\eeqq 
Moreover, we can write $\tilde u_i = u_i + v_i$ such that for some $C$ depending on $A$ and $\Omega$,
\beqq\label{eq-vest-1}
\|v_i\|_{H^1(\Omega)} \leq C \|A\|_{L^\infty({\Omega})}  |k| e^{C |k|}.  
\eeqq
\end{lemma}
\begin{proof}
The uniqueness and existence of $\tilde u_i\in C^2(\overline\Omega)$ follow from the standard elliptic PDE theory (see, e.g., \cite[Theorem 6.14]{GiTr}). For the estimate \eqref{eq-vest-1}, note that $u_i$ satisfies the equation $\nabla\cdot \nabla u_i = 0$ on $\mbr^2$ by the choice of $k, \zeta$. Thus, using \eqref{eq-tu}, $v_i$ are solutions of 
\beqq\label{eq-v-1}
\left\{\begin{aligned}
\nabla \cdot ((I+A)\nabla v_i) &= - \nabla\cdot ( A \nabla u_i),\quad  \text{in } \Omega, \\
v_i &= 0,\quad \text{on } \p \Omega. 
\end{aligned}\right.
\eeqq 
By the standard elliptic estimate (see e.g. \cite[Theorem 8.12]{GiTr}), we deduce from \eqref{eq-v-1} that 
\beq
\|v_i\|_{H^1(\Omega)}\leq C \|\nabla\cdot ( A \nabla u_i)\|_{H^{-1}(\Omega)} \leq C \|A\|_{L^\infty(\Omega)} \|u_i\|_{H^1(\Omega)} \leq C \|A\|_{L^\infty(\Omega)} |k|e^{C |k|}. 
\eeq
This completes the proof of the lemma. 
\end{proof}
\begin{remark}
By repeating the argument of Lemma \ref{lm-per1}, we can also derive 
 \begin{equation*} 
\|v_i\|_{H^2(\Omega)} \leq C \|A\|_{C^1(\overline{\Omega})}  |k|^2 e^{C |k|}.  
\end{equation*}
\end{remark}

\begin{lemma}\label{lm-per2}
Suppose that $\gamma$ is given by \eqref{eq-gamma} where $A\in C^1(\overline{\Omega})$ satisfies \eqref{eq-ell} and $\delta \in C^1(\overline{\Omega})$  is compactly supported in $\Omega$ such that $\|\delta\|_{L^\infty(\Omega)}< M< 1$. Then there are unique solutions $w_i\in C^2(\overline\Omega)$, $i = 1, 2$, of  
\beqq\label{eq-w}
\left\{\begin{aligned}
\nabla \cdot (\gamma \nabla w_i) &= 0,\quad \text{in } \Omega, \\
w_i &=  u_i,\quad \text{on } \p \Omega
\end{aligned}\right.
\eeqq  
of the form  $w_i = \tilde u_i + \tilde v_i$ such that for some $C$ depending on $\Omega$, $A$ and $M$, 
\beqq\label{eq-vest}
\|\tilde v_i\|_{H^1(\Omega)} \leq C \|\delta\|_{L^\infty(\Omega)}  |k| e^{C |k|}.
\eeqq 
\end{lemma}
\bpf
The existence and uniqueness of $w_i\in C^2(\overline\Omega)$ follow from \cite[Theorem 6.14]{GiTr}. By \eqref{eq-w} and \eqref{eq-tu}, $\tilde v_i$ are solutions of 
\beqq\label{eq-v}
\left\{\begin{aligned}
\nabla \cdot (\gamma \nabla \tilde v_i) &= -\nabla\cdot (\delta \gamma_0 \nabla \tilde u_i),\quad \text{in } \Omega, \\
\tilde v_i &= 0,\quad \text{on } \p \Omega. 
\end{aligned}\right.
\eeqq 
%We write the equation as 
%\beq
%\nabla \cdot (\gamma_0 \nabla \tilde v_i) = -\nabla\cdot (\delta \gamma_0 \nabla \tilde u_i) - \nabla \cdot (\delta \gamma_0 \nabla \tilde v_i)
%\eeq
%For $\|\delta\|_{C^1}< M < 1$, we know that $\nabla\cdot (\gamma \cdot\nabla)$ is elliptic in view of \eqref{eq-ell}. Thus, by the standard elliptic estimate, we deduce  that 
By multiplying the equation by $\tilde v_i$, integrating over the domain $\Omega$ and integrating by parts, we get 
\beq
\| \nabla \tilde v_i\|^2_{L^2(\Omega)}\leq C \|\delta\|_{L^\infty(\Omega)} \|\nabla \tilde u_i\|_{L^2(\Omega)}\|\nabla \tilde v_i\|_{L^2(\Omega)}, 
\eeq
where we have used the condition $\|\delta\|_{L^\infty(\Omega)}< M< 1$, and $C$ depends on $M.$ Then we obtain
\beq
\|\nabla \tilde v_i\|_{L^2(\Omega)}\leq  C \|\delta\|_{L^\infty(\Omega)} \|\nabla \tilde u_i\|_{L^2(\Omega)}. 
\eeq
Using Lemma \ref{lm-per1} and the Poincar\'{e} inequality, we complete the proof of the lemma. 
\epf

Now we can state the proof of Theorem \ref{thm-main2d}.
\begin{proof}
Using Lemmas \ref{lm-per1} and \ref{lm-per2}, we can write $w_i = \tilde u_i + \tilde v_i.$ Using Green's identity for problems \eqref{eq-w} and \eqref{eq-tu}, we get 
\begin{align*}
\int_{\p \Omega} w_1(x) \La_\gamma w_2(x)\d S &= \int_\Omega \gamma(x) \nabla w_1(x) \cdot \nabla w_2(x) \d x,  \\
\int_{\p \Omega} \tilde u_1(x) \La_{\gamma_0} \tilde u_2(x)\d S &= \int_\Omega \gamma_0(x) \nabla \tilde u_1(x) \cdot \nabla \tilde u_2(x) \d x. 
\end{align*}
Now by using the relations $\gamma=(1+\delta)\gamma_0$ and $\gamma_0=I+A$, and substituting the choices of $\tilde u_1$ and $\tilde u_2$, we have
\begin{comment}
\begin{align*}
& \int_{\p \Omega} \tilde u_1(x) (\La_\gamma - \La_{\gamma_0}) \tilde u_2(x)\d S \nonumber \\
 = & - \int_{\p\Omega} \tilde v_1(x) \La_\gamma w_2(x) \d S - \int_{\p\Omega} w_1(x) \La_\gamma \tilde v_2(x) \d S + \int_{\p\Omega} \tilde v_1(x) \La_\gamma \tilde v_2(x) \d S\nonumber\\
& +  \int_\Omega \delta(x)\gamma_0(x) \nabla \tilde u_1(x) \cdot \nabla \tilde u_2(x) \d x   +  \int_\Omega \gamma(x)  \nabla \tilde v_1(x) \cdot \nabla \tilde u_2(x)  \d x \nonumber \\
& + \int_\Omega \gamma(x) \nabla \tilde u_1(x) \cdot \nabla \tilde v_2(x) \d x  + \int_\Omega \gamma(x) \nabla \tilde v_1(x) \cdot \nabla \tilde v_2(x) \d x
\end{align*}
Now the relation  implies
$$\int_\Omega \delta(x)\gamma_0(x)\nabla \tilde u_1\cdot\nabla \tilde u_2\d x = \int_\Omega \delta(x)\nabla \tilde u_1\cdot\nabla \tilde u_2\d x + \int_\Omega \delta(x)A(x)\nabla \tilde u_1\cdot\nabla \tilde u_2\d x.$$
Consequently,
\end{comment}
\begin{align}
 \int_{\p \Omega} \tilde u_1(x) (\La_\gamma - \La_{\gamma_0}) \tilde u_2(x)\d S= \int_\Omega \delta(x) \nabla \tilde u_1(x) \cdot \nabla \tilde u_2(x) \d x  + E_1, \label{eq-id0} 
 \end{align}
 with the error term $E_1$ given by
 \begin{align*}
 E_1=& - \int_{\p\Omega} \tilde v_1(x) \La_\gamma w_2(x) \d S - \int_{\p\Omega} \tilde{u}_1(x) \La_\gamma \tilde v_2(x) \d S 
 +  \int_\Omega \delta(x)A(x) \nabla \tilde u_1(x) \cdot \nabla \tilde u_2(x) \d x\\
 &+  \int_\Omega \gamma(x)  \nabla \tilde v_1(x) \cdot \nabla \tilde u_2(x)  \d x  + \int_\Omega \gamma(x) \nabla \tilde u_1(x) \cdot \nabla \tilde v_2(x) \d x  + \int_\Omega \gamma(x) \nabla \tilde v_1(x) \cdot \nabla \tilde v_2(x) \d x.
\end{align*}
Meanwhile, the choices of $u_1$ and $u_2$ imply the identity $\int_\Omega \delta(x) \nabla u_1(x) \cdot \nabla  u_2(x) \d x = -\frac{|k|^2}{2}\widehat{\delta}(-k)$. Then with the relation $\tilde u_i=u_i+v_i$, we obtain 
\begin{align*}
 \int_\Omega \delta(x) \nabla \tilde u_1(x) \cdot \nabla \tilde u_2(x) \d x  
 = &  -\frac{|k|^2}{2} \widehat{\delta}(-k) + E_2, 
\end{align*}
with the error term $E_2$ given by
\begin{align*}
    E_2=  \int_\Omega \delta(x) \nabla u_1(x) \cdot \nabla  v_2(x) \d x
 + \int_\Omega \delta(x) \nabla   v_1(x) \cdot \nabla  u_2(x) \d x +  \int_\Omega \delta(x) \nabla   v_1(x) \cdot \nabla   v_2(x) \d x.
 \end{align*}
Next we bound the two error terms $E_1$ and $E_2$ separately. Using Lemmas \ref{lm-per1} and \ref{lm-per2} and the trace theorem, we can estimate the error term $E_1$ by  
\begin{align*}
|E_1| \leq &\|\tilde v_1\|_{H^{\frac12}(\p\Omega)} \|\La_\gamma w_2\|_{H^{-\frac12}(\p\Omega)} + \|\tilde u_1\|_{H^{\frac12}(\p\Omega)} \|\La_\gamma \tilde v_2\|_{H^{-\frac12}(\p\Omega)} + C\|\delta\|_{L^\infty(\Omega)}\|\tilde u_1\|_{H^1(\Omega)}\|\tilde u_2\|_{H^1(\Omega)}\\
 & + C \|\tilde v_1\|_{H^1(\Omega)} \|\tilde u_2\|_{H^1(\Omega)} + C\|\tilde u_1\|_{H^1(\Omega)} \|\tilde v_2\|_{H^1(\Omega)} + C\|\tilde v_1\|_{H^1(\Omega)} \|\tilde v_2\|_{H^1(\Omega)}\\
 \leq &  C\|\delta\|_{L^\infty(\Omega)}\|\tilde u_1\|_{H^1(\Omega)}\|\tilde u_2\|_{H^1(\Omega)} +  C( \|\tilde v_1\|_{H^1(\Omega)} \|w_2\|_{H^1(\Omega)} + \|\tilde u_1\|_{H^1(\Omega)} \|\tilde v_2\|_{H^1(\Omega)} \\
& + \|\tilde v_1\|_{H^1(\Omega)} \|\tilde u_2\|_{H^1(\Omega)} + \|\tilde u_1\|_{H^1(\Omega)} \|\tilde v_2\|_{H^1(\Omega)} + \|\tilde v_1\|_{H^1(\Omega)} \|\tilde v_2\|_{H^1(\Omega)})\\
\leq &C\|\delta\|_{L^\infty(\Omega)} |k| e^{C |k|}. 
\end{align*}
Similarly, by Lemma \ref{lm-per1}, we can bound the error term $E_2$ by  
\beq
\begin{split}
|E_2| &\leq C \|\delta\|_{L^\infty(\Omega)} (\|u_1\|_{H^1(\Omega)} \|v_2\|_{H^1(\Omega)} + \|v_1\|_{H^1(\Omega)} \|u_2\|_{H^1(\Omega)} + \|v_1\|_{H^1(\Omega)} \|v_2\|_{H^1(\Omega)})\\
&\leq C\|\delta\|_{L^\infty(\Omega)} |k| e^{C |k|}. 
\end{split}
\eeq
Now we can deduce from the identity \eqref{eq-id0} that
\begin{align}
\frac{|k|^2}{2}|\hat\delta(-k)| \leq &\|\tilde u_1\|_{H^\ha(\p\Omega)} \|\tilde u_2\|_{H^\ha(\p\Omega)} \|\La_{\gamma} - \La_{\gamma_0}\|_* +  C\|\delta\|_{L^\infty(\Omega)}    |k| e^{C |k|}\nonumber\\
\leq& C|k|e^{C |k|} \|\La_\gamma - \La_{\gamma_0}\|_* + C  \|\delta\|_{L^\infty(\Omega)}   |k| e^{C |k|}.\label{eq-temp-1}
\end{align}
Then by multiplying both sides by $|\chi(k)|$ and noting that the factors involving $k$ are uniformly bounded on the set $B_1(0)$, we deduce
\begin{align*} 
\int_{\mbr^2}|\chi(k)| |k|^4|\hat\delta(k)|^2 \d k \leq  C_0 \|\La_\gamma - \La_{\gamma_0}\|_*^2 +   C_1 \|\delta\|_{L^\infty(\Omega)}^2.
\end{align*}
From the assumption $\|\delta\|_{L^\infty(\Omega)} \leq N \|\chi(D)\lap \delta \|_{L^2(\mbr^2)}$ and  Plancherel's theorem, we get 
\beq
\|\chi(D) \lap \delta \|_{L^2(\mbr^2)}^2 \leq  C_0  \|\La_\gamma - \La_{\gamma_0}\|_*^2 + C_1 N^2 \|\chi(D)\lap \delta\|_{L^2(\mbr^2)}^2, 
\eeq 
for some $C_0$, $C_1$ depending on $\chi$. Let $N >0$ be small enough so that $C_1 N^2  < 1/2$. Then we get 
\beq
\|\chi(D)\lap \delta \|_{L^2(\mbr^2)}^2 \leq 2C_0 \|\La_\gamma - \La_{\gamma_0}\|_*^2,
\eeq 
which implies 
$\|\delta \|_{L^\infty(\Omega)}^2 \leq  2 C_0 N^2   \|\La_\gamma - \La_{\gamma_0}\|_*^2.$ 
This completes the proof of the theorem.
\end{proof}

\begin{remark} \label{rmk:constant}
The proof indicates that the constants $C_0$ and $C_1$ depend on the support of the characteristic function $\chi$. More precisely, for $R>0$, let $\chi_R$ be the characteristic function of $B_R(0)$. Then it follows from the estimate \eqref{eq-temp-1} that the constants $C_0$ and $ C_1$ grow exponentially for $R$ large. Thus $N< (\frac{1}{2C_1})^\ha$ indicates that $N$ should be exponentially small for large $R$. This suggests that in the first step of the deep Calder\'on method in Section \ref{subsec-deep}, it is preferable to take the truncation radius $R$ relatively small, which is indeed commonly adopted in practice. 
\end{remark}

\begin{remark}
Conditional stability estimates are useful for analyzing the regularizing property of numerical schemes for inverse problems \cite{CenJinQuanZhou:2026}.  The regularizing property of the closely related D-bar method has been investigated in \cite{KnudsenLassas:2009,KnudsenRasmussen:2022}. It is of interest to use Theorem \ref{thm-main2d} to establish the regularizing property of the Calder\'on method and D-bar method.
\end{remark}

 %%%%%%%%%%%%%%%%%%
\section{The stability result in higher dimensions}\label{sec-pf3d}
We prove an analogous result of Theorem \ref{thm-main2d} for dimensions $n\geq 3$. Let $\Omega$ be an open bounded domain of $\mbr^n$ ($n\geq3$) with a smooth boundary $\partial\Omega$. Consider conductivities of the form $\gamma = (1 + \delta)(I + A)$ in \eqref{eq-gamma} and let $\gamma_0 = I + A.$ Also, we assume that $A\in C^1(\overline\Omega)$, and that $\delta\in C^1(\overline\Omega)$ is compactly supported in $\Omega$. 

For conductivities in a fixed conformal class, it is convenient to transform the conductivity equation to a Schr\"odinger equation. Thus, we define a smooth Riemannian metric $g$ on the domain $\Omega$ by 
\beq
g = (\det \gamma)^{\frac{1}{n-2}} \gamma^{-1},
\eeq
with $\gamma^{-1}$ being the inverse of $\gamma$ as a matrix. By the regularity assumptions of $A$ and $\delta$, $g\in C^1(\overline\Omega)$. Let $\lap_g$ be the associated Laplace-Beltrami operator, which in local coordinates is given by
\beq
\lap_g  = \frac{1}{\sqrt{\det g}} \sum_{i, j = 1}^n \frac{\p }{\p x_i}\Big(\sqrt{\det g} g^{ij} \frac{\p }{\p x_j}\Big).
\eeq
Then consider the following Dirichlet boundary value problem
\beqq\label{eq-rie-0}
\left\{\begin{aligned}
    \lap_g u &= 0,\quad \text{in } \Omega, \\\quad u&= f,\quad \mbox{on }\partial\Omega.
\end{aligned}\right. 
\eeqq
%Suppose that $0$ is not an eigenvalue. Then 
For $f\in H^{\frac12}(\p\Omega)$, there is a unique solution $u\in H^1(\Omega)$ on $\Omega$ and the Dirichlet-to-Neumann map of problem \eqref{eq-rie-0} is defined by 
\beq
\La_gf = \sum_{i, j = 1}^n \nu_i g^{ij} \sqrt{\det g}  \frac{\p u}{\p x_j}  |_{\p \Omega}
\eeq
which is bounded from $H^{\frac12}(\p\Omega)$ to $H^{-\frac12}(\p\Omega)$. It is proved in \cite{LeUh} that $\La_g = \La_\gamma$. We can write $\gamma = \phi \gamma_0$ and $ \phi  = 1 + \delta >0$. Let  $g_0$ be defined by 
\beqq\label{eq-g0}
g_{0} = (\det \gamma_0)^{\frac{1}{n-2}} \gamma_0^{-1}.
\eeqq
Note that 
\beq
g = \phi^{\frac{n}{n-2}} (\det \gamma_0)^{\frac{1}{n-2}}  \phi^{-1} \gamma_0^{-1} = \phi^{\frac{2}{n-2}} g_{0}.
\eeq
Thus, $g$ is conformal to $g_0.$ Now consider the following Dirichlet problem for the Schr\"odinger operator on $(\Omega, g)$ 
\beqq\label{eq-rie}
\left\{\begin{aligned}
(-\lap_g  + q) u &= 0,\quad \text{in } \Omega,\\
\quad u& = f,\quad \mbox{on }\partial\Omega.
\end{aligned}\right.
\eeqq 
The associated Dirichlet-to-Neumann map is well-defined and is denoted by $\La_{g, q}$. For a general metric $g$ and potential $q$, $0$ can be an eigenvalue for the elliptic operator in \eqref{eq-rie} (equipped with the zero Dirichlet boundary condition), which however is not the case for $g$, $q$ derived from the conductivity problem. According to  \cite[Proposition 8.2]{DKSU}, we have 
\beqq\label{eq-conf}
\La_{g, 0} = \La_{g_0, q}\quad \mbox{with }  q = -\phi^{-1 - \frac{n-2}{4}} \lap_g (\phi^{\frac{n-2}{4}}). 
\eeqq
Therefore, we have proved that $\La_\gamma = \La_{g_0, q}$ with $g_0$ in \eqref{eq-g0} and $q$ in \eqref{eq-conf}. 

We study the inversion of the map $\gamma \rightarrow \La_\gamma$ in two steps. First, we recover $q$ from the Dirichlet-to-Neumann map $\La_{g_0, q}$. Second, we recover the perturbation $\delta$ from the potential $q$. 
For the first step, we can follow the same approach as in the 2D case by using perturbations of harmonic functions. However, for dimensions $n \geq 3$, it is more convenient to use complex geometric optics (CGO) solutions. % which will allow removing the smallness condition on $\delta$. 
Consider
\beqq\label{eq-schro1}
-\lap u + q u = 0\quad \text{in } \mbr^n.
\eeqq
We look for CGO solutions of the form 
\beqq\label{eq-cgo}
u(x) = e^{\xi \cdot x} (1 + \psi(x, \xi)),
\eeqq
%Consider 
%\beqq\label{eq-schro0}
%\nabla \cdot (\gamma(x)\nabla u(x)) = 0 \text{ in } \mbr^n
%\eeqq
where $\xi\in \mbc^n$ is a complex vector with $\xi\cdot \xi = 0$. When $q=  0$, $u(x) = e^{\xi \cdot x}$ is the harmonic function we have used in Section \ref{sec-pf2d}. 
For $n\geq 3$, the CGO solution was constructed in \cite{SyUh}. We  use the weighted Sobolev spaces in \cite{SyUh}. For $\eta \in \mathbb{R}$ and $m\in \mathbb{N}$, let
\beq
\|f\|_{L^2_\eta(\mathbb{R}^n)}^2 = \int_{\mbr^n} (1 + |x|^2)^\eta |f(x)|^2 \d x\quad \text{and}\quad
\|f\|_{H^m_\eta(\mathbb{R}^n)}^2 = \sum_{|\alpha|\leq m} \|D^\alpha f\|_{L^2_\eta(\mathbb{R}^n)}^2.
\eeq
Then we have the following estimate.
\begin{theorem}[{\cite[Corollary 2.5]{SyUh}}]\label{thm-cgo1}
Let $\xi\in \mbc$, $\xi \cdot \xi = 0$, $|\xi|> \beta >0$. Let $s>\frac n2$ and $\eta \in(-1, 0)$. Then there is $\eps(\Omega, \beta) >0$ such that if 
$\|q\|_{H^s(\mathbb{R}^n)} < \eps |\xi|$, 
there is a unique solution $u$ of the form \eqref{eq-cgo} and a constant $C$ depending on $\eps$ such that 
\beq
\|\psi\|_{H^s_{\eta}(\mathbb{R}^n)} \leq C|\xi|^{-1}\|q\|_{H^s(\mathbb{R}^n)}.
\eeq
\end{theorem}

Note that the Schr\"odinger equation $-\lap_{g_0} u + qu = 0$ involves a smooth metric $g_0$. It seems not known how to construct CGO solutions for a smooth metric for dimension $n\geq3$. Below, we treat $g_0$ as a perturbation of the Euclidean metric and find a perturbation of the CGO solution \eqref{eq-cgo}.
\begin{lemma}\label{lm-cgo2}
Suppose $q\in H^s(\mbr^n)$, with $s>\max(\frac{n}{2}, 2), n\geq 3$, is compactly supported in $\Omega$. Let $g_0$ be the $C^1(\overline\Omega)$ Riemannian metric defined in \eqref{eq-g0}. Then there exists $w\in H^2(\Omega)$  satisfying 
\beq
-\lap_{g_0} w + qw = 0 \quad \text{in } \Omega
\eeq
of the form $w = u + v$ where $u$ is given by \eqref{eq-cgo} with $\psi$ in Theorem \ref{thm-cgo1} and $v$ satisfies  
\beq
\|v\|_{H^2(\Omega)} \leq  C|\xi|^2 e^{C |\xi|},
\eeq
for $|\xi|$ sufficiently large so that $\|q\|_{H^s(\mbr^n)} < \eps |\xi|$, where the constant $C$ depends on $A$, $\Omega$ and $\eps$ in Theorem \ref{thm-cgo1}.
\end{lemma}
\bpf
Using the expressions of $g_0$ in \eqref{eq-g0}, we have 
\beq
-\lap_{g_0} w  = -\lap w + P(x, \p) w\quad \mbox{in }\Omega, 
\eeq
where $P(x, \p)$ is a second-order differential operator of the form 
\beq
P(x, \p) = \sum_{i, j = 1}^n a_{ij}\p_i \p_j  + \sum_{k = 1}^n b_k \p_k, 
\eeq
where $a_\bullet, b_\bullet$ %\sout{are smooth functions on $\Omega$ and their $C^2(\overline{\Omega})$ norms are bounded (depending on the $C^2(\overline{\Omega})$ norm of $A$)} 
are continuous on $\overline\Omega$ and their $C^0$ norms are bounded (depending on the $C^1(\overline{\Omega})$ norm of $A$).  Using the governing equations for $w$ and $u$,  $v$ satisfies  
\beq
\left\{\begin{aligned}
-\lap_{g_0} v + q v &= P(x, \p) u,\quad \text{in } \Omega,\\
v &= 0, \quad\text{on } \p\Omega.
\end{aligned}\right.
\eeq
Using the representation \eqref{eq-cgo}  and the regularity estimate of $\psi$ in Theorem \ref{thm-cgo1} (with the condition $s>\max(\frac{n}{2},2)$), the function $u$ belongs to $H^2(\Omega)$. Then we can estimate
\beq
\|P(x, \p)u\|_{L^2(\Omega)}\leq C|\xi|^2 e^{C |\xi|}.  
\eeq
Thus, by the standard elliptic regularity theory, we get 
\beq
\|v\|_{H^2(\Omega)} \leq C \|P(x, \p)u\|_{L^2(\Omega)}\leq C|\xi|^2 e^{C |\xi|}. 
\eeq
This completes the proof of the lemma.
\epf
 
Now consider the recovery of the potential $q$ from $\La_{g_0, q}$. For $N>0$, we define the following admissible set of the potential $q$:
\beqq\label{eq-dn}
\begin{gathered}
\mcd_{N} = \{q \in L^\infty(\Omega): \text{$q$ is compactly supported in $\Omega$, }\|q\|_{L^\infty(\Omega)} \leq N \|\chi(D) q\|_{L^2(\mbr^n)}\}. 
\end{gathered}
\eeqq
\begin{proposition}\label{prop-main1} 
Let $n\geq 3$, $s>\max(\frac{n}{2}, 2)$, and $M>0$. Then for $N>0$ sufficiently small,  for $\|q\|_{H^s(\Omega)} \leq M$ and $q\in \mcd_N$,  
we have 
\beqq\label{eq-lipest}  
\|q\|_{L^\infty(\Omega)} \leq  N^{-1} \| \La_{g_0, q} -   \La_{g_0, 0}\|_*. 
\eeqq
\end{proposition}
\bpf 
Without loss of generality, let $q_1 = 0$, $q_2 =q$ and for $i=1,2$ consider 
\beqq\label{eq-schro12}
\left\{\begin{aligned}
-\lap_{g_0} w_i + q_i w_i &= 0,\quad \text{in } \Omega,\\
w_i &= f,\quad \text{on } \p \Omega. 
\end{aligned}\right.
\eeqq 
Then we have
\beqq\label{eq-int}
\int_{\Omega} (q_1 -q_2)w_1 w_2 \d x = \int_{\p \Omega} w_1 (\La_{g_0, q_1} -  \La_{g_0, q_2}) w_2 \d S. 
\eeqq
Let $w_i=u_i+v_i$, $i = 1, 2$ be the perturbed CGO solution given in Lemma \ref{lm-cgo2}. We set 
$
\xi_1 = \mathrm{i}( \frac{k}{2} +    \tau \eta)  + \xi$ and $\xi_2 = \mathrm{i} (\frac{k}{2} -   \tau \eta) - \xi$,
where $k\in \mbr^n$, $|k| \geq \beta$, $\tau >1$ and $\xi, \eta \in \mbr^n$ are chosen such that 
\beq
k\cdot \eta = k \cdot \xi = \eta \cdot \xi = 0, \quad |\eta| = 1,\quad  |\xi|^2 = \tau^2 + \frac{|k|^2}{4}. 
\eeq
To apply Theorem \ref{thm-cgo1}, it suffices to choose $\xi_i$ such that  $M\leq \eps |\xi_i|$.  From the  construction of $\xi_i$, it suffices to choose $\eps$, $\tau$ and $ k$ such that $M \leq \eps \tau$.  We get from the identity \eqref{eq-int} that 
\beqq\label{eq-id1}
\begin{aligned}
\int_{\Omega} (q_1 -q_2)e^{\mathrm{i} k\cdot x}\d x =& \int_{\p \Omega} w_1 (\La_{g_0, q_1} -  \La_{g_0, q_2}) w_2 \d S \\
&- \int_{\Omega} (q_1 -q_2) (u_1 v_2 + u_2 v_1 + v_1 v_2) \d x\\
&- \int_{\Omega} (q_1 -q_2) e^{\mathrm{i} k\cdot x}(\psi_1+  \psi_2  + \psi_1 \psi_2) \d x:={\rm I} + {\rm II} + {\rm III}.
\end{aligned}
\eeqq
Note that $q_1 -q_2 = -q$.  The left hand side gives  $-\hat q(-k)$. Next we bound the three terms ${\rm I}$, $\rm II$ and ${\rm III}$ separately. To estimate the term ${\rm I}$, using Theorem \ref{thm-cgo1}, Lemma \ref{lm-cgo2} and the trace theorem, we get (for $i= 1,2$) 
\beq
\begin{split}
\|w_i\|_{H^{1/2}(\p\Omega)} \leq C\|w_i\|_{H^1(\Omega)} \leq C(\|u_i\|_{H^1(\Omega)} + \|v_i\|_{H^1(\Omega)}) \leq C|\xi_i|^2 e^{C|\xi_i|}.
\end{split}
\eeq 
For $\tau$ large (hence $|\xi_i|$ large), we can find $C_1>0$ (depending on the constant $C$ in Lemma \ref{lm-cgo2}) so that $\|w_i\|_{H^{\frac12}(\p\Omega)} \leq C_1 e^{C_1|\xi_i|}.$ Then we estimate 
\beqq\label{eq-b}
\begin{aligned}
|{\rm I}|&\leq  \|w_1\|_{H^\ha(\partial\Omega)} \|w_2\|_{H^\ha(\partial\Omega)} \|\La_{g_0, q} - \La_{g_0, 0}\|_* \\
&\leq C_1^2 e^{2 C_1(|k| + \tau)} \| \La_{g_0, q} - \La_{g_0, 0}\|_*, 
\end{aligned}
\eeqq
By a similar argument, we can obtain 
\begin{align}
|{\rm II} |&\leq  \|q\|_{L^\infty(\Omega)}(\|u_1\|_{L^2(\Omega)} \|v_2\|_{L^2(\Omega)} + \|u_2\|_{L^2(\Omega)} \|v_1\|_{L^2(\Omega)} + \|v_1\|_{L^2(\Omega)} \|v_2\|_{L^2(\Omega)})\nonumber\\
&\leq C_2\|q\|_{L^\infty(\Omega)}   e^{C_2 |\xi|}, \label{eq-b1}
\end{align}
where the constant $C_2$ also depends on the constant $C$ in Lemma \ref{lm-cgo2}. Last, we have 
\begin{align}
|{\rm III}|&\leq  C\|q\|_{L^\infty(\Omega)}(\|\psi_1\|_{L^2(\Omega)}+  \|\psi_2\|_{L^2(\Omega)}  + \|\psi_1\|_{L^2(\Omega)} \|\psi_2\|_{L^2(\Omega)})\nonumber\\ 
&\leq C\|q\|_{L^\infty(\Omega)}(2|\xi|^{-1}\|q\|_{H^s(\Omega)}+|\xi|^{-2}\|q\|_{H^s(\Omega)}^2)
\leq C_3\|q\|_{L^\infty(\Omega)}  |\xi|^{-1}, \label{eq-b2} 
\end{align}
where $C_3>0$ depends on $\Omega$, $M$ and $\beta$. Now we can deduce from the identity \eqref{eq-id1} that 
\beq 
\begin{gathered}
|\widehat q(k)| 
\leq C e^{C(|k|+ \tau)} \|\La_{g_0, q} - \La_{g_0, 0}\|_* + C  \|q\|_{L^\infty(\Omega)}    e^{C (|k| + \tau)}
\end{gathered}
\eeq 
for some $C>0$. Thus, 
\beq 
\begin{gathered}
\int_{\mbr^n}|\chi(k)|   |\widehat q(k)|^2 \d k \leq  C e^{C\tau}\|\La_{g_0, q} - \La_{g_0, 0}\|_*^2 +   C  e^{C\tau}\|q\|_{L^\infty(\Omega)}^2, 
\end{gathered}
\eeq
where the constant $C$  depends on $M$ and $\eps$. From the assumption $\|q\|_{L^\infty(\Omega)} \leq N \|\chi(D) q \|_{L^2(\mbr^n)}$ and the Plancherel theorem, we get 
\beq
\|\chi(D) q\|_{L^2(\mbr^n)}^2 \leq C  e^{C\tau} \|\La_{g_0, q} -  \La_{g_0, 0}\|_*^2 +   N^2 C  e^{ C\tau} \|\chi(D) q\|_{L^2(\mbr^n)}^2.
\eeq
Let $\tau>\frac{M}{\eps}.$ Then we can choose $N$ sufficiently small such that $N^2 C  e^{C \tau} < \frac{1}{2}$. We deduce 
\beq
\|q\|_{L^\infty(\Omega)}^2 \leq 2 C  e^{C\tau} \|\La_{g_0, q} -  \La_{g_0, 0}\|_*^2  \leq % 2 C  N^2 \Big(\frac{1}{2 N^2 C_1}\Big)^{2/C_2}
N^{-2} \|\La_{g_0, q} -  \La_{g_0, 0}\|_*^2. 
\eeq
This implies the desired estimate \eqref{eq-lipest}.
\epf

Finally, we recover the conductivity perturbation $\delta$ from the potential $q$.  
\begin{theorem}\label{thm-main3d} 
Let $n\geq 3,$ $s>\max(\frac{n}{2}, 2)$, $M>0$ and $c_0<1$. Suppose that $A\in C^1(\overline\Omega)$ and $\delta\in C^1(\overline\Omega)$ is compactly supported in $\Omega$. Then for $N>0$ sufficiently small and  for $\gamma = (1 + \delta)(I + A)$ satisfying {\rm(i)} $\|\delta\|_{L^\infty(\Omega)}\leq c_0$ and
{\rm(ii)} $\|q\|_{H^s(\Omega)}\leq M$ and $q \in \mcd_{N}$,
where $q$ is defined as in \eqref{eq-conf}, we have for some $C$ independent $\delta$,  
\beqq\label{eq-lipest1} 
\|\delta\|_{L^\infty(\Omega)} \leq  C N^{-1}\|\La_{\gamma} - \La_{\gamma_0}\|_*. 
\eeqq
\end{theorem}
\begin{proof}  
Using the definition of $q$ in \eqref{eq-conf}, we get 
\beq
\lap_g (\phi^{\frac{n-2}{4}}) = q \phi^{1 + \frac{n-2}{4}} \quad \text{in } \Omega.
\eeq
Also, since $\delta$ is compactly supported in $\Omega$, we have 
$\phi^{\frac{n-2}{4}} = 1$ on  $\p \Omega.$ Then by the elliptic regularity estimate and the fact $\|\phi\|_{L^\infty(\Omega)}\leq 1+\|\delta\|_{L^\infty(\Omega)}$,  we get 
\begin{align*}
\|\phi^{\frac{n-2}{4}} - 1\|_{L^2(\Omega)}\leq \|\phi^{\frac{n-2}{4}} - 1\|_{H^{2}(\Omega)} \leq \|q \phi^{1 + \frac{n-2}{4}}\|_{L^2(\Omega)}   \leq C\|q\|_{L^2(\Omega)}\leq C \|q\|_{L^\infty(\Omega)}. 
 \end{align*}
Since $\phi = 1 + \delta$ and $\|\delta\|_{L^\infty(\Omega)} \leq c_0 <1$, and the function $s\mapsto (1 + s)^{\frac{n-2}{4}} - 1$ is invertible for $|s| \leq c_0$, we deduce 
\beqq\label{eq-max1}
\begin{gathered} 
\|\delta\|_{L^2(\Omega)}\leq   C\|q\|_{L^\infty(\Omega)}, 
 \end{gathered}
\eeqq
where $C$ depends on $c_0.$ Using Proposition \ref{prop-main1}, we finish the proof of the theorem. 
\end{proof}

\begin{remark}
Theorem \ref{thm-main3d} provides conditions on the conductivity $\gamma$ in terms of the potential function $q$, i.e., $q\in\mathcal{D}_N$, so that the inversion of $\La_\gamma$ is stable. In practice, for a given training set, one can always find $M, N$ such that these conditions are satisfied. It would be interesting to find a more direct characterization on the admissible set of conductivities $\gamma$. In this regard, one can adapt the proof of Theorem \ref{thm-main2d} to higher dimensions at least for small $\delta$, which might also be useful in practice. 
\end{remark}

\section{Numerical experiments and discussions}\label{sec-num}
In this section, we present several numerical experiments to illustrate the deep Calder\'on method for the two-dimensional anisotropic EIT, including the instability on out-of-distribution test data.
\subsection{Implementation details} First we describe the setting of numerical experiments. The EIT inverse problem is posed on the unit disk $\Omega=\{x:|x|<1\}$. The conductivity images are of size $32\times32$ (with zero padding outside $\Omega$). We employ $L-1=31$ trigonometric current patterns as the Neumann boundary conditions, with the $k$th current density $g_k$ given by
\begin{equation}\label{current}
    g_k=\begin{cases}
    \cos(k\theta),\quad &k=1,2,\cdots,L/2,\\
    \sin((k-L/2)\theta),\quad &k=L/2+1,L/2+2,\cdots,L-1.
\end{cases}
\end{equation}
To obtain the Dirichlet data, we employ the standard Galerkin FEM with conforming linear elements to solve the direct problem, as is commonly used in numerical simulation. However, this corresponds to the inverse problem using the Neumann-to-Dirichlet data instead of the Dirichlet-to-Neumann data in the analysis. It is known that knowing the Dirichlet-to-Neumann map $\Lambda_\gamma$ is equivalent to knowing the Neumann-to-Dirichlet map \cite[Section 2.9]{FSU}. The stability estimate for the Neumann-to-Dirichlet map can be obtained from that for the Dirichlet-to-Neumann map $\Lambda_\gamma$; see e.g., \cite[Section 4.1]{AlbertiSantacesaria:2019}. In any case, these maps are only approximations of the actual measurements in EIT \cite{Cheney}. Note also the use of finitely many pairs of Cauchy data gives only an approximate Neumann-to-Dirichlet map, which is commonly adopted in practical inversion.

Following the standard isotropic Calder\'on method in \cite{Cen}, we obtain the Calder\'on reconstruction image of the anisotropic conductivity. See  Section 2.1 for further discussions on the choice.
The Calder\'{o}n method involves one hyper-parameter, the truncation radius $R$, whose determination requires some care. Note that in Example \ref{exam1}, the Fourier transform of a Gaussian is also a Gaussian, which is smooth and decays exponentially in the Fourier domain. Thus the frequency components beyond the cut-off radius $R$ are already negligible. However in Example \ref{exam2}, the Fourier transform of piecewise constant functions oscillates and decays slowly. Beyond the cut-off  radius $R$, there still exist significant high-frequency components arising from sharp edges. To ensure training stability, we set a smaller truncation radius $R$ to prevent oscillations that could hinder the training of the U-net $f_\theta$. 

Next we employ a U-net $f_\theta$ to postprocess the Calder\'{o}n reconstruction. The U-net \cite{Ron15} is one state-of-the-art convolutional encoder-decoder architecture for image segmentation and medical image reconstruction. The structure of the U-net is schematically illustrated in Fig.\ \ref{unet}, which consists of a contracting path (encoder) and an expanding path (decoder). The contracting path comprises multiple blocks, each containing one convolution, one activation, and one max-pooling layer. The role of the max-pooling layers is twofold: to down-sample the feature maps, thereby reducing their spatial dimensions, and to enhance the extraction of salient features by preserving the most activated responses within local regions. Each block begins with a convolutional layer
%$W_i^{\mathrm{c}} * z_i^{\mathrm{c}} + b_i^{\mathrm{c}}$ 
to detect local patterns. This is followed by a rectified linear unit (ReLU) activation $\rho(x)=\max\{x,0\}$. Finally, a max-pooling layer reduces the spatial dimensions by outputting the maximum value within each non-overlapping rectangular region, providing a form of translation invariance and reducing computational load. 
%The entire process for a single block is captured by:
%$$z_{i+1}^{\mathrm{c}} := \tau_i^{\mathrm{c}}\left(z_i^{\mathrm{c}}\right) = \mathcal{M}\left(\rho\left(W_i^{\mathrm{c}} * z_i^{\mathrm{c}}+b_i^{\mathrm{c}}\right)\right).$$
The input to the first contracting block is the Calder\'{o}n reconstruction $\tilde{\gamma}$. Similarly, the expanding path comprises several blocks, each using a transposed convolution to upscale the input. 
%The expanding block at the $(i+1)$-th layer is defined as:
%$$z_{i+1}^e := \tau_i^e (Z_i^e) = \mathcal{C}\left(\rho \left( \mathcal{T} (Z_i^e, W_i^e, b_i^e) \right)\right),
%$$
%where $\mathcal{T}$ denotes the transposed convolution operator, $W_i^e$ and $b_i^e$ are the convolutional filter and bias vector at the $i$-th layer, and $\mathcal{C}$ represents the concatenation operation. 
Additionally, we incorporate a skip connection from the input to the output at each level of the U-net.
%$$
%z_{i+1}^e := \tilde{\tau}_i^e (z_i^e) = \tau_{L-i+1}^c (z_{L-i+1}^c) + \tau_i^e (z_i^e),$$
%where $L$ is the total number of levels in the contracting path. 
It encourages the network to learn only the residual—the difference between input and output—thereby avoiding redundancy and alleviating the vanishing / exploding gradient problem during training \cite{He2016}.
%With $L$ contracting and $L$ expanding blocks, the full CNN model $f_\theta$ is represented by
%$$f_\theta (\tilde{\gamma}) = \mathcal{P} \circ \tilde{\tau}_L^e \circ \cdots \circ \tilde{\tau}_1^e \circ \tau_L^c \circ \cdots \circ \tau_1^c (\tilde{\gamma}).$$
%Here, $\mathcal{P}$ denotes the final output layer, which is implemented as a $1 \times 1$ convolutional layer followed by a Leaky Rectified Linear Unit (LeakyReLU) activation function:
%$$
%\text{LeakyReLU}(x) = \begin{cases} 
%x & \text{if } x \geq 0 \\\alpha x & \text{if } x %< 0 
%\end{cases}$$
%where $\alpha$ is a small positive slope coefficient for negative inputs, preventing neurons from "dying" during training \cite{Maas2013}. 
In the final output layer, we also apply a $1 \times 1$ convolutional layer followed by a leaky ReLU activation, which projects the multi-channel feature maps from the last expanding block to the desired output dimensions. The set of all trainable parameters (filters and biases) is denoted by $\theta$. We also experimented with shallower neural networks, which however tend to suffer from instability and poorer reconstruction quality.

\begin{figure}[hbt!]
\centering\setlength{\tabcolsep}{2pt}
\includegraphics[height=9cm]  {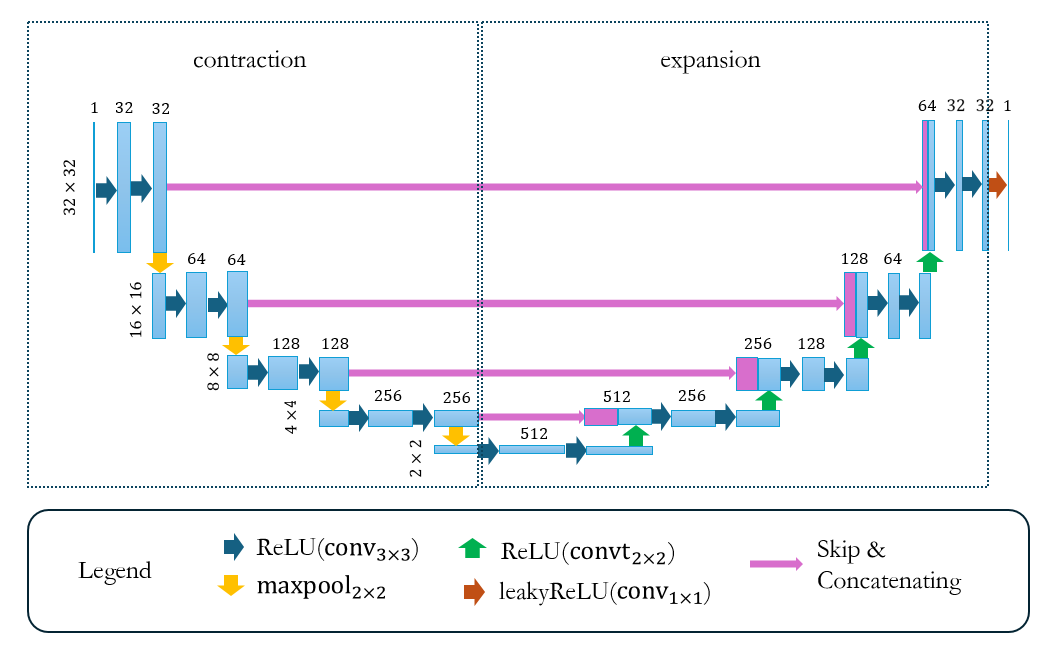}
\caption{\label{unet} The structure of U-net $f_\theta$.}
\end{figure}

To train the U-net $f_\theta$ on the paired training dataset $\{(\widetilde \gamma_n,\gamma_n^\dag)\}_{n=1}^N$ (with $\widetilde\gamma_n$ and $\gamma_n^\dagger$ being the Calder\'{o}n reconstruction and the true conductivity, respectively), we employ the standard mean squared error loss:
\begin{equation}
    \mathcal{L}(\theta) = \frac{1}{N} \sum_{n=1}^{N} \Big\| f_\theta (\tilde{\gamma}_n) - \gamma_n^\dag \Big\|_2^2.
\end{equation}
The total number of trainable parameters is 7,766,629 when the input image is of size $32\times32$.
Throughout, we employ $N=2,000$ training data pairs and validate the generalization ability of the neural network on 300 unseen samples. 
%We generate the noisy operator $\Lambda_\gamma^\delta$ (in its discrete form) by   $\Lambda_\gamma^\delta = \Lambda_{\gamma_0} + (1 + \xi^\delta) \odot (\Lambda_\gamma - \Lambda_{\gamma_0})$. Here, $\xi^\delta$ is an $L \times (L-1)$ random matrix whose each element $\xi_{ij}^\delta$ is independently and identically distributed following a Gaussian distribution with mean 0 and variance $\delta^2$, i.e., $\xi_{ij}^\delta \sim N(0, \delta^2)$. 
To minimize the loss $\mathcal{L}(\theta)$, we use the Adam optimizer \cite{KingmaBa} with a batch size 10. The training is performed on a computing node equipped with an NVIDIA L40 GPU (with 48GB memory), using the PyTorch platform. To quantitatively measure the accuracy of a reconstruction, we compute its $L^1(\Omega)$ and $L^2(\Omega)$ relative errors. All errors are computed after training is completed. The Python codes for reproducing all the experiments will be made available at the github link \url{https://github.com/hhjc-web/dcm-for-anisotropic-EIT}. 

\subsection{Numerical results and discussions} %In this part we present numerical results for two examples, one with the Gaussian and the other with the piecewise constant, to complement the theoretical findings.
In this part, we present numerical results for two examples corresponding to the settings in Section \ref{sec:trainingset}. The experiments are designed to verify the stability/instability analysis and test the generalization ability of the deep Calder\'on method. 
We remark that it is generally not clear how to sample functions from the admissible set $\mce_{M, N}$ in Theorem \ref{thm-main2d} for neural network training so as to achieve the universal approximation property, which in practice is limited by the complex loss landscape. Thus, in the experiments below, we use  ``low-dimensional" subsets of $\mce_{M, N}$ for training. We aim to obtain an approximation of the inversion of $\La_\gamma$ near the sets, which is sufficient for verifying the stability properties. 

\begin{example}\label{exam1}
Consider the setting of Gaussian data in Section \ref{sec:trainingset}. The background matrix $A$ and scalar-valued function $\delta$ are respectively given by  $$A=\begin{pmatrix}
	0&0.2e^{-|x|^2}\\0.2e^{-|x|^2}&0
\end{pmatrix}\quad\mbox{and}\quad\delta(x)=ae^{-\frac{b|x-x_0|^2}{2}}.$$
To show the effectiveness of the deep Calder\'{o}n method for recovering $\gamma$, we investigate three cases, and train the model separately for each case:
\begin{itemize}
\item[{\rm(i)}] $a\sim U(-0.5,0.5)$, $b\sim U(30,50)$, and $x_0\sim U(-0.5,0.5)^2$;
\item[{\rm(ii)}] $a\sim U(0.2,1.2)$, $b\sim U(30,50)$, and $x_0\sim U(-0.5,0.5)^2$;
\item[{\rm(iii)}] $a\sim U(0.2,1.2)$, $b=40$, and $x_0\sim U(-0.5,0.5)^2$.
\end{itemize}
In cases {\rm(i)} and {\rm(ii)}, we focus on the stability of the trained neural network. The difference is that in case {\rm(i)}, the conductivity perturbation $\delta$ can change signs with respect to the background but in case {\rm(ii)} the conductivity perturbation $\delta$ has a fixed sign. In case {\rm(iii)}, we mainly focus on the generalization ability of the trained neural network. To show the stability of the training and to capture the characteristics of U-net training dynamics, we present the recovery results at different stages of training, on the test data with $\delta_{t}(x)=a_{t}e^{-\frac{b_{t}|x|^2}{2}}$, with $b_{t}\in\{20,40,60\}$. Note that only $b_t = 40$ is in distribution, and the other two are out-of-distribution. 
\end{example}

In each of the settings (i)--(iii), we first generate 2,300 training data pairs, then use 2,000 of them for training the U-net and the remaining 300 pairs for validation. No test data are generated during the training process. The truncated radius $R$ of the Calder\'on method and training details are summarized in Table \ref{table:exam1}.

\begin{table}[hbt!]
\centering
\begin{threeparttable}
\caption{\label{table:exam1} The truncation radius $R$ of Calder\'on method and training details for Example~\ref{exam1}.}
\centering
\begin{tabular}[5pt]{c|c|c|c|c|c}
\toprule
Case & $R$ & training data & validation data &  learning rate & training time (min)\\ 
\midrule
(i) & 1.8 & 2,000 & 300 & $5\times10^{-4}$ & 10\\
(ii) & 1.8 & 2,000 & 300 & $5\times10^{-4}$ & 6\\
(iii) & 1.8 & 2,000 & 300 & $1\times10^{-4}$ & 18\\
\bottomrule
\end{tabular}
\end{threeparttable}
\end{table}

In case (i), we show the reconstructions and the training progress (by evaluating the model on the test set at different stages of training) in Fig.\ \ref{fig:exam1.1}. At the early training stage (after 20 epochs), the neural network is still suboptimal, and there are pronounced artifacts in the reconstructions, especially the dip close the left boundary, regardless of whether the test samples are from the training data distribution or outside of the distribution. This is likely due to the fact that the training has not identified a good parameter configuration. Note that we have actually trained many networks with the experimental setting, and the artifacts observed at the 20th epoch appear in the vast majority of the independent training runs. Hence it is not an isolated phenomenon but rather a persistent behavior under the given configuration. Further, the reconstructions exhibit noticeable oscillations overall, which is especially pronounced for large $b$.  After 40 epochs, the oscillations still persist and remain  pronounced in regions of higher variance. However, the neural network has significantly improved and the artifacts have largely been eliminated. Eventually, after 400 epochs, the training process has stabilized. The oscillations have become almost negligible, and the trained model performs well on in-distribution data. For out-of-distribution data, the reconstruction quality is very good for $b_t=20$ (small) but becomes worse for $b_t = 60$ (large). 
These results are consistent with the theoretical predictions in Section \ref{sec:trainingset}. Next we show one case in which $a_t$ is below 0 in the last row of Fig. \ref{fig:exam1.1}. While the neural network ultimately achieves decent results, the artifacts persists at the 40 epoch, suggesting more severe instability in the regime. Furthermore, the final $L^1(\Omega)$ and $L^2(\Omega)$ errors shown in Table \ref{table:exam1.1} are also larger that the preceding cases.

\begin{table}[hbt!]
\centering
\begin{threeparttable}
\caption{\label{table:exam1.1} The relative $L^1(\Omega)$ and $L^2(\Omega)$ errors of the numerical reconstructions for Example~\ref{exam1}(i) and Example \ref{exam2}(i).}
\centering
\begin{tabular}[5pt]{ccc|ccc}
\toprule
\multicolumn{3}{c}{Example \ref{exam1}(i)}& \multicolumn{3}{c}{Example \ref{exam2}(i)}\\
\cmidrule(lr){1-3} \cmidrule(lr){4-6}
$(b_t,a_t)$ & $L^1(\Omega)$ & $L^2(\Omega)$ & $(\beta_t,\alpha_t)$ & $L^1(\Omega)$ & $L^2(\Omega)$ \\ 
\midrule
$(60,0.05)$ & 3.64e-4 & 5.11e-4 & $(0.3,0.05)$ & 6.77e-4 & 1.90e-3\\
$(40,0.05)$ & 2.48e-4 & 2.99e-4 & $(0.5,0.05)$ & 1.34e-3 & 3.22e-3\\
$(20,0.05)$ & 4.02e-4 & 5.58e-4 & $(0.7,0.05)$ & 3.58e-3 & 7.20e-3\\
$(40,-0.05)$ & 4.41e-4 & 5.60e-4 & $(0.8,1.05)$ & 7.59e-2 & 1.21e-1\\
\bottomrule
\end{tabular}
\end{threeparttable}
\end{table}

\begin{figure}[t]
\centering\setlength{\tabcolsep}{1pt}
\begin{tabular}{cccccc}
\includegraphics[height=1.9cm]  {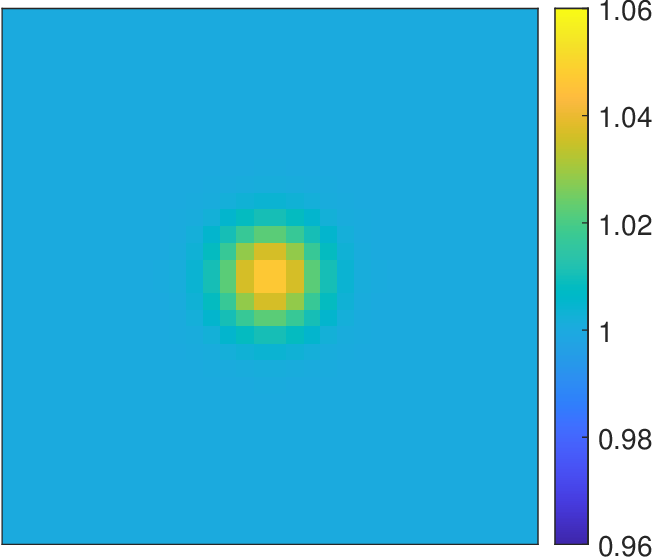} & \includegraphics[height=1.9cm]  {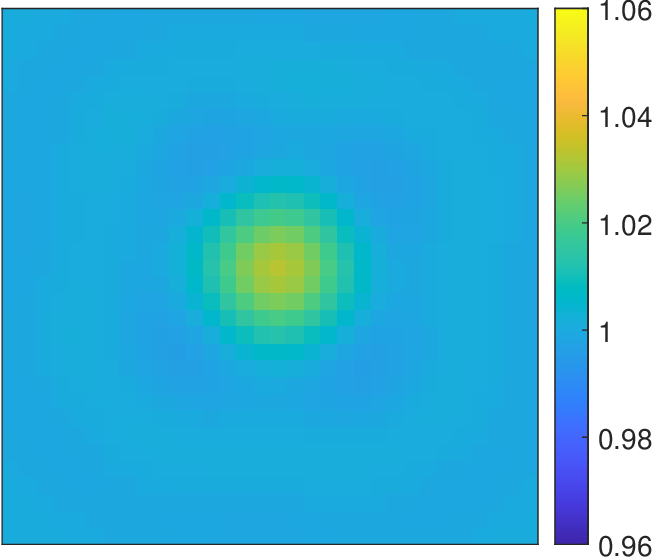} &
\includegraphics[height=1.9cm]  {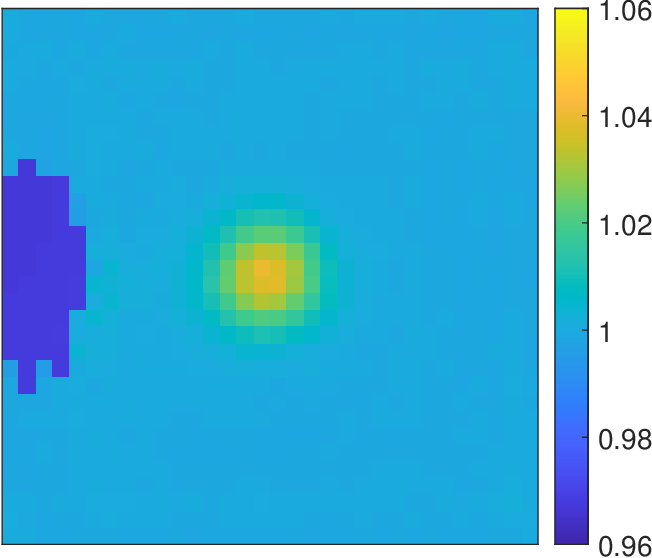} &
\includegraphics[height=1.9cm]  {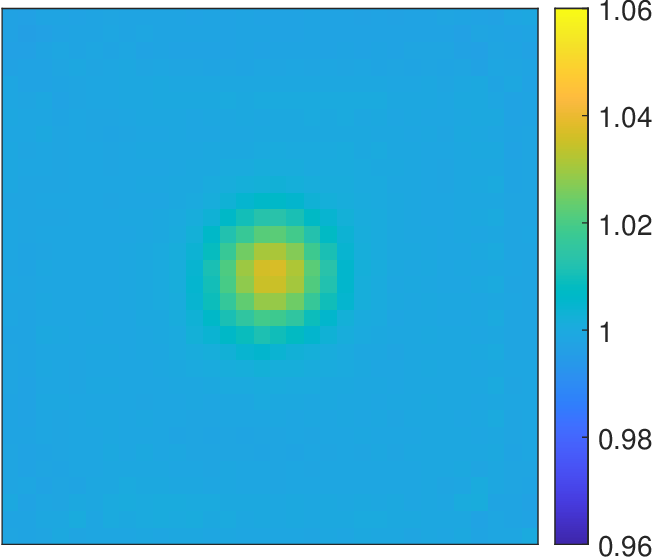} &
\includegraphics[height=1.9cm]  {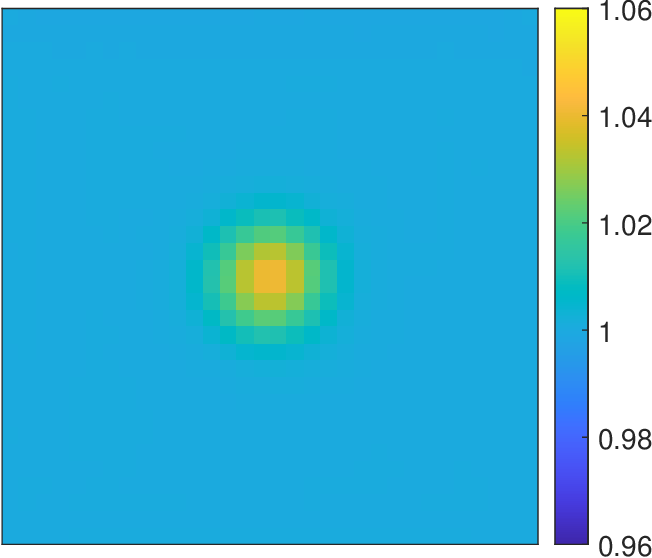} &
\includegraphics[height=1.9cm]  {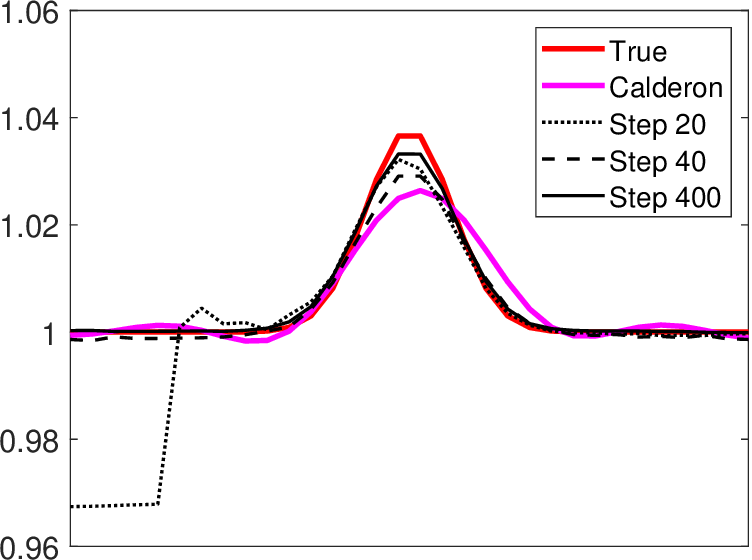} \\
\includegraphics[height=1.9cm]  {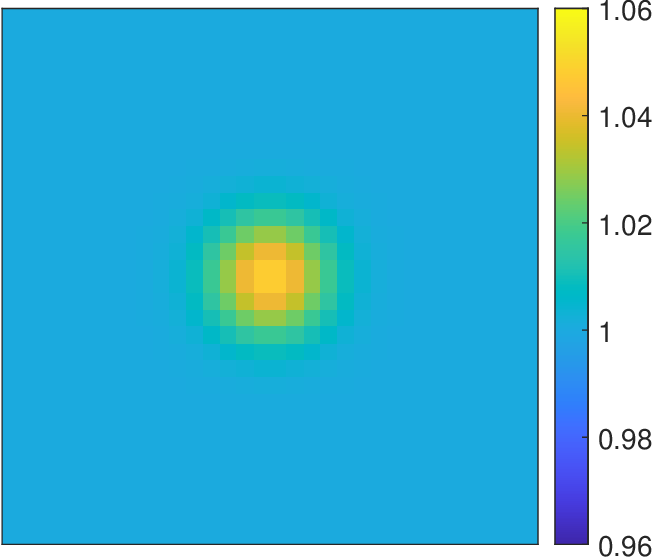} & \includegraphics[height=1.9cm]  {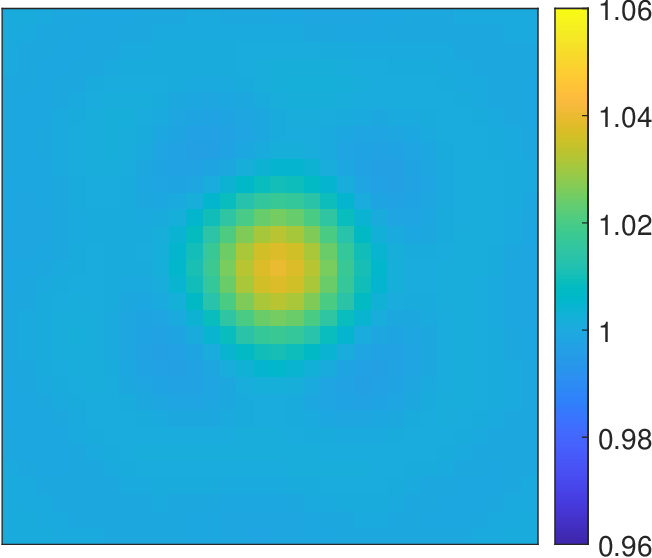} &
\includegraphics[height=1.9cm]  {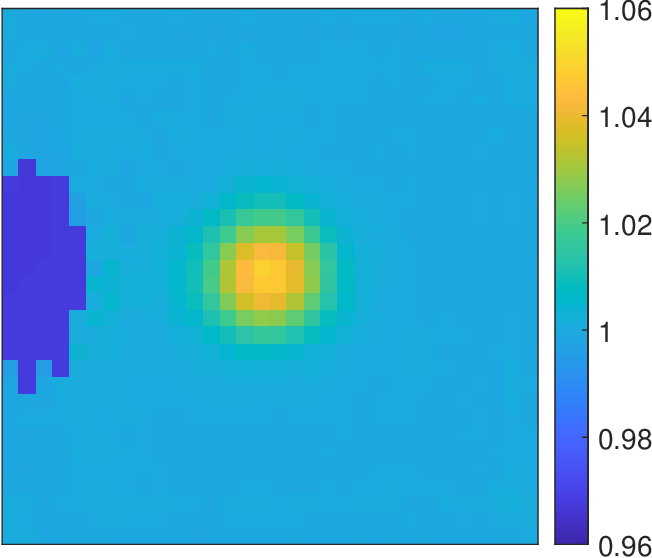} &
\includegraphics[height=1.9cm]  {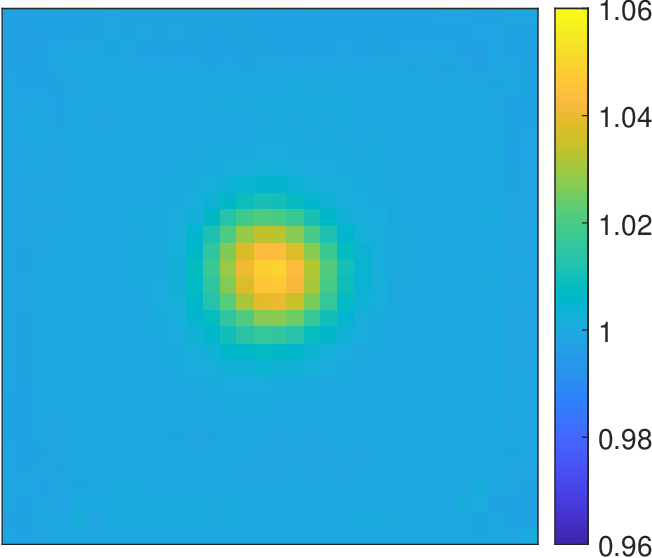} &
\includegraphics[height=1.9cm]  {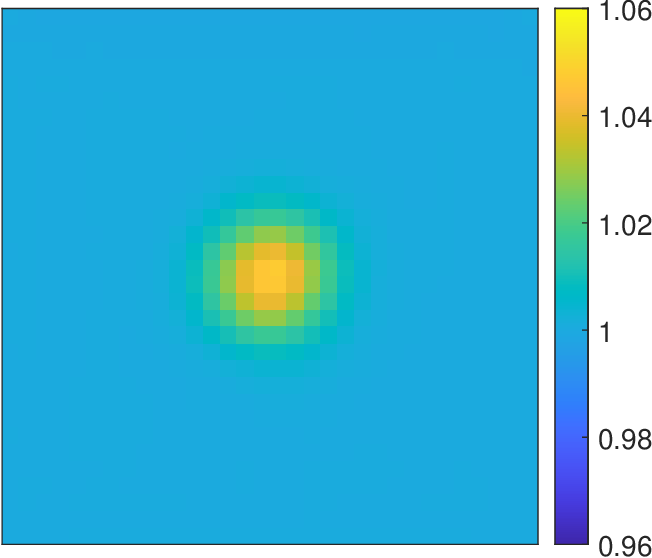} &
\includegraphics[height=1.9cm]  {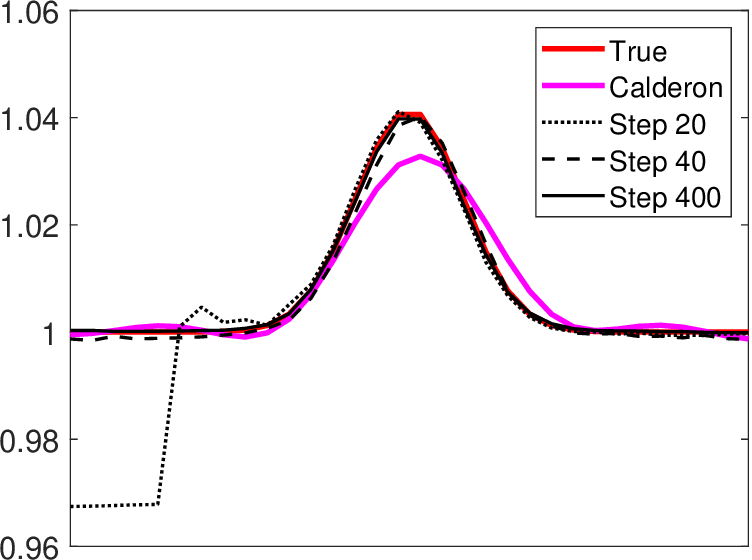}\\
\includegraphics[height=1.9cm]  {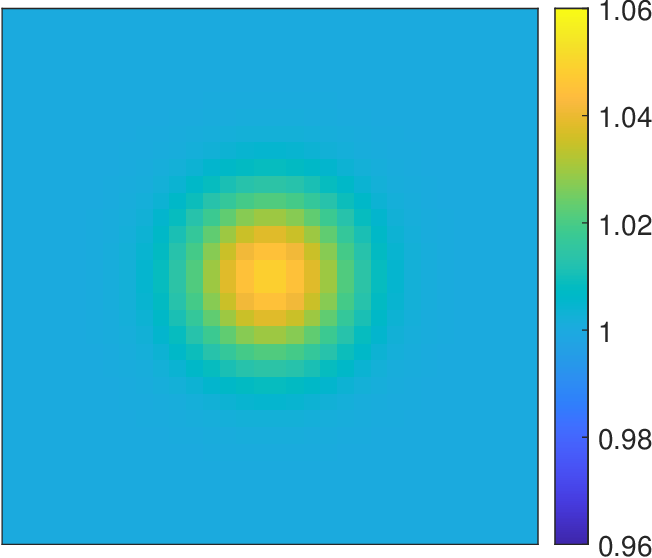} & \includegraphics[height=1.9cm]  {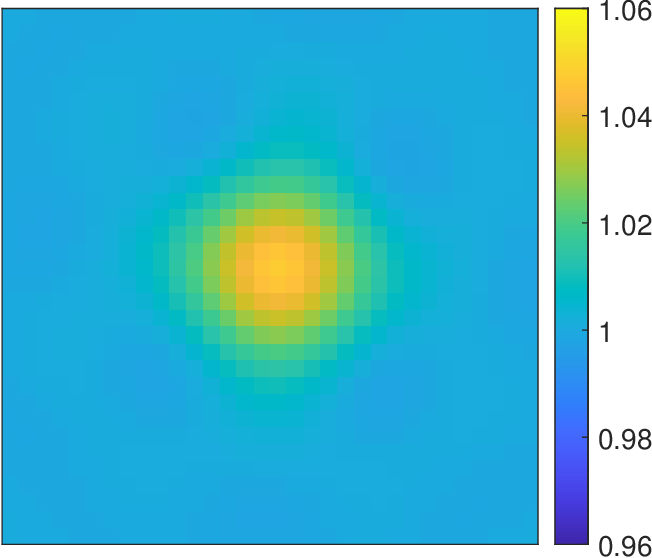} &
\includegraphics[height=1.9cm]  {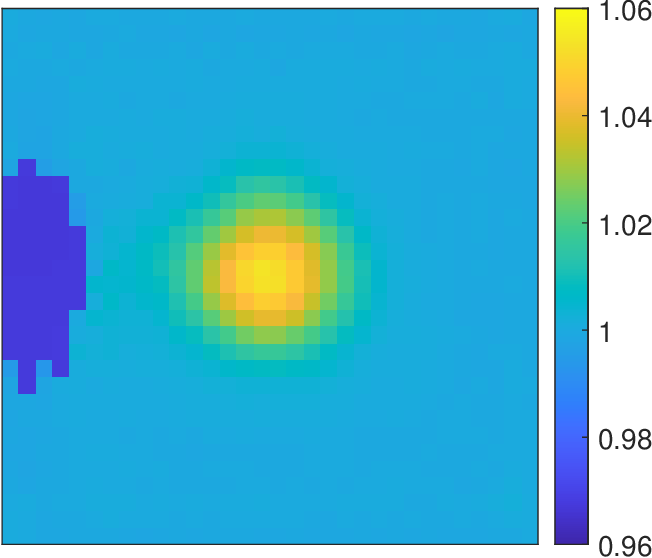} &
\includegraphics[height=1.9cm]  {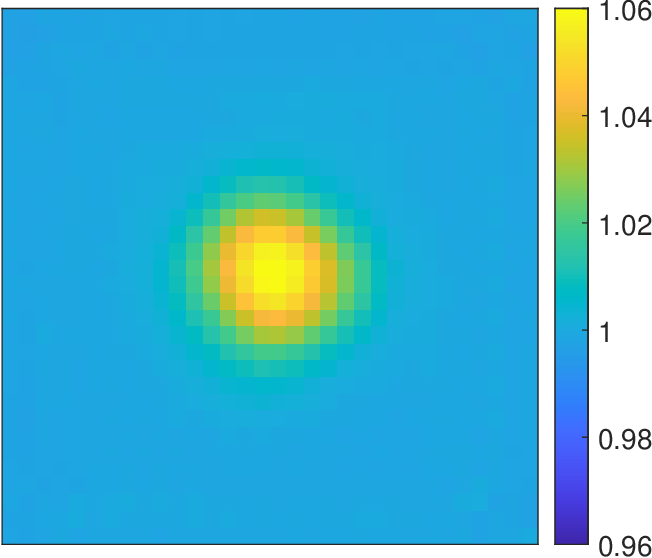} &
\includegraphics[height=1.9cm]  {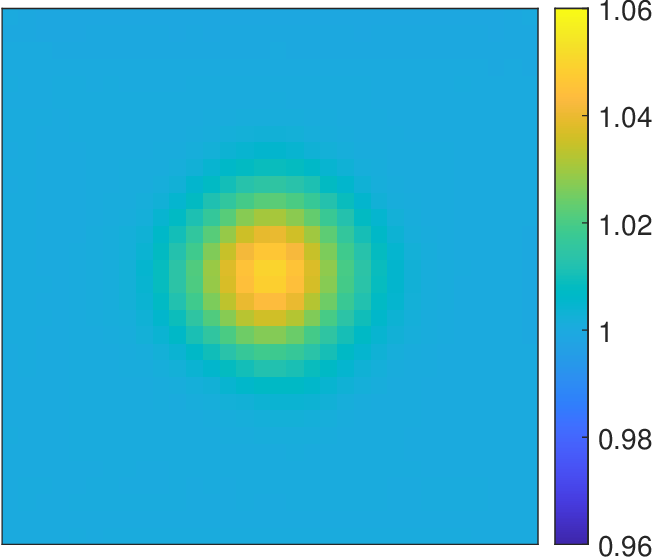} &
\includegraphics[height=1.9cm]  {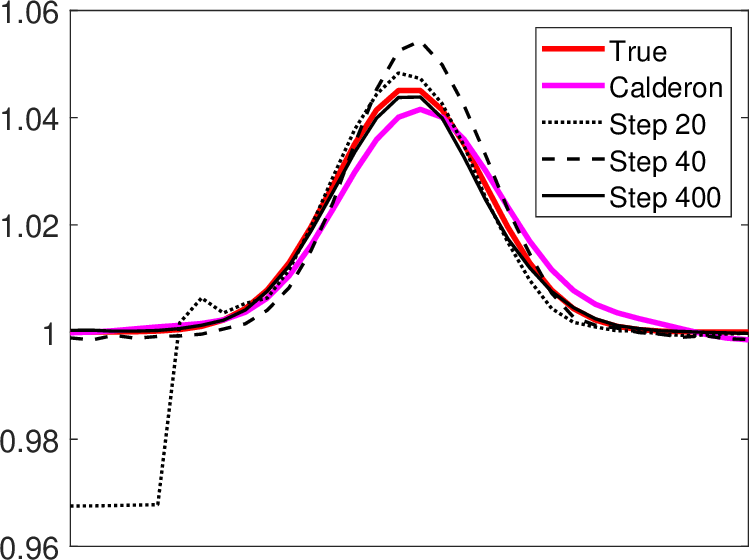}\\
\includegraphics[height=1.9cm]  {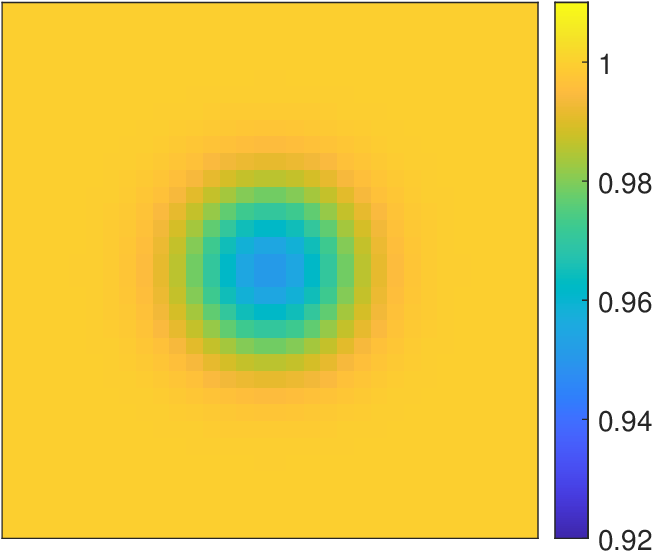} & \includegraphics[height=1.9cm]  {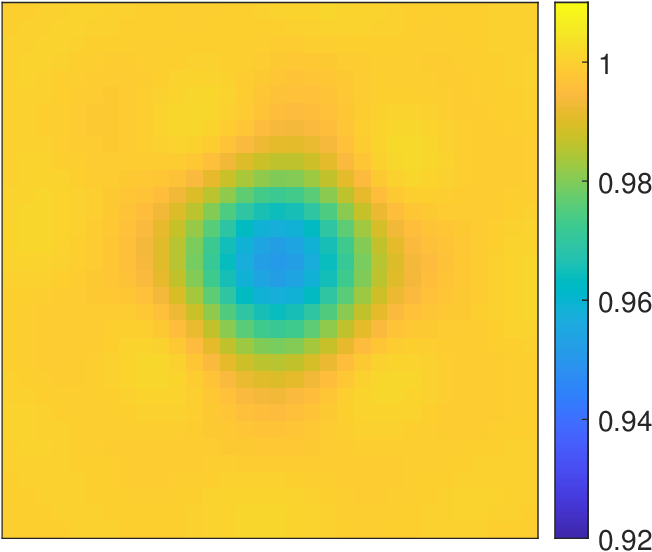} &
\includegraphics[height=1.9cm]  {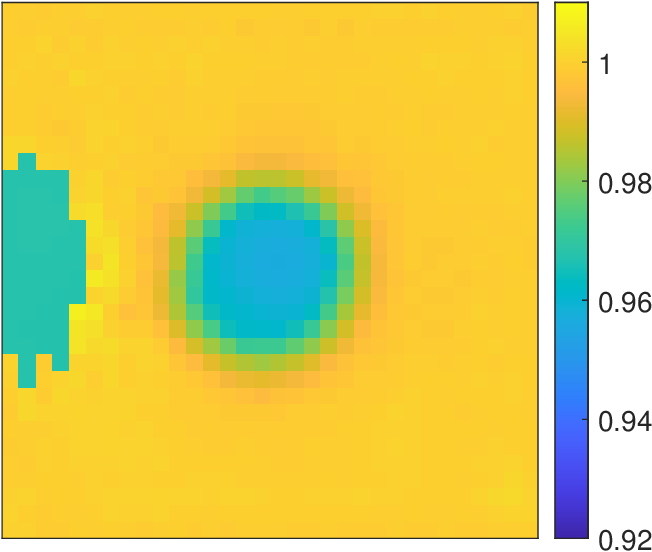} &
\includegraphics[height=1.9cm]  {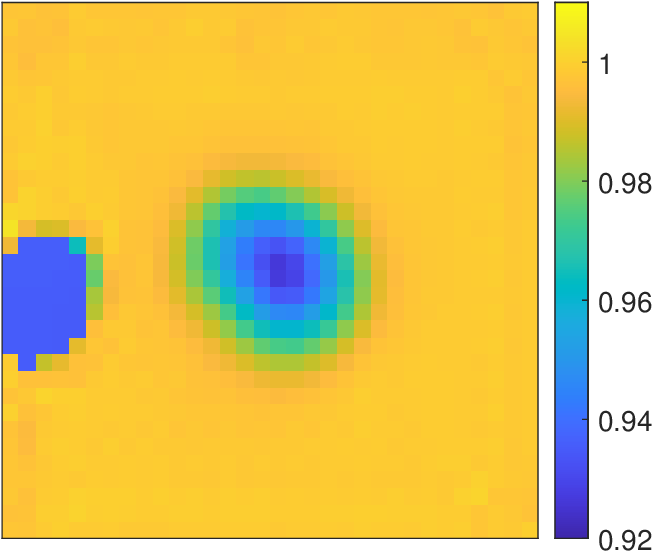} &
\includegraphics[height=1.9cm]  {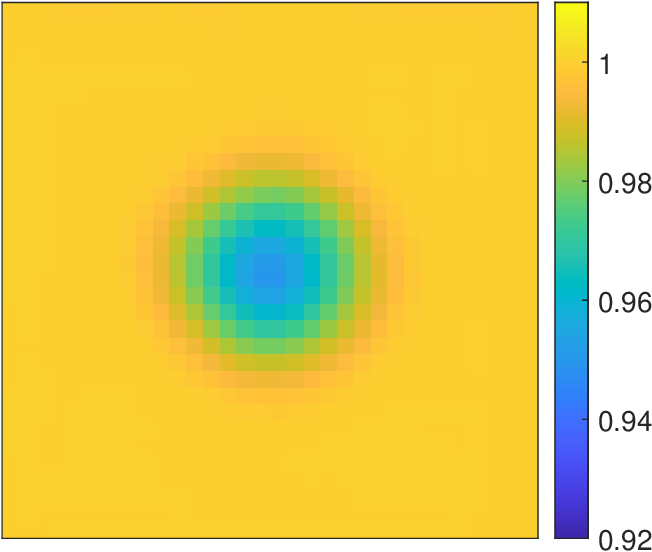} &
\includegraphics[height=1.9cm]  {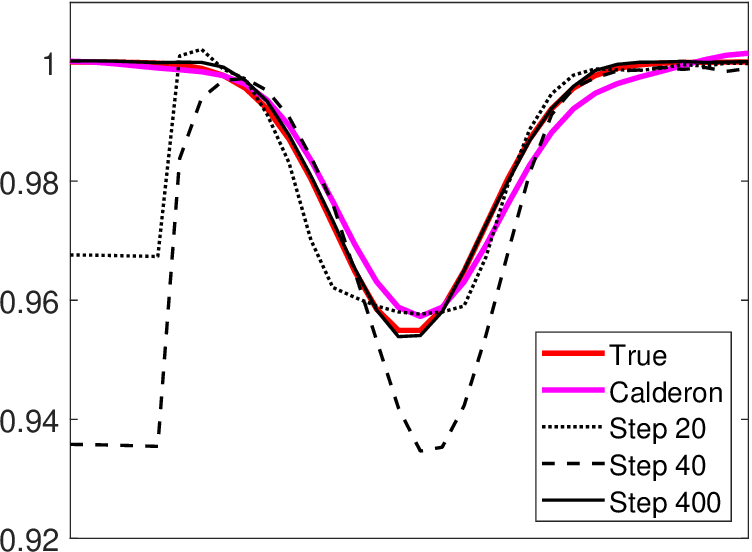} \\
        (a) true & (b) Calder\'on & (c) step 20 & (d) step 40 & (e) step 400 & (f) slice
	\end{tabular}
	\caption{\label{fig:exam1.1} The true image, Calder\'on reconstruction, deep reconstructions (after 20, 40 and 400 epochs) and slice at $x_2=0$ for Example \ref{exam1} (i), with the settings from top to bottom being 
 $(b_{t},a_t)=(60,0.05)$, $(b_t,a_t)=(40,0.05)$, $(b_t,a_t)=(20,0.05)$ and $(b_t,a_t)=(40,-0.05)$.}
\end{figure}

\begin{comment}
\begin{figure}[t]
\centering\setlength{\tabcolsep}{1pt}
\begin{tabular}{cccccc}
\includegraphics[height=1.9cm]  {ex1true4.eps} & \includegraphics[height=1.9cm]  {ex1input4.eps} &
\includegraphics[height=1.9cm]  {ex1pred4it20.eps} &
\includegraphics[height=1.9cm]  {ex1pred4it40.eps} &
\includegraphics[height=1.9cm]  {ex1pred4it400.eps} &
\includegraphics[height=1.9cm]  {ex1cutoff4.eps} \\
        (a) true & (b) Calder\'on & (c) step 20 & (d) step 40 & (e) step 400 & (f) slice
	\end{tabular}
	\caption{\label{fig:exam1.1below} The true image, Calder\'on reconstruction, deep reconstructions (after 20, 40 and 400 epochs) and slice at $x_2=0$ for Example \ref{exam1} (i), with $b_{t}=40$, $a_t=-0.05$.}
\end{figure}
\end{comment}

In case (ii), the training data distributions are nearly identical to case (i), except for the amplitude $a$ of the Gaussian. This case allows only the conductivity distribution above the background value 1. The test result after 20 epochs is shown in Fig.\ \ref{fig:exam1.2}. It is observed that the training process is much more stable and the convergence is much faster than in case (i), and the results are already quite satisfactory at the 20th epoch. In addition, the reconstructions are now free from oscillations. Fig.\ \ref{fig:exam1.2} also shows the results after 200 epochs, for which neural network training has nearly stabilized, and the trained neural network can also produce satisfactory approximations for test data in the training set. The results on the data outside the training distribution have a similar behavior to case (i), but are slightly better. Note that the recent theoretical study \cite{WaWa} has shown that to recover conductivities of a fixed sign, it is possible to achieve Lipschitz type stability. The improved stability phenomenon observed in case (ii) (e.g., without obvious oscillations in the training) are likely due to the sign condition. However, the precise mechanism requires further investigation, which we plan to pursue in future work. 

\begin{figure}[t]
	\centering\setlength{\tabcolsep}{2pt}
	\begin{tabular}{ccccc}
		\includegraphics[height=2cm]  {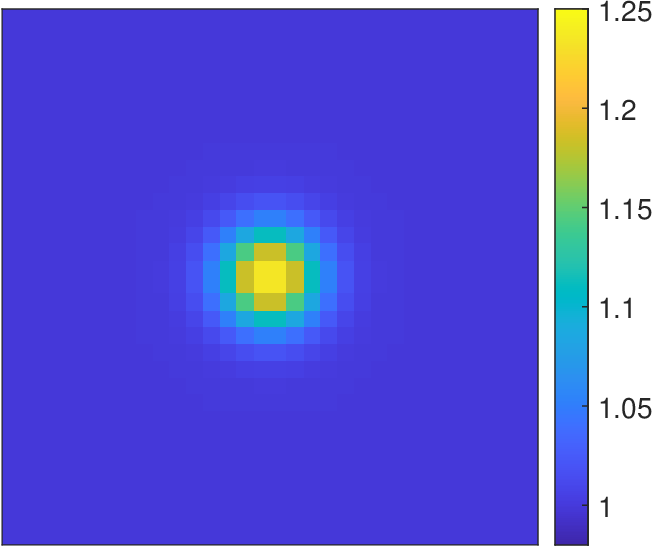} & \includegraphics[height=2cm]  {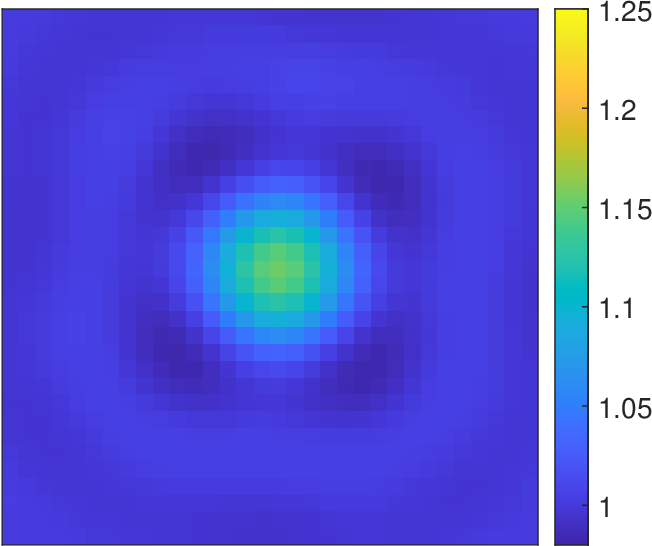} &
		\includegraphics[height=2cm]  {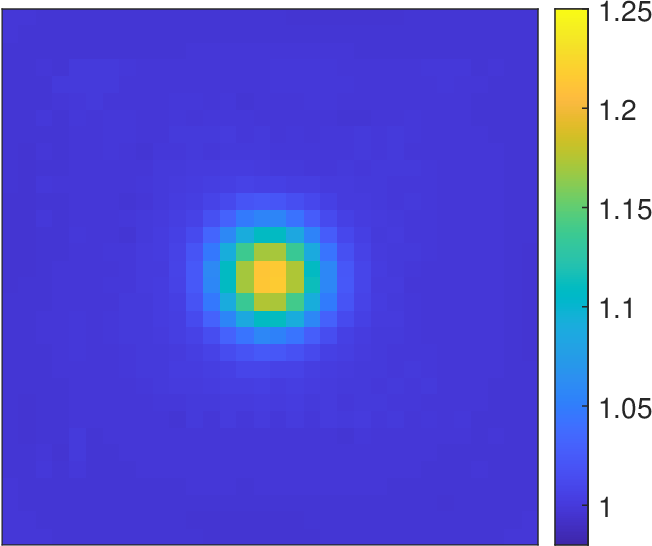} &
        \includegraphics[height=2cm]  {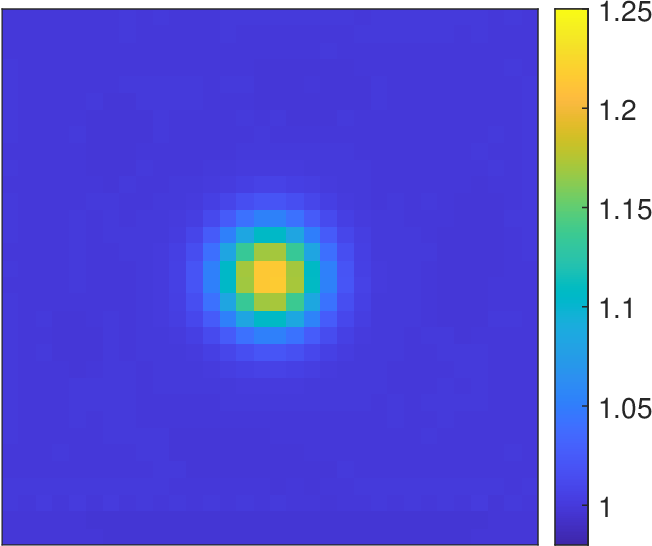} &
		\includegraphics[height=2cm]  {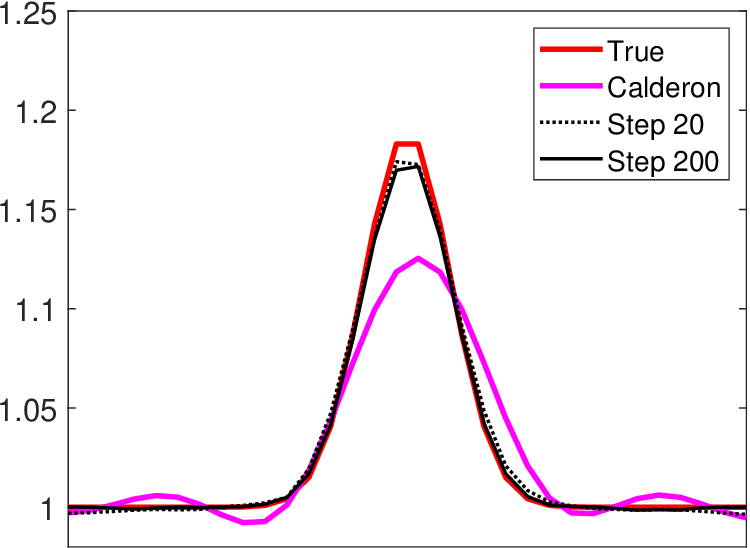} \\
		\includegraphics[height=2cm]  {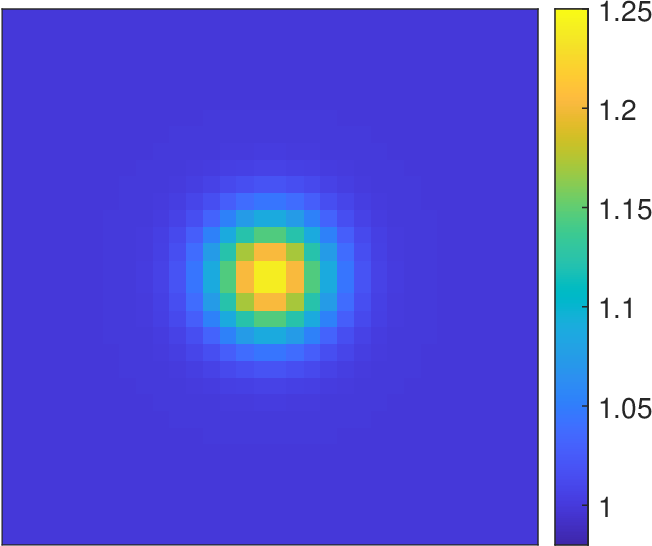} & \includegraphics[height=2cm]  {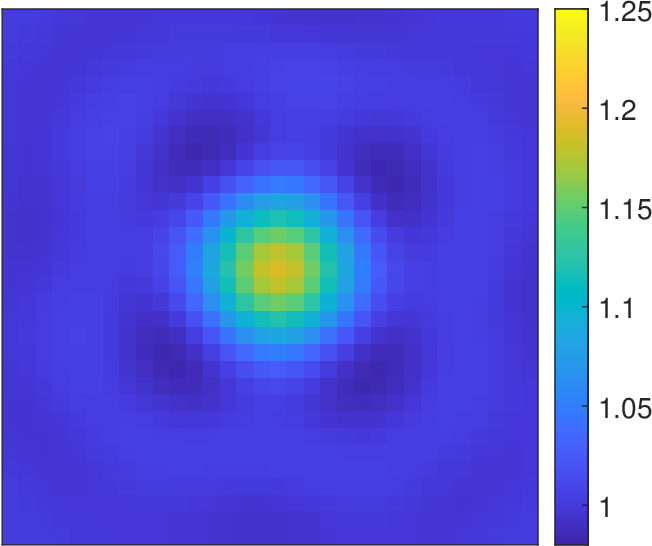} &
		\includegraphics[height=2cm]  {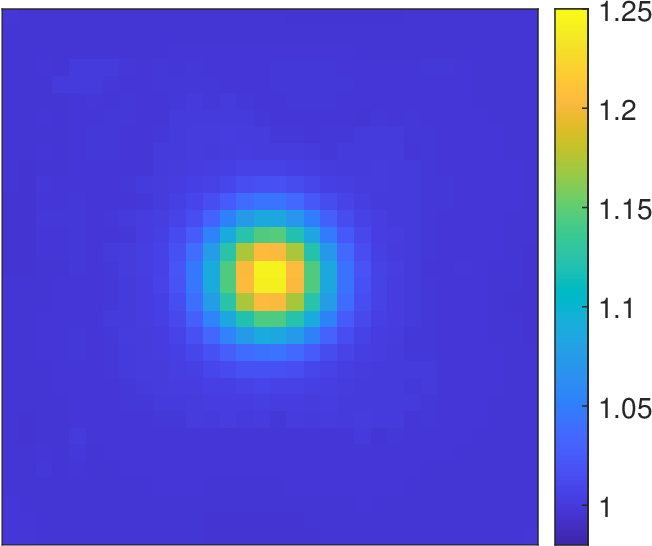} &
        \includegraphics[height=2cm]  {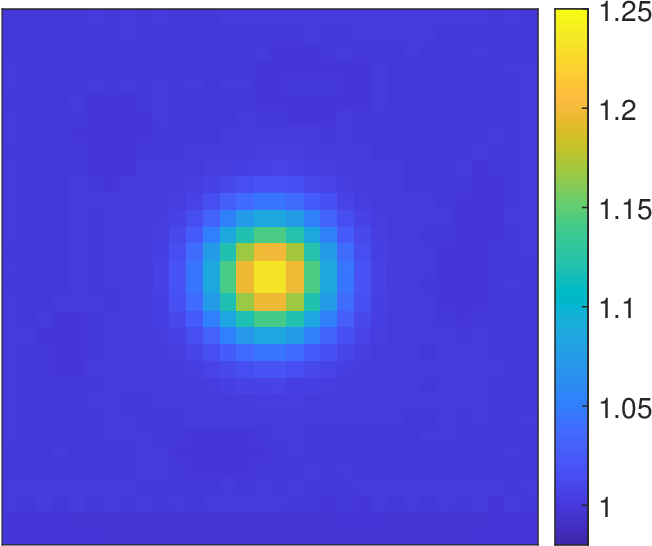} &
		\includegraphics[height=2cm]  {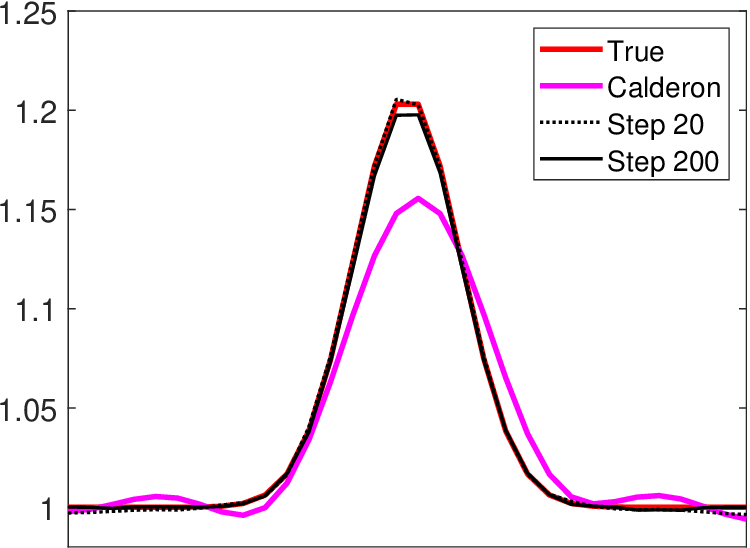}\\
		\includegraphics[height=2cm]  {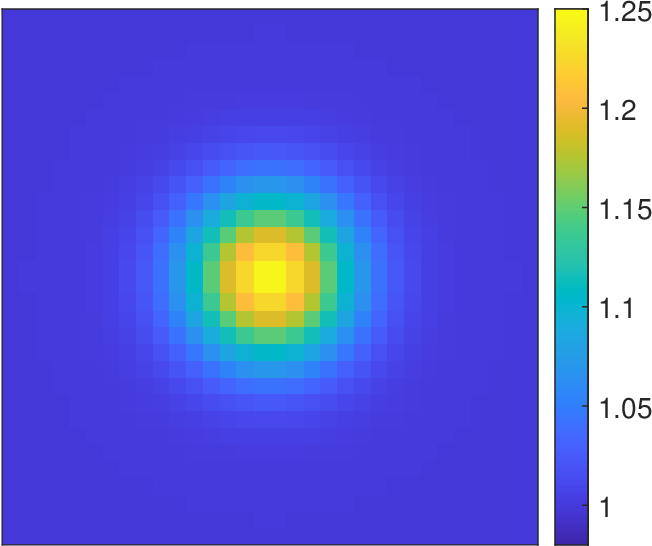} & \includegraphics[height=2cm]  {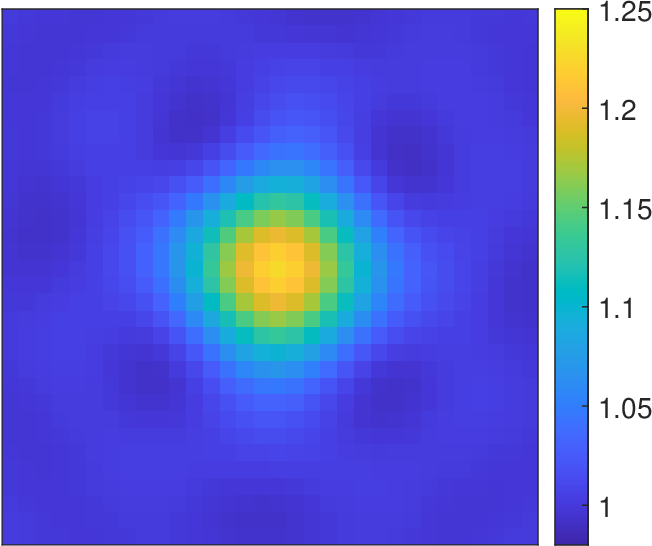} &
		\includegraphics[height=2cm]  {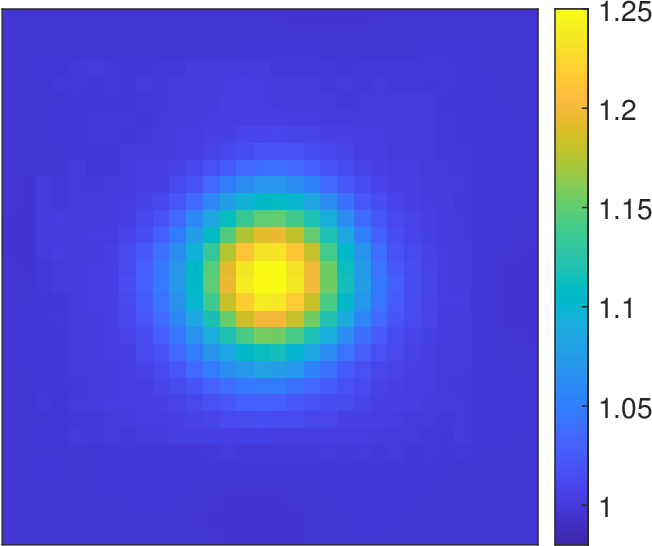} &
        \includegraphics[height=2cm]  {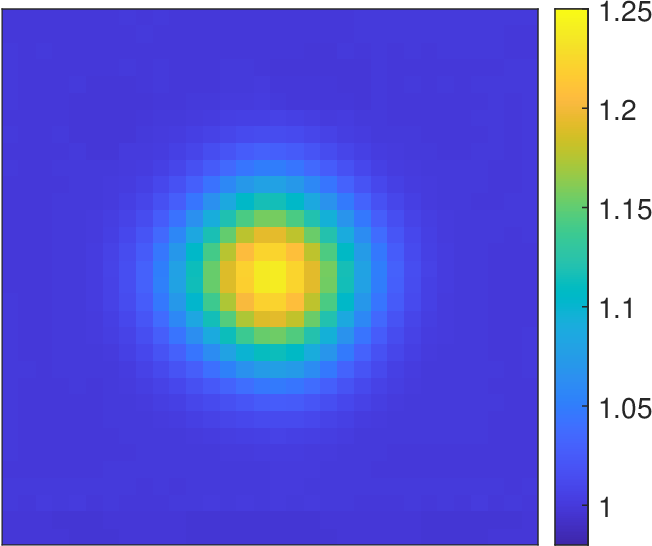} &
		\includegraphics[height=2cm]  {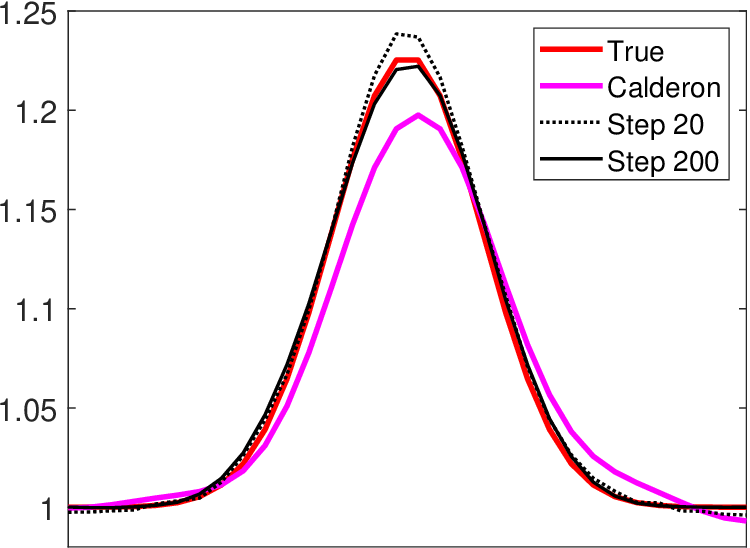}\\
        (a) true & (b) Calder\'on & (c) step 20 & (d) step 200 & (e) slice
	\end{tabular}
	\caption{\label{fig:exam1.2} The true image, Calder\'on reconstruction, deep reconstructions (after 20 and 200 epochs) and slices at $x_2=0$ for Example \ref{exam1} (ii), with $b_{t}=60$ (top), $40$ (middle), $20$ (below).}
\end{figure}

In case (iii), we test the generalization ability of the trained neural network. We fix $b=40$ and keep the distributions of $x_0$ and $a$ unchanged as in case (ii), and retrain the neural network. The test data is now given by $\delta_{t}(x)=\sum_{i=1}^na_ie^{-\frac{40|x-x_i|^2}{2}},$
with $n=2$ or 3, and $a_i\overset{\text{i.i.d.}}{\sim} U(0.2,1.2)$. In other words, the model is trained exclusively on the data from a single Gaussian but is evaluated on data drawn from the mixture of two or three  Gaussians. The relevant test results in Fig.\ \ref{fig:exam1.3} show that the reconstruction performance remains quite remarkable, which indicates that the deep Calder\'on method learns the inverse of the operator $\gamma\rightarrow\Lambda_\gamma$ rather than merely performing data fitting. Additionally, the trained neural network can correct some anisotropy, and can compensate for the errors due to the use of the isotropic Calder\'on method to solve anisotropic problems.

\begin{figure}[t]
	\centering\setlength{\tabcolsep}{2pt}
	\begin{tabular}{cccc}
		\includegraphics[height=2.2cm]  {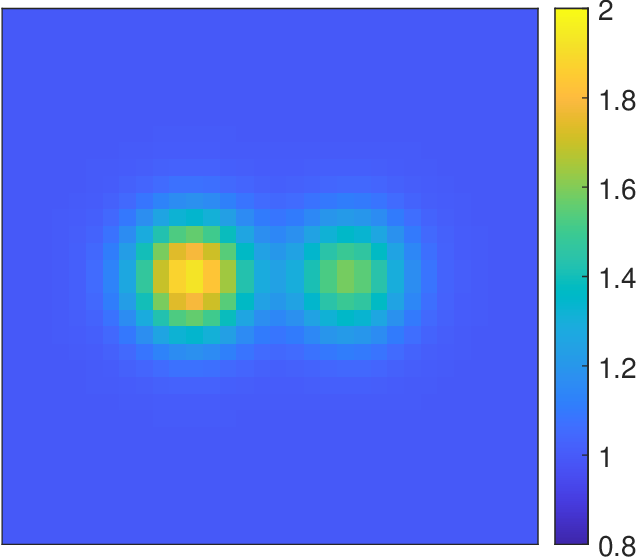} & \includegraphics[height=2.2cm]  {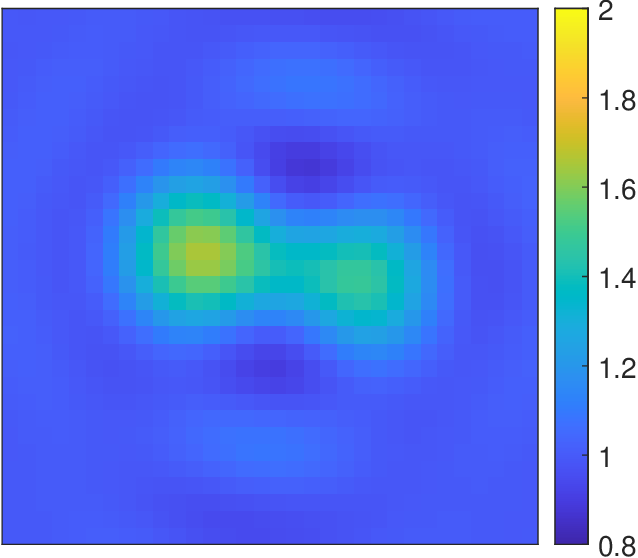} &
		\includegraphics[height=2.2cm]  {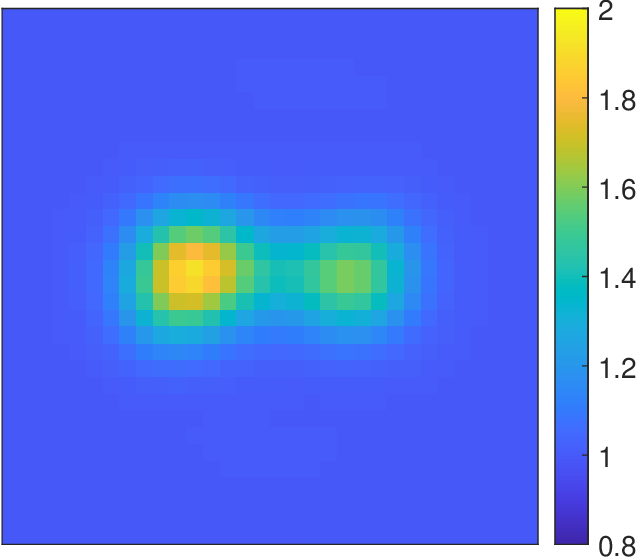} &
		\includegraphics[height=2.2cm]  {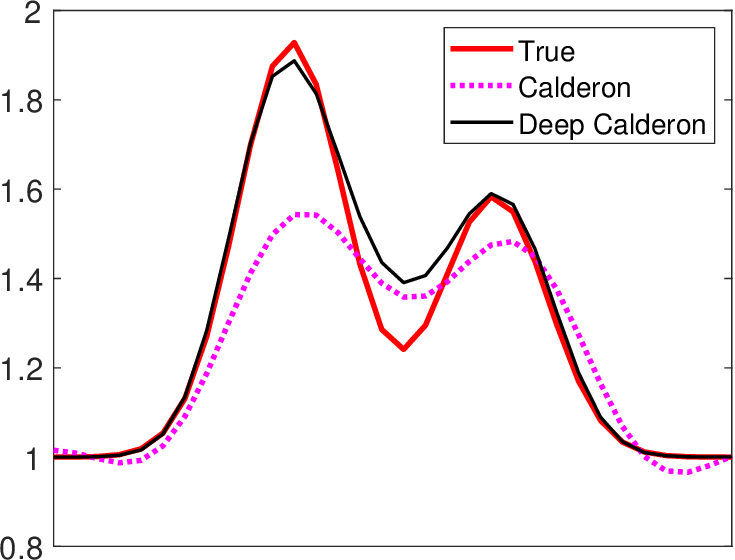} \\
        \includegraphics[height=2.2cm]  {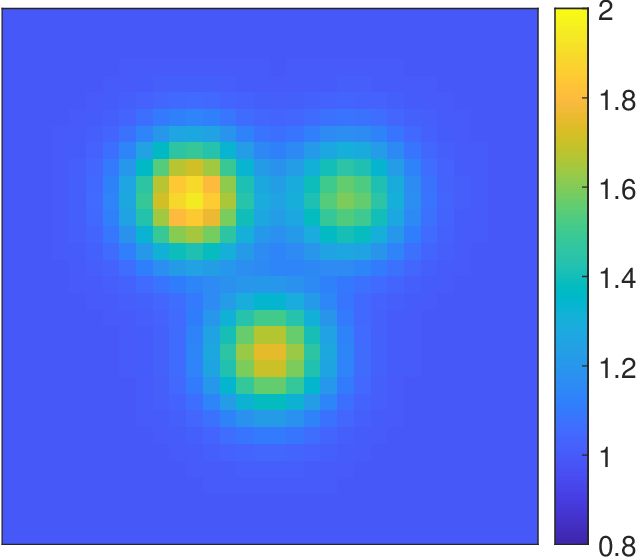} & \includegraphics[height=2.2cm]  {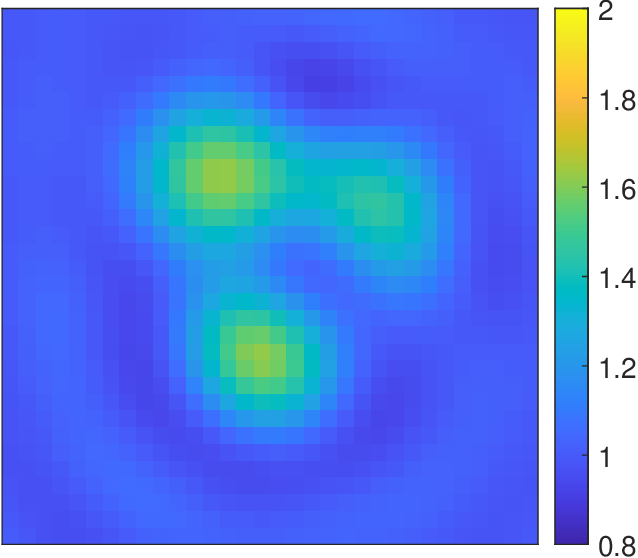} &
		\includegraphics[height=2.2cm]  {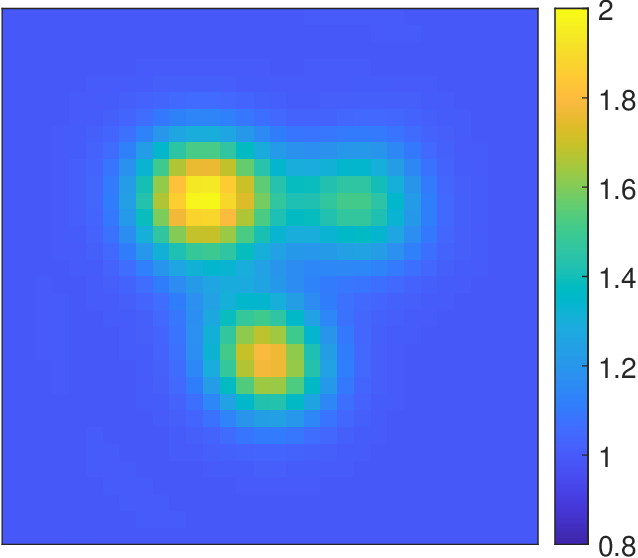} &
		\includegraphics[height=2.2cm]  {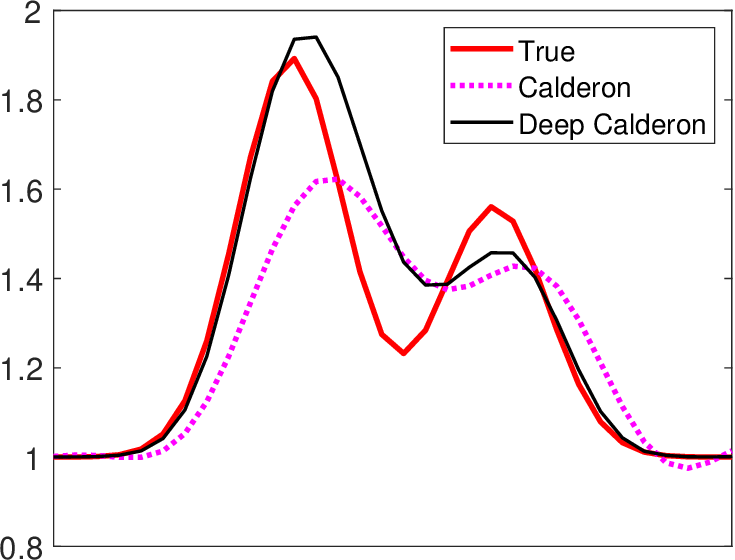}\\
        (a) true & (b) Calder\'on & (c) Deep Calder\'on  & (d) slice
	\end{tabular}
	\caption{\label{fig:exam1.3} The true image, Calder\'on reconstruction, deep reconstructions and slice at $x_2=0$ (2 Gaussian case) and $x_2\approx0.3$ (3 Gaussian case) for Example \ref{exam1} (iii) with the centers at: $(-0.3,0)$ and $(0.3,0)$ (top) and $(-0.3,0)$, $(0.3,0)$ and $(0,-0.3)$ (bottom).}
\end{figure}

Finally, we show the dynamics of the loss for the three cases in Fig.\ \ref{fig:exam1loss}. The result shows that as the training tasks become easier, the value of the loss function decreases faster accordingly, and the fluctuation in the test error also diminishes. However, the final values of the loss are close in the three cases, around $10^{-5}$. This finding aligns well with the preceding observations.

\begin{figure}[t]
	\centering\setlength{\tabcolsep}{2pt}
	\begin{tabular}{ccc}
		\includegraphics[height=3.7cm]  {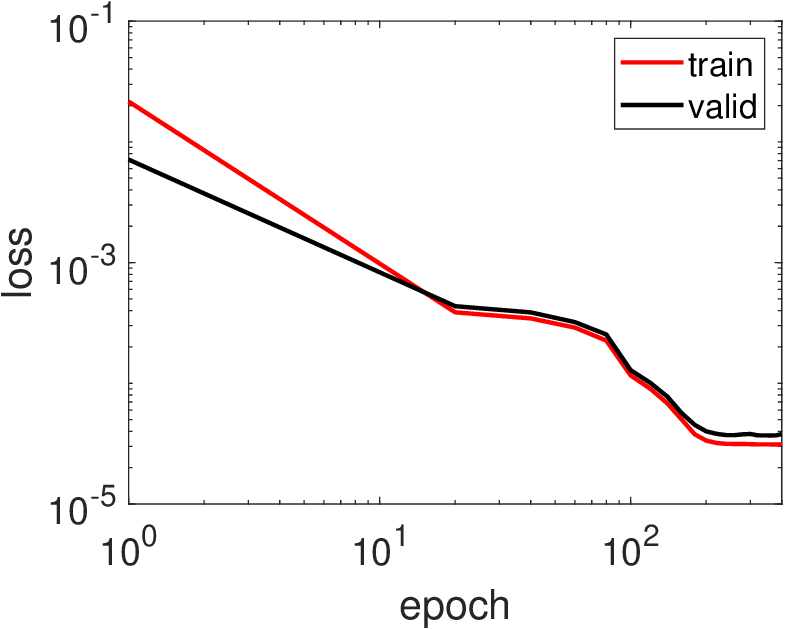} & \includegraphics[height=3.7cm]  {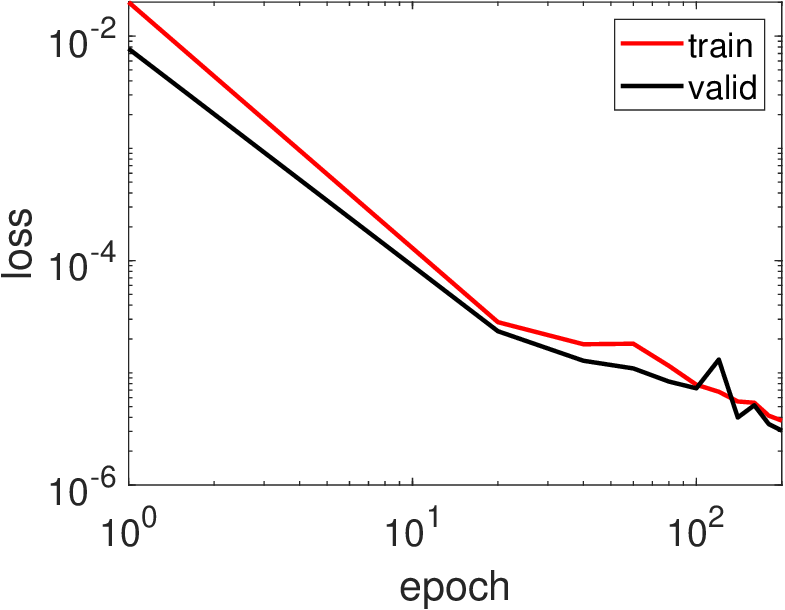} &
		\includegraphics[height=3.7cm]  {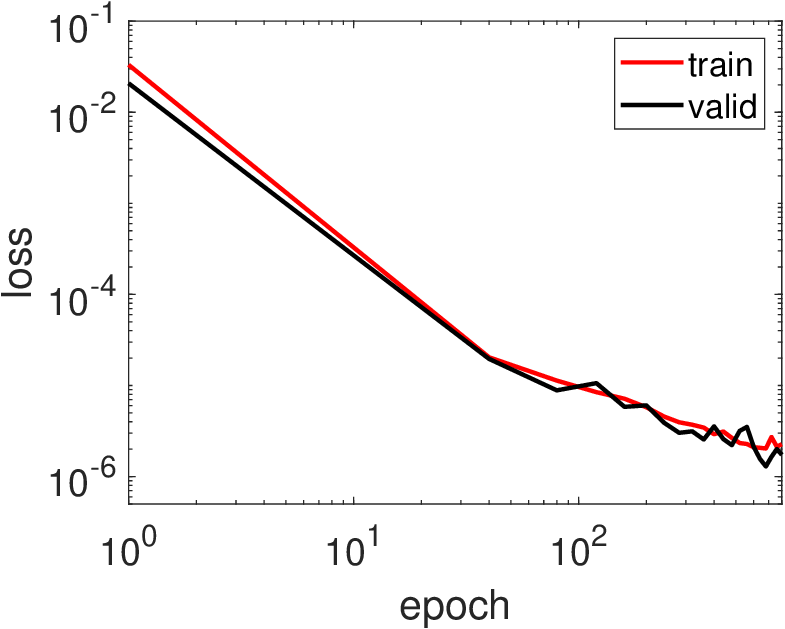} \\
        (a) (i) & (b) (ii) & (c) (iii)
	\end{tabular}
	\caption{\label{fig:exam1loss} The dynamics of the training error and validation loss for Example \ref{exam1}.}
\end{figure}

\begin{example}\label{exam2}
Consider conductivity perturbations with one characteristic function which is a ball centered at the origin. This example is taken from \cite{Mur}. Select the conformal factor $A$ and the discontinuous isotropic conductivity $\delta(x)$ as
\begin{equation*} 
A=\begin{pmatrix}
0.3&0\\0&0
\end{pmatrix}\quad \mbox{and}\quad \delta(x)=\alpha\chi_{\{x:|x-x_0|<\beta\}},
\end{equation*}
where $\chi_S$ denotes the characteristic function of the set $S$. To validate the effectiveness of the method, we use three cases and train the model for each case separately:
\begin{itemize}
    \item[{\rm(i)}] $\alpha\sim U(-0.5,0.5)$, $\beta\in(0.2,0.4)$ and $x_0\sim U(-0.5,0.5)^2$;
    \item[{\rm(ii)}] $\alpha\sim U(0.2,1.2)$, $\beta\in(0.2,0.4)$ and $x_0\sim U(-0.5,0.5)^2$;
    \item[{\rm(iii)}] $\alpha\sim U(0.2,1.2)$, $\beta=0.2$ and $x_0\sim U(-0.5,0.5)^2$.
\end{itemize}
Similar to Example \ref{exam1}, we focus on stability questions in cases {\rm(i)} and {\rm(ii)}. Case {\rm(iii)} is designed to test generalization ability. For each case, we show the neural network reconstruction at different stages of training on the test data with 
$\delta_t(x)=\alpha_{t}\chi_{\{x:|x-x_0|<\beta_{t}\}}$, with $\beta_{t}\in\{0.3,0.5,0.7\}$.
\end{example}

Like in Example \ref{exam1}, for each of the settings (i)--(iii), we first generate 2,300 training data points, then use 2,000 of them for training the U-net and the remaining 300 pairs for validation. The truncation radius $R$ of the Calder\'on method and training details are summarized in Table \ref{table:exam2}.

\begin{table}[hbt!]
\centering
\begin{threeparttable}
\caption{\label{table:exam2} The truncation radius $R$ for the Calder\'on method and the training details for Example~\ref{exam2}.}
\centering
\begin{tabular}[5pt]{c|c|c|c|c|c}
\toprule
Case & $R$ &training data & validation data &  learning rate & training time (min)\\ 
\midrule
(i) & 1.2 & 2,000 & 300 & $5\times10^{-4}$ & 21\\
(ii) & 1.2 & 2,000 & 300 & $1\times10^{-3}$ & 11\\
(iii) & 1.2 & 2,000 & 300 & $1\times10^{-3}$ & 18\\
\bottomrule
\end{tabular}
\end{threeparttable}
\end{table}

In case (i), we show the neural network reconstructions at different stages of training in Fig.\ \ref{fig:exam2.1}. 
First, we show the results of the neural network at the 100th epoch  in Fig.\ \ref{fig:exam2.1} (column 3). The plot and the slice show clear oscillations for the U-net training results. This is related to the fact that the training data parameter $\alpha$ can be either positive or negative. This oscillatory phenomenon occurs regardless of whether the support of $\delta$ is large or small, and persists throughout almost the entire region. In sharp contrast, if $\alpha$ contains only positive values as in case (ii), then such a phenomenon does not happen. When the training continues to 200 epochs, the oscillation phenomenon has eased to some extent, but still persists. When reaching 800 epochs, the training has nearly converged: for data within the training set, the method can provide good predictions, while for radius $b$ outside the training set, the reconstruction is still good for $\beta_t = 0.3$ (small), but becomes noticeably worse for $\beta_t = 0.7$ (large). These findings agree well with the analysis in Section \ref{sec:trainingset}. 
The oscillations at this stage have become very slight, but still exist. Since this involves a non-smooth function, the overall training process is more complex than in the Gaussian case shown in Example \ref{exam1} and requires more epochs to reach convergence. Moreover, the instability during training becomes more pronounced. We  present the relative $L^1(\Omega)$ and $L^2(\Omega)$ errors in Table \ref{table:exam1.1}. So far the test samples are only mildly beyond the training set. To further demonstrate the extent of generalization capability of the trained neural network, we also conduct one test with samples far outside the training set, i.e., $(\beta_{t},\alpha_t)=(0.8,1.05)$, for which the results are also worse; see the last row of Fig. \ref{fig:exam2.1}. This is also confirmed by the quantitative results presented in Table \ref{table:exam1.1}: the relative errors of the reconstructions are substantially larger.

\begin{figure}[htbp]
	\centering\setlength{\tabcolsep}{2pt}
	\begin{tabular}{cccccc}
		\includegraphics[height=1.9cm]  {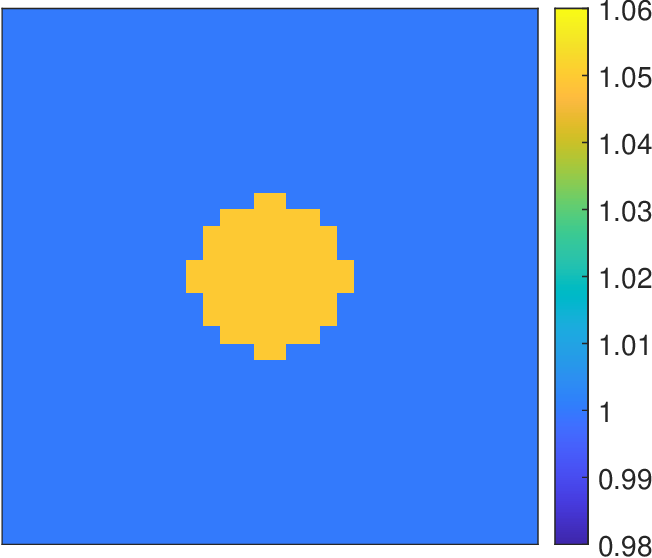} & \includegraphics[height=1.9cm]  {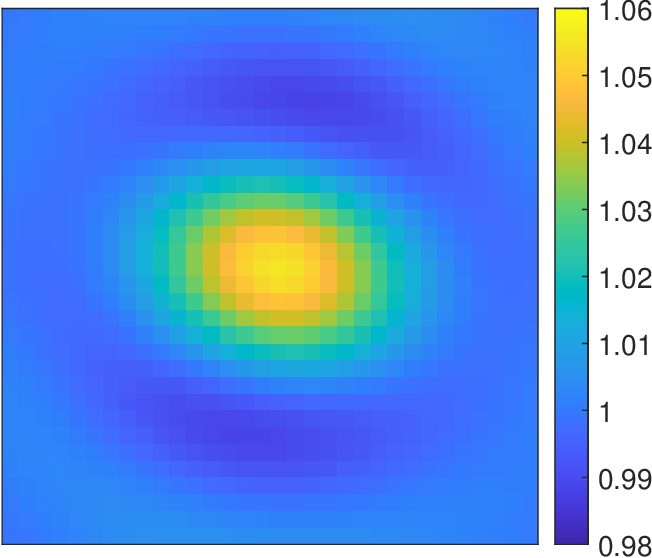} &
		\includegraphics[height=1.9cm]  {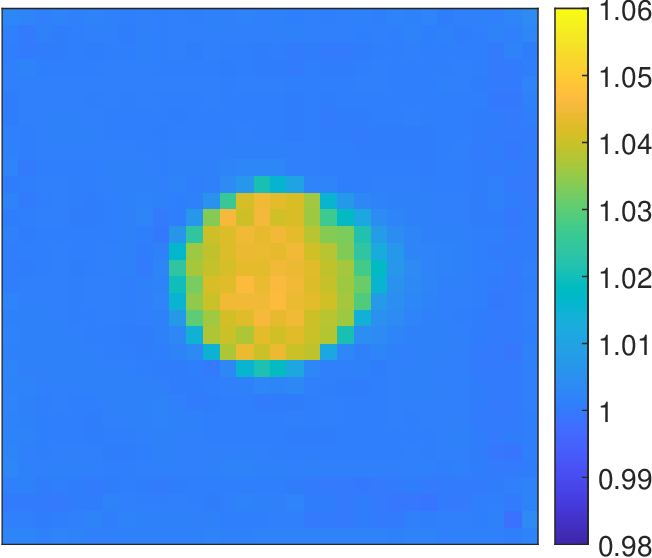} &
        \includegraphics[height=1.9cm]  {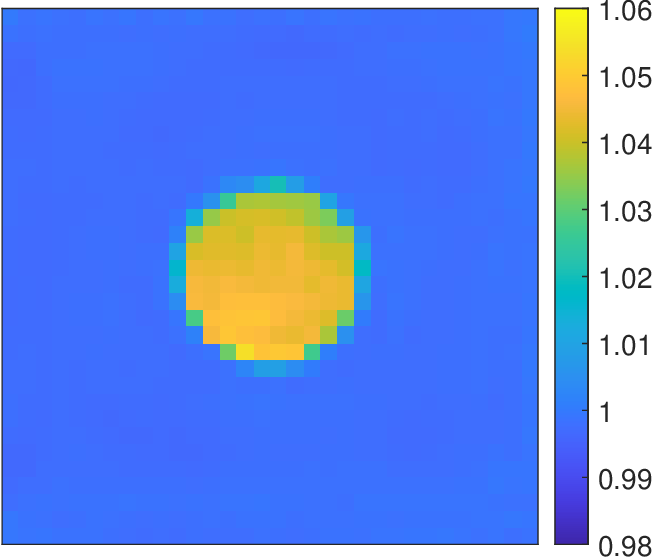} &
        \includegraphics[height=1.9cm]  {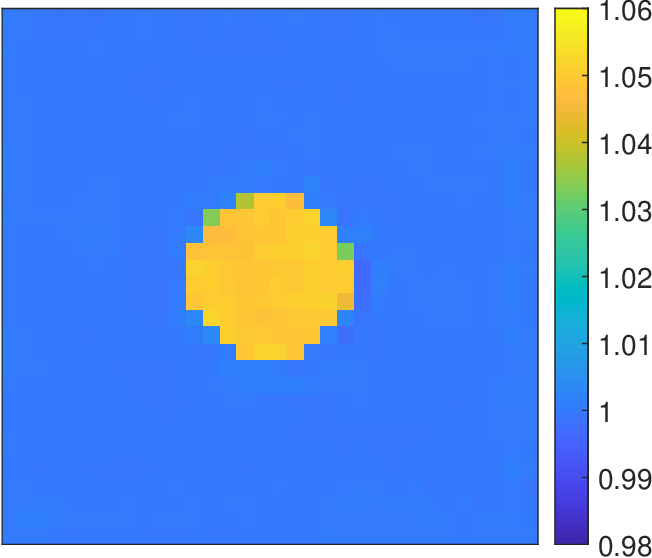} &
		\includegraphics[height=1.9cm]  {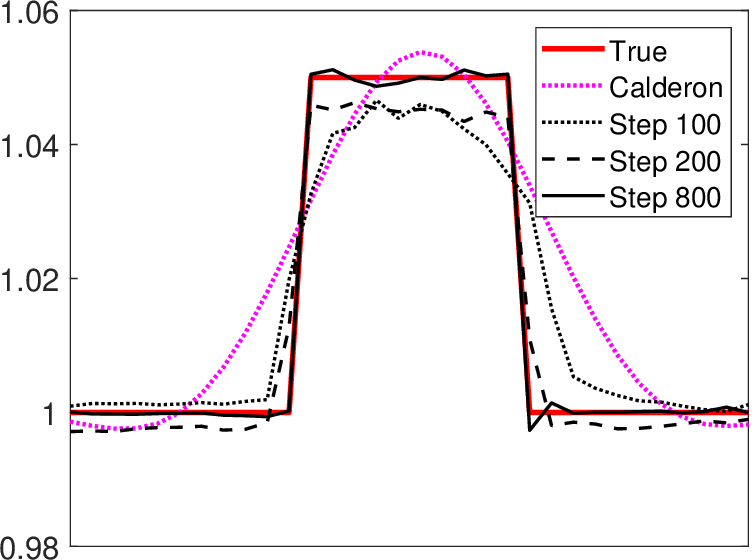} \\
		\includegraphics[height=1.9cm]  {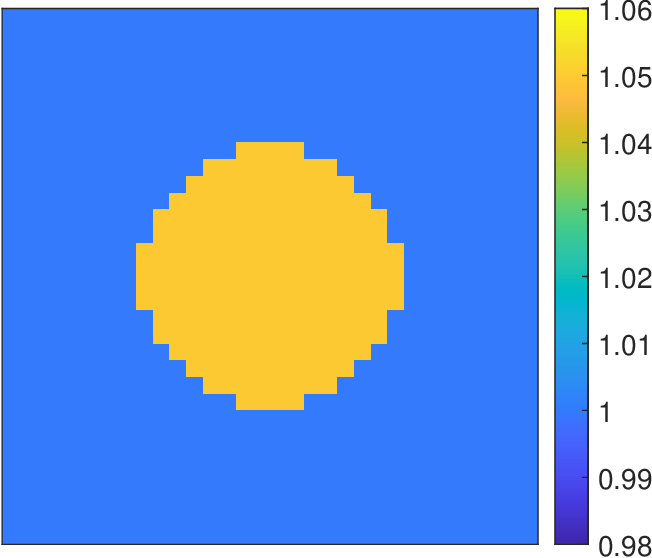} & \includegraphics[height=1.9cm]  {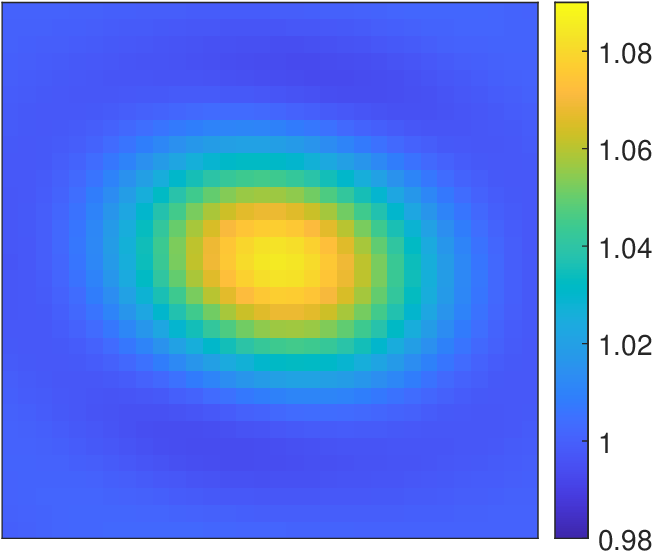} &
		\includegraphics[height=1.9cm]  {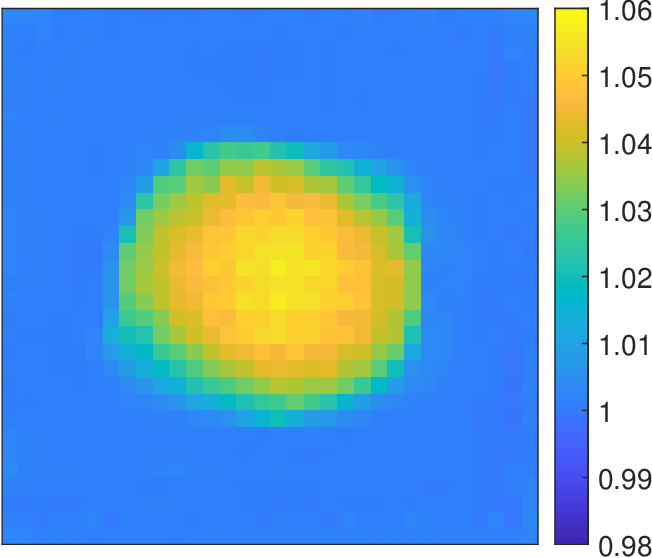} &
        \includegraphics[height=1.9cm]  {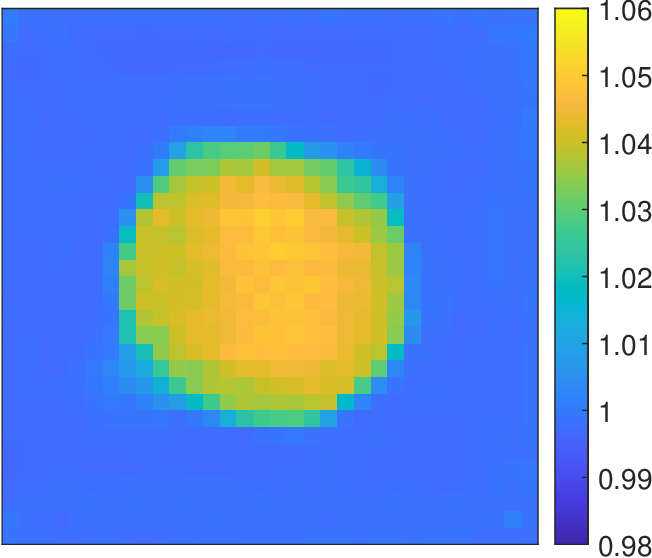} &
        \includegraphics[height=1.9cm]  {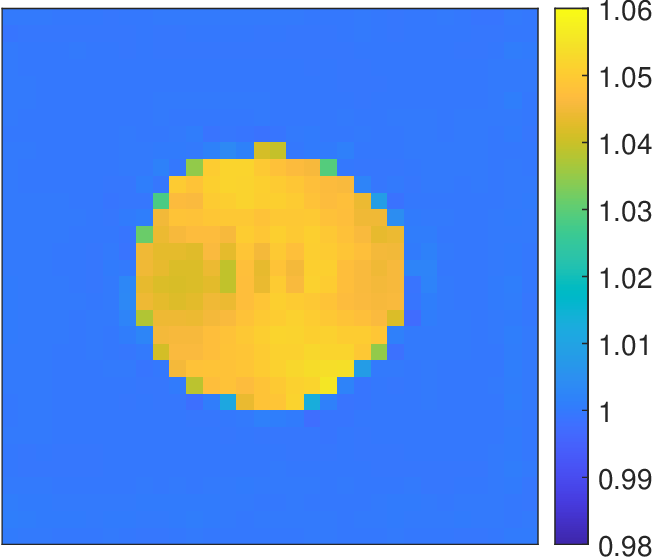} &
		\includegraphics[height=1.9cm]  {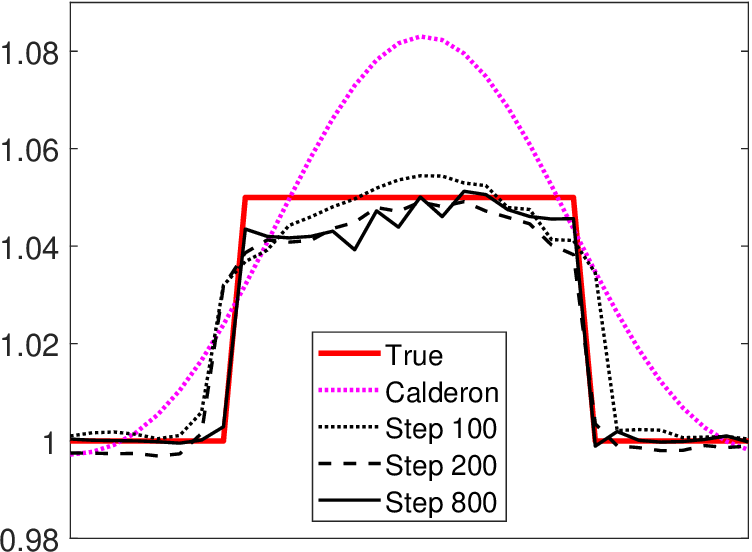}\\
		\includegraphics[height=1.9cm]  {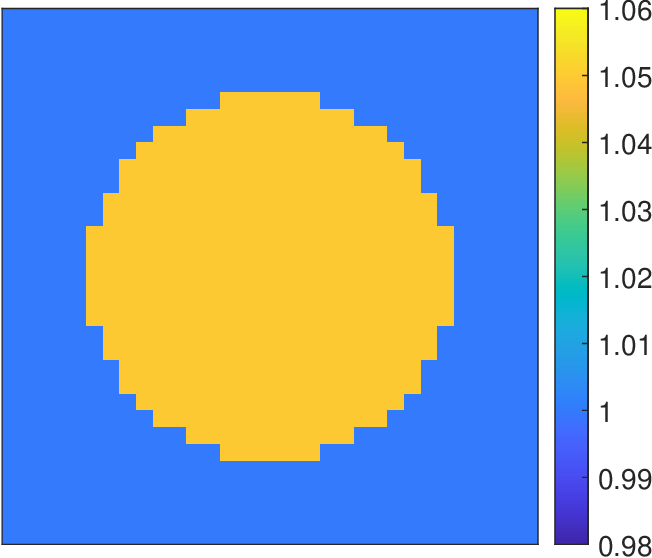} & \includegraphics[height=1.9cm]  {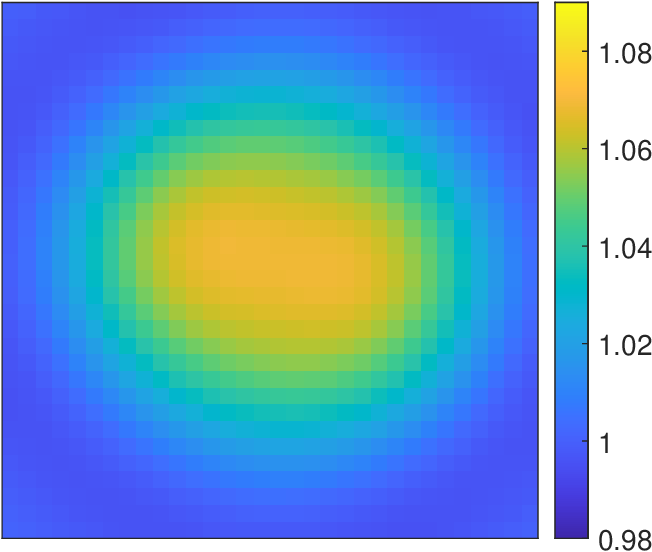} &
		\includegraphics[height=1.9cm]  {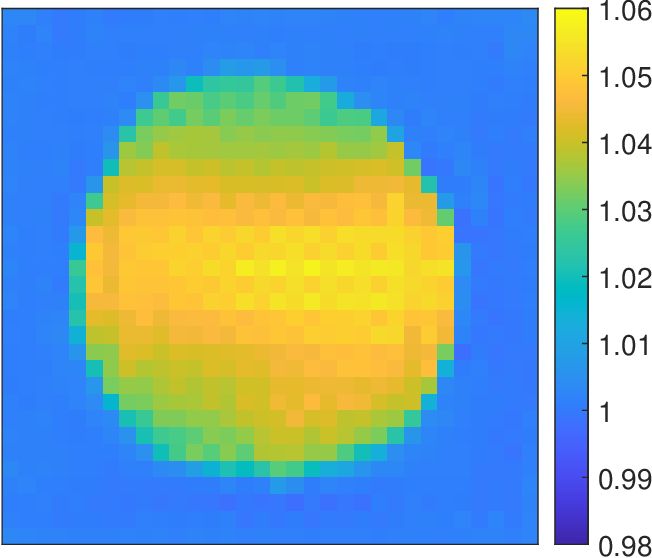} &
        \includegraphics[height=1.9cm]  {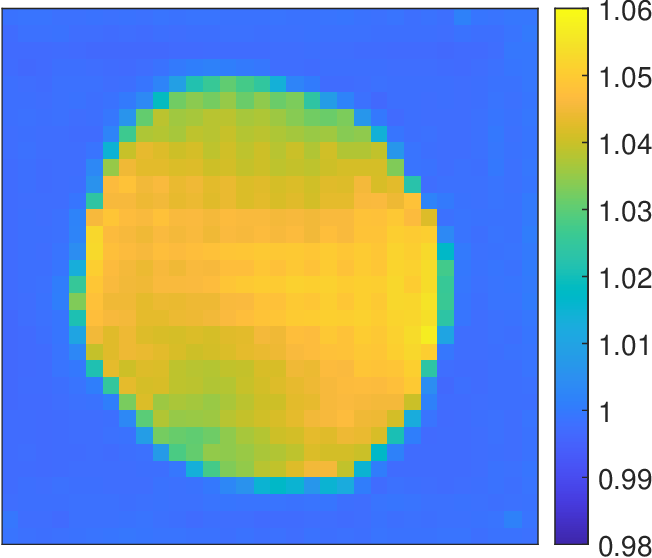} &
        \includegraphics[height=1.9cm]  {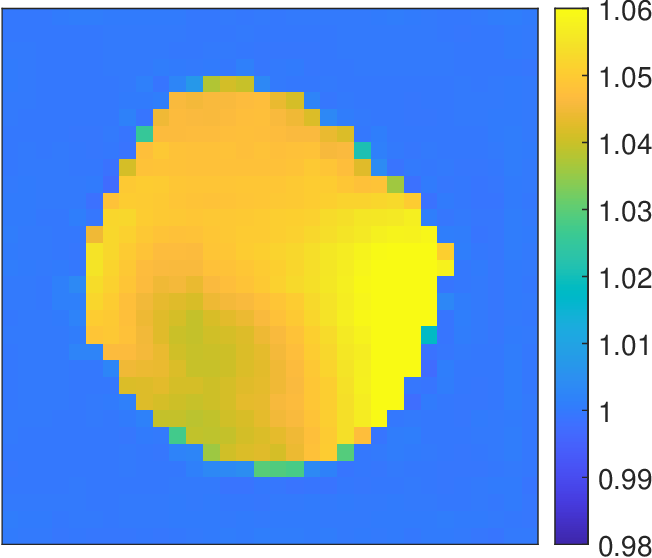} &
		\includegraphics[height=1.9cm]  {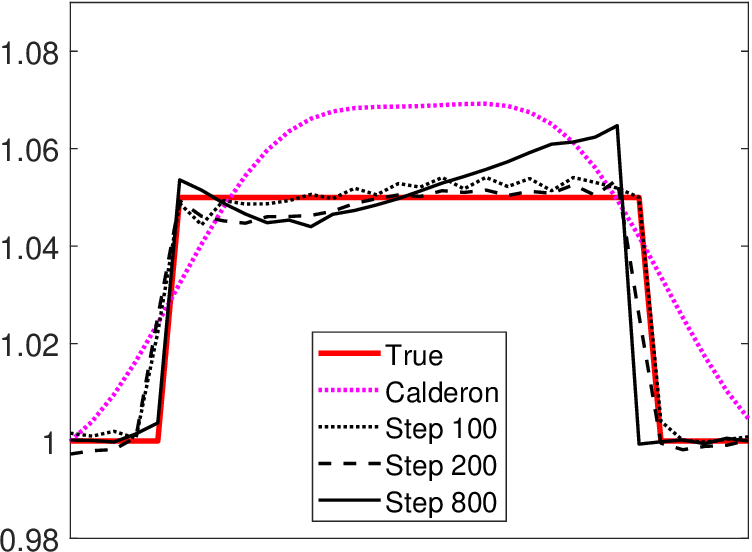}\\
		\includegraphics[height=1.9cm]  {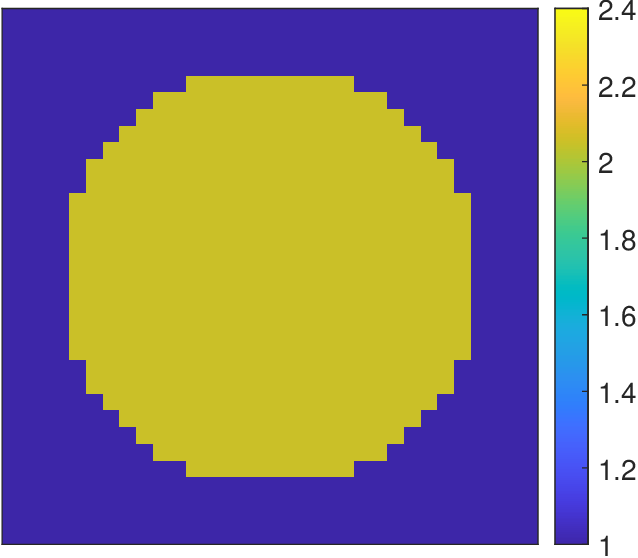} & \includegraphics[height=1.9cm]  {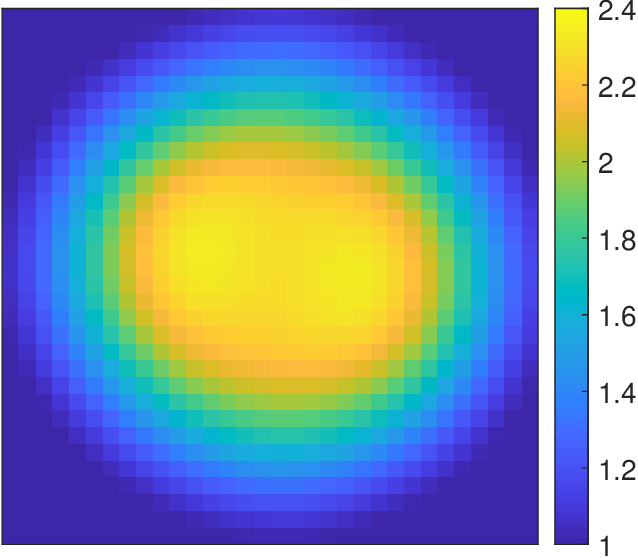} &
		\includegraphics[height=1.9cm]  {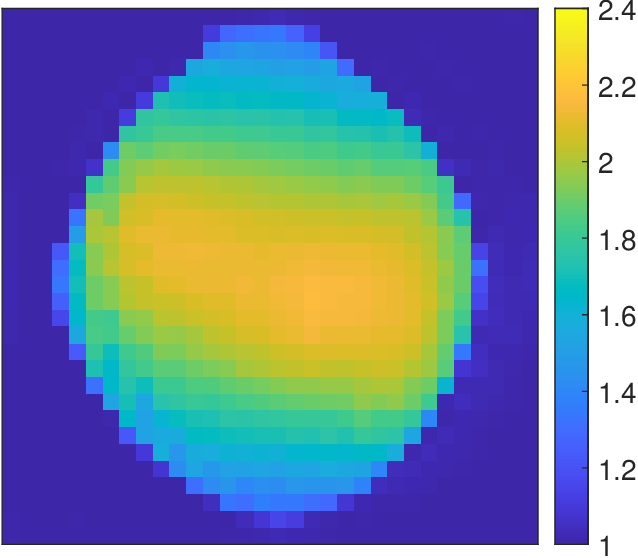} &
        \includegraphics[height=1.9cm]  {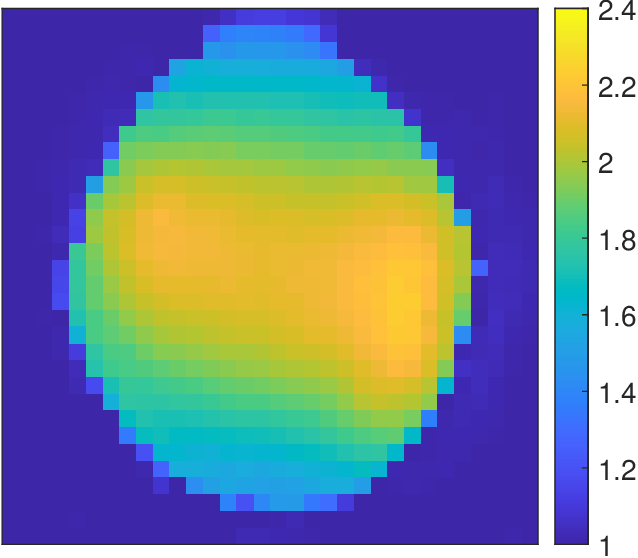} &
        \includegraphics[height=1.9cm]  {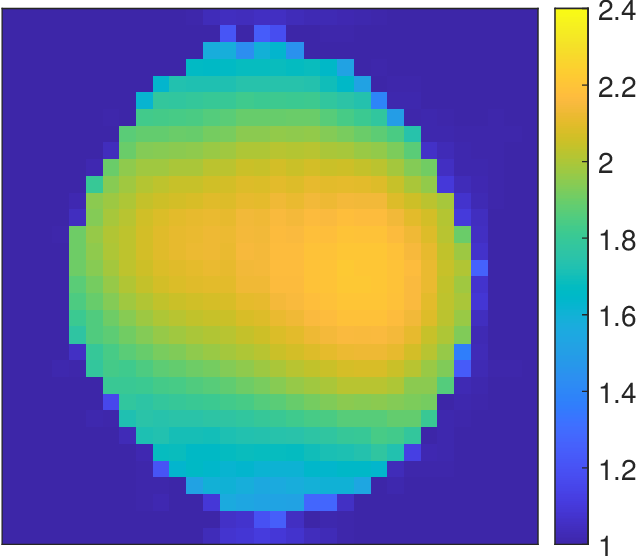} &
		\includegraphics[height=1.9cm]  {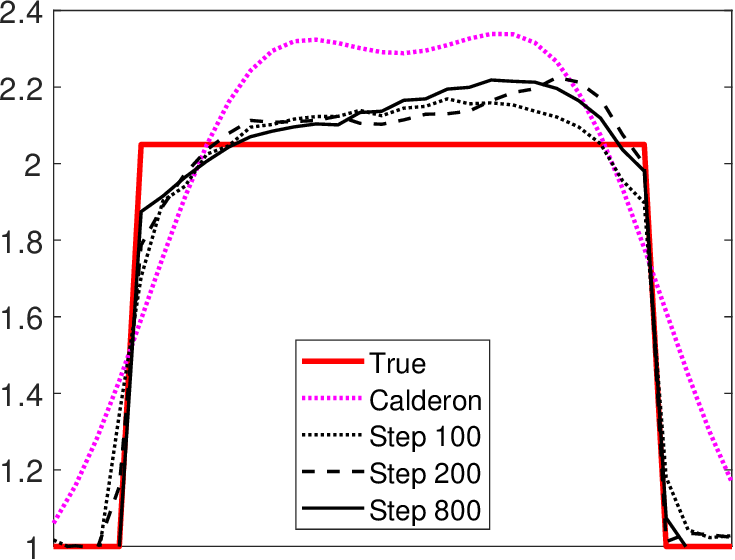} \\
        (a) true & (b) Calder\'on & (c) step 100 & (d) step 200 & (e) step 800 & (f) slice
	\end{tabular}
	\caption{\label{fig:exam2.1} The true image, Calder\'on reconstruction, deep  reconstructions (after 100, 200 and 800 epochs) and slice at $x_2=0$ for Example \ref{exam2}(i), with the settings from the top to the bottom being $(\beta_{t},\alpha_t)=(0.3,0.05)$, $(\beta_t,\alpha_t)=(0.5,0.05)$, $(\beta_t,\alpha_t)=(0.7,0.05)$, and $(\beta_t,\alpha_t)=(0.8,1.05)$.}
\end{figure}

\begin{comment}
\begin{figure}[htbp]
	\centering\setlength{\tabcolsep}{2pt}
	\begin{tabular}{cccccc}
		\includegraphics[height=1.9cm]  {ex3true4.eps} & \includegraphics[height=1.9cm]  {ex3input4.eps} &
		\includegraphics[height=1.9cm]  {ex3pred4it100.eps} &
        \includegraphics[height=1.9cm]  {ex3pred4it200.eps} &
        \includegraphics[height=1.9cm]  {ex3pred4it800.eps} &
		\includegraphics[height=1.9cm]  {ex3cutoff4.eps} \\
        (a) true & (b) Calder\'on & (c) step 100 & (d) step 200 & (e) step 800 & (f) slice
	\end{tabular}
	\caption{\label{fig:exam2.1out} The true image, Calder\'on reconstruction, deep  reconstructions (after 100, 200 and 800 epochs) and slice at $x_2=0$ for Example \ref{exam2}, with $\beta_{t}=0.8$, $\alpha_t=1.05$.}
\end{figure}
\end{comment}

\begin{comment}
\begin{table}[hbt!]
\centering
\begin{threeparttable}
\caption{\label{table:exam2.1} The relative error of the reconstructions for for Example~\ref{exam2}(i).}
\centering
\begin{tabular}[5pt]{c|c|c}
\toprule
$(\beta_t,\alpha_t)$ & $L^1(\Omega)$ & $L^2(\Omega)$ \\ 
\midrule
$(0.3,0.05)$ & 6.77e-4 & 1.90e-3 \\
\hline
$(0.5,0.05)$ & 1.34e-3 & 3.22e-3 \\
\hline
$(0.7,0.05)$ & 3.58e-3 & 7.20e-3 \\
\hline
$(0.8,1.05)$ & 7.59e-2 & 1.21e-1 \\
\bottomrule
\end{tabular}
\end{threeparttable}
\end{table}
\end{comment}

In case (ii), the parameter settings, training set, and test set are identical with that for case (i), except for the setting of $\alpha$: $\alpha$ is always positive in case (ii). The test results after 100 epochs are shown in Fig.\ \ref{fig:exam2.2}. It is observed that the training results in this case are significantly more stable than those in Fig.\ \ref{fig:exam2.1} (for case (i)) under the same setting. Indeed, we have achieved convergence of the training after 400 epochs (versus 800 epochs in case (i)). It can be seen that the data within the training distribution is well reconstructed, whereas the reconstruction for the data outside the training distribution is relatively poor and deteriorates as the distance from the training set increases, which also agrees with the analysis in Section \ref{sec:trainingset}. Moreover, the training images exhibit little oscillations, and the training behavior is much more stable compared with case (i). This indicates that the stability properties of these two cases are not identical, and a detailed analysis of this aspect is a topic that we wish to explore in future work.

\begin{figure}[htbp]
	\centering\setlength{\tabcolsep}{2pt}
	\begin{tabular}{ccccc}
		\includegraphics[height=2cm]  {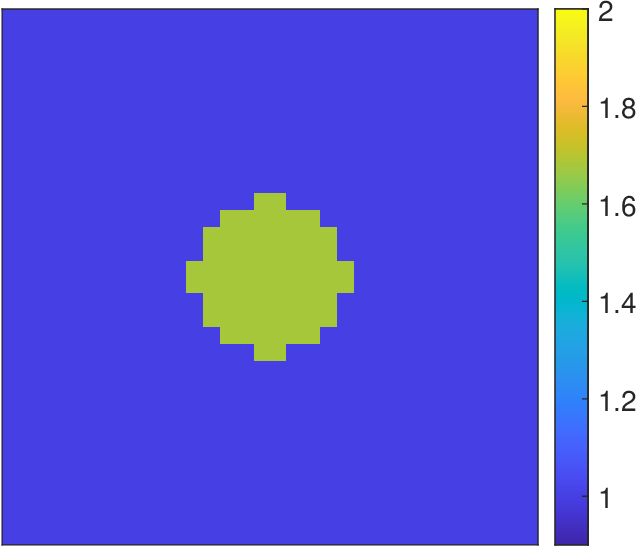} & \includegraphics[height=2cm]  {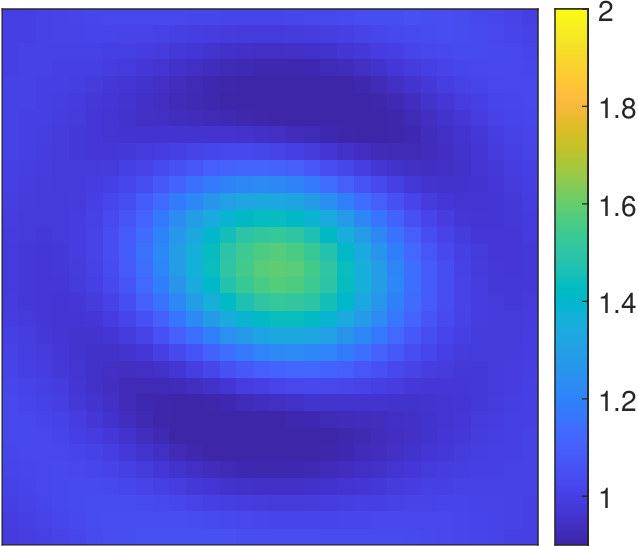} &
		\includegraphics[height=2cm]  {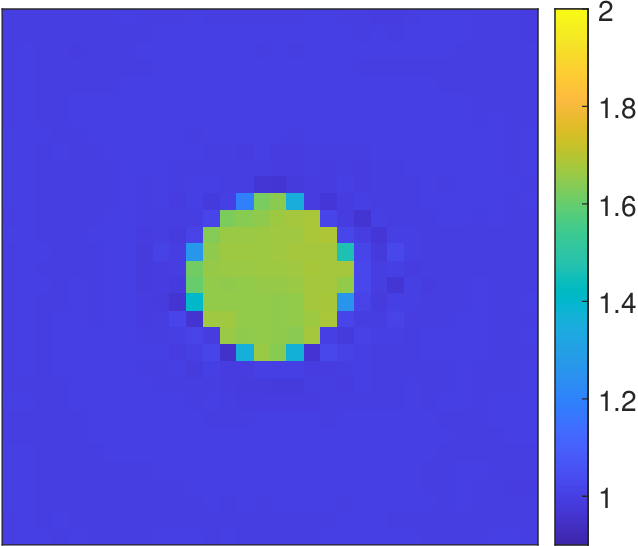} &
        \includegraphics[height=2cm]  {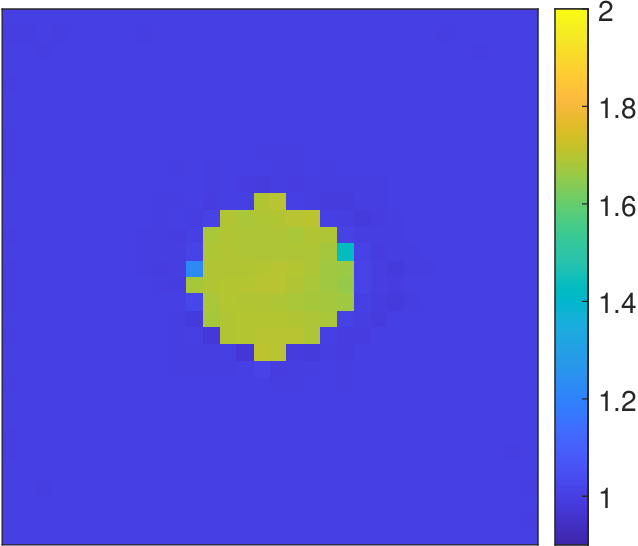} &
		\includegraphics[height=2cm]  {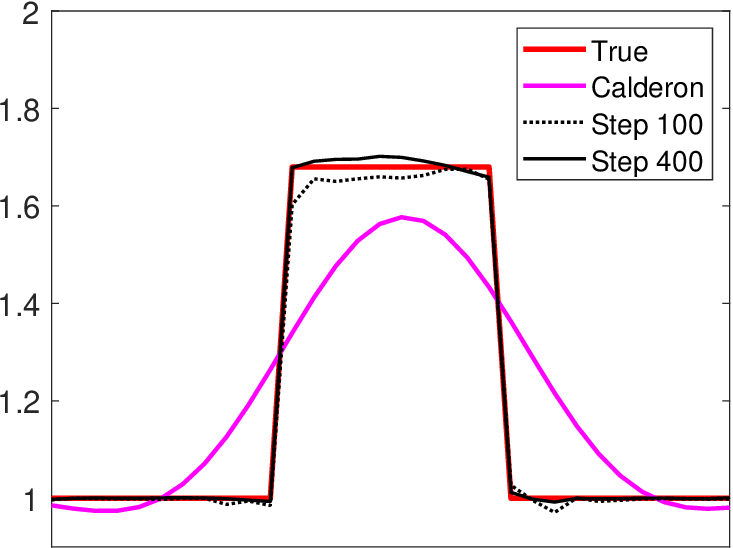} \\
		\includegraphics[height=2cm]  {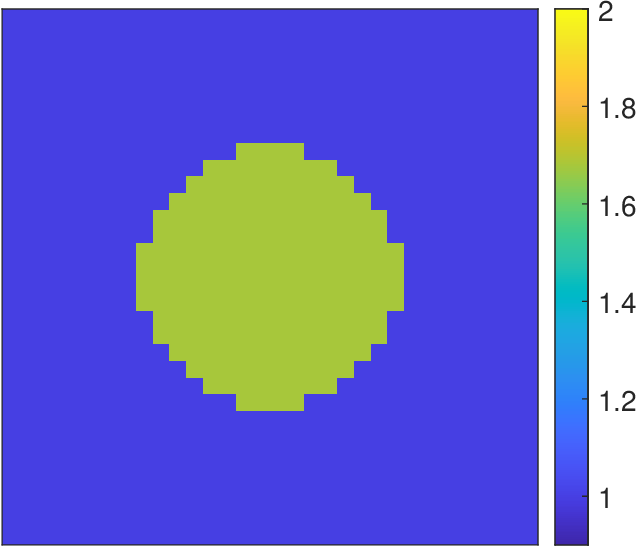} & \includegraphics[height=2cm]  {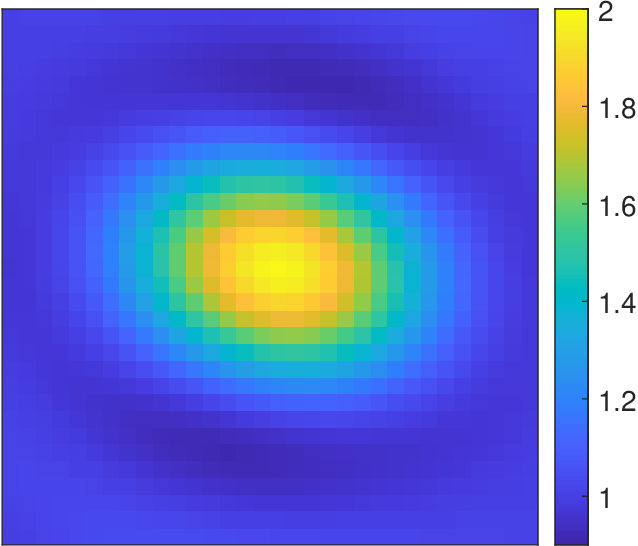} &
		\includegraphics[height=2cm]  {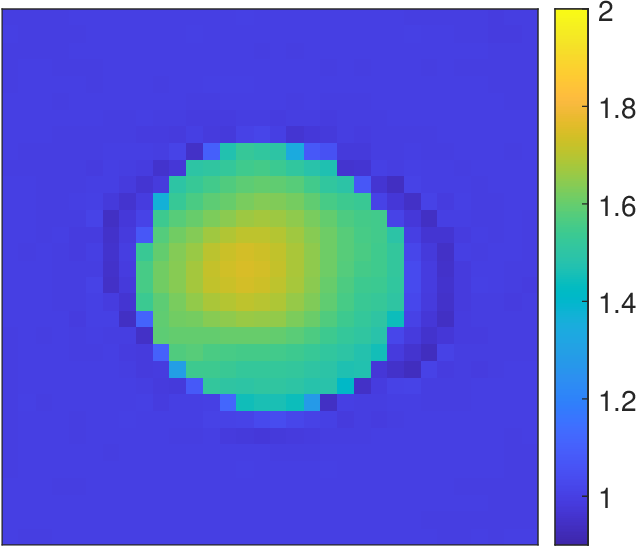} &
        \includegraphics[height=2cm]  {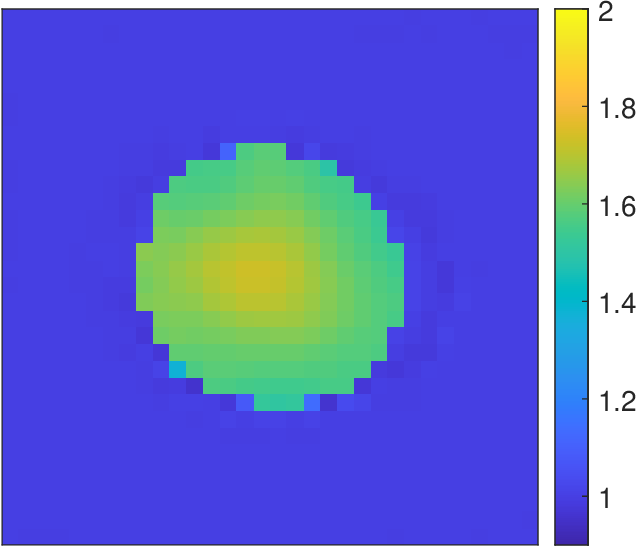} &
		\includegraphics[height=2cm]  {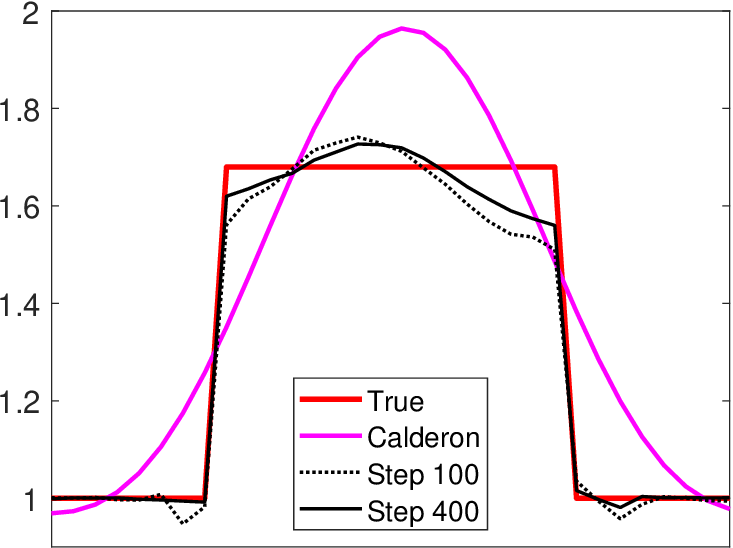}\\
		\includegraphics[height=2cm]  {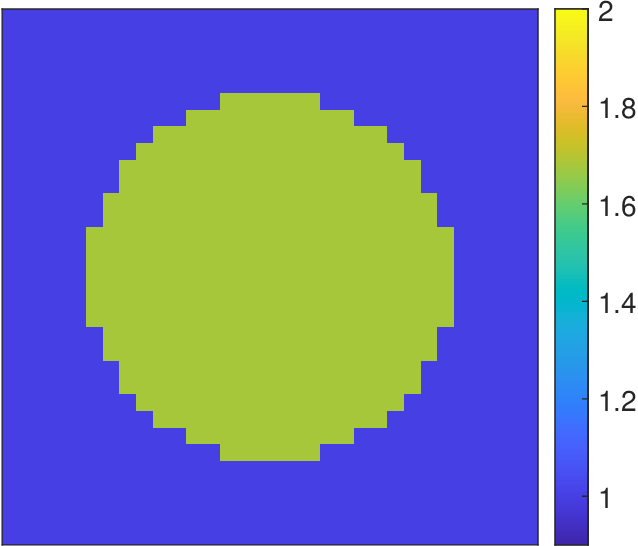} & \includegraphics[height=2cm]  {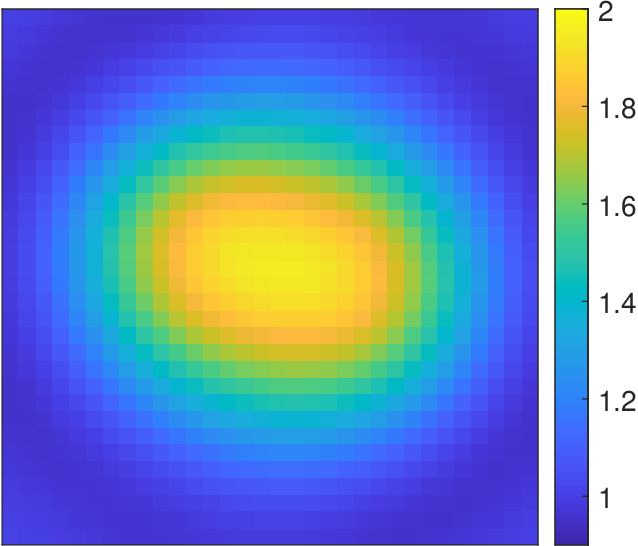} &
		\includegraphics[height=2cm]  {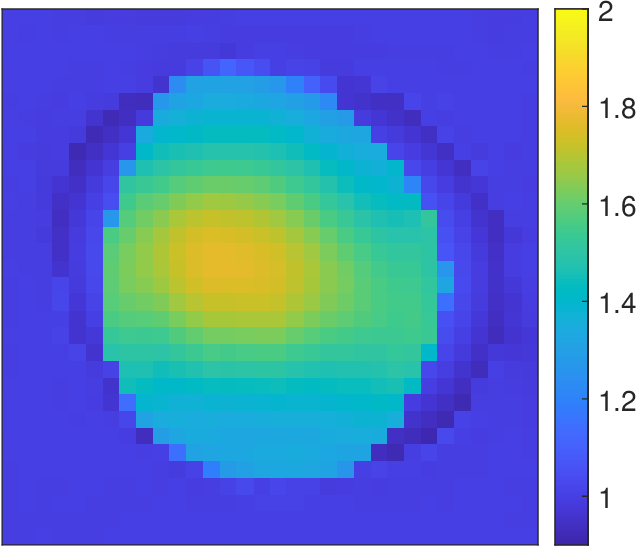} &
        \includegraphics[height=2cm]  {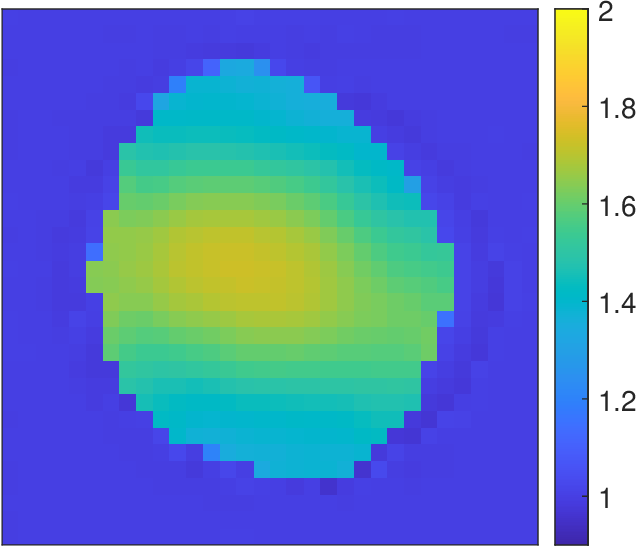} &
		\includegraphics[height=2cm]  {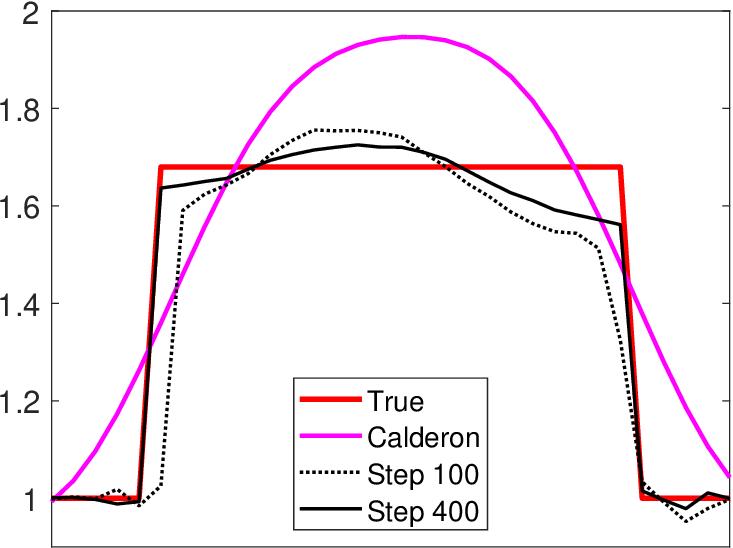}\\
        (a) true & (b) Calder\'on & (c) step 100 & (d) step 400 & (e) slice
	\end{tabular}
	\caption{\label{fig:exam2.2} The true image, Calder\'on reconstruction, deep reconstructions (after 100 and 400 epochs) and slice at $x_2=0$ for Example \ref{exam2} (ii), with $\beta_{t}=0.3$ (top), $0.5$ (middle), $0.7$ (below).}
\end{figure}

Case (iii) is to illustrate the generalization ability of the trained neural network. The test data involves multiple disks: 
$\delta_{t}(x)=\sum_{i=1}^n\alpha_i\chi_{\{x:|x-x_i|<0.2\}}$,
with $n=2$ or 3, and $\alpha_i\overset{\text{i.i.d.}}{\sim} U(0.2,1.2)$. The test results are shown in Fig.\ \ref{fig:exam2.3} with the information about the centers of the disks given in the caption. The current scenario differs somewhat from the Gaussian case. When the two disks are relatively close to each other, the deep Calder\'on method fails to resolve them. However, if the two disks are sufficiently far apart, the reconstruction can be successful. We hypothesize that when the two disks are close, their Fourier transforms might have cancellations so the new conductivity might not satisfy the stability condition in \eqref{eq-dn0}. When the disks are far apart, due to the decay of the Fourier transform, such cancellations are reduced. This point can also be supported by the non-negative Gaussians shown in Example \ref{exam1} (iii). There, the Fourier transforms are also Gaussians with fixed sign, so there are no significant cancellations. Nevertheless, note that the inversion of $\La_\gamma$ is highly nonlinear and that the choice of training data in this case is very specific. There might be other contributing factors to the generalizability. Regardless of these external factors, the trained neural network clearly enjoys considerable generalization capability, and can achieve successful reconstructions in both two- and three-disk scenarios when the disks are well-separated.

\begin{figure}[htbp]
	\centering\setlength{\tabcolsep}{2pt}
	\begin{tabular}{cccc}
		\includegraphics[height=2.2cm]  {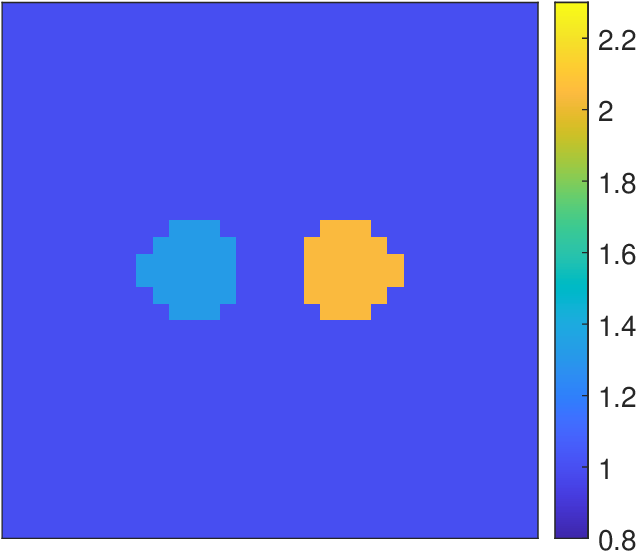} & \includegraphics[height=2.2cm]  {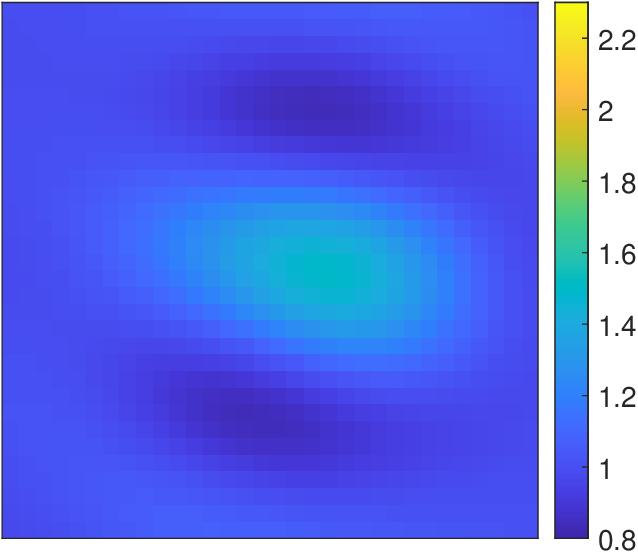} &
		\includegraphics[height=2.2cm]  {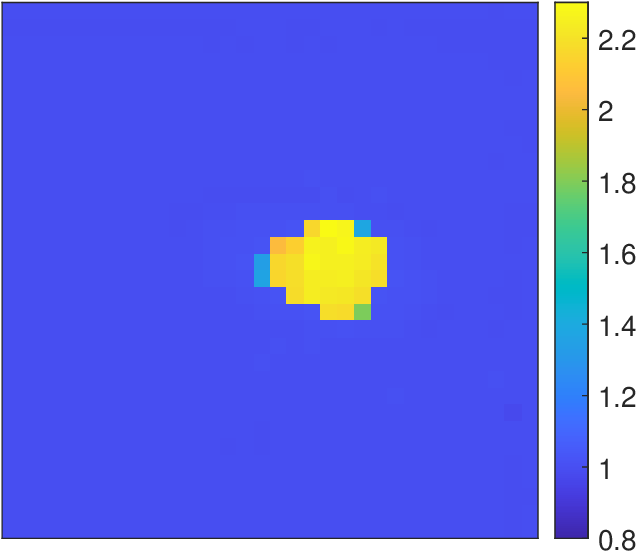} &
		\includegraphics[height=2.2cm]  {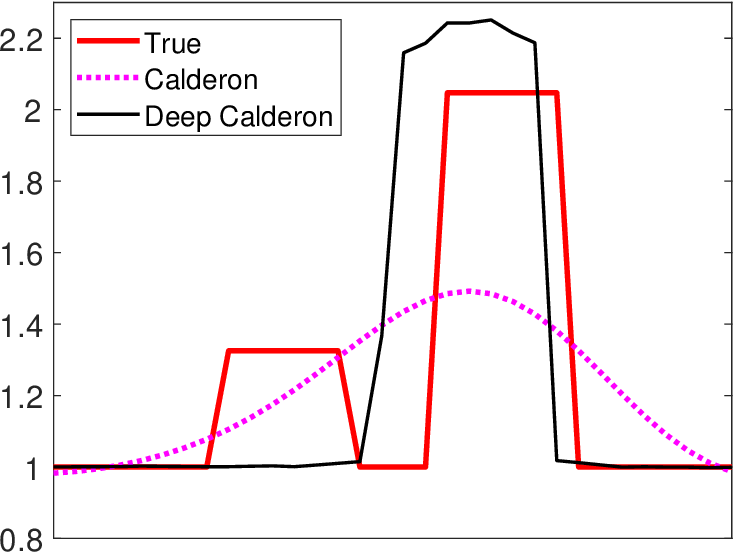} \\
		\includegraphics[height=2.2cm]  {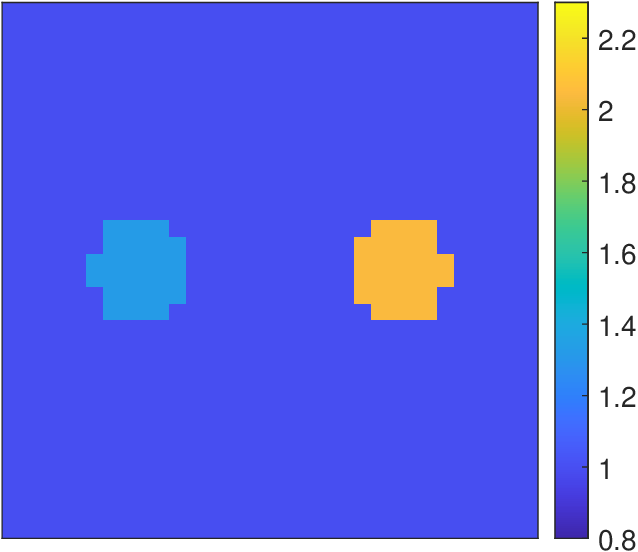} & \includegraphics[height=2.2cm]  {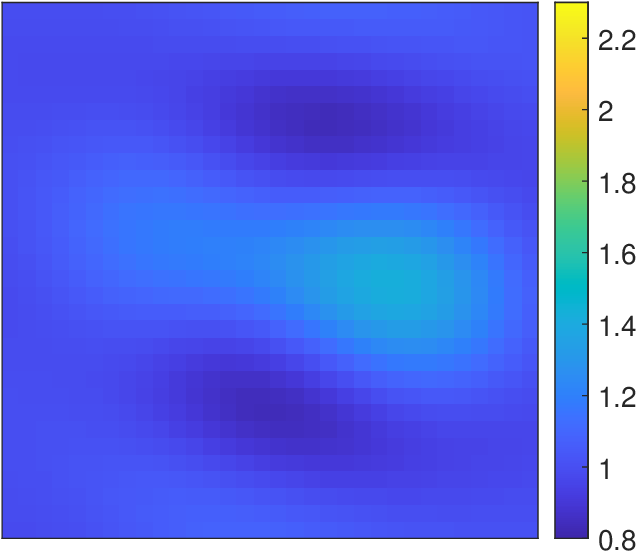} &
		\includegraphics[height=2.2cm]  {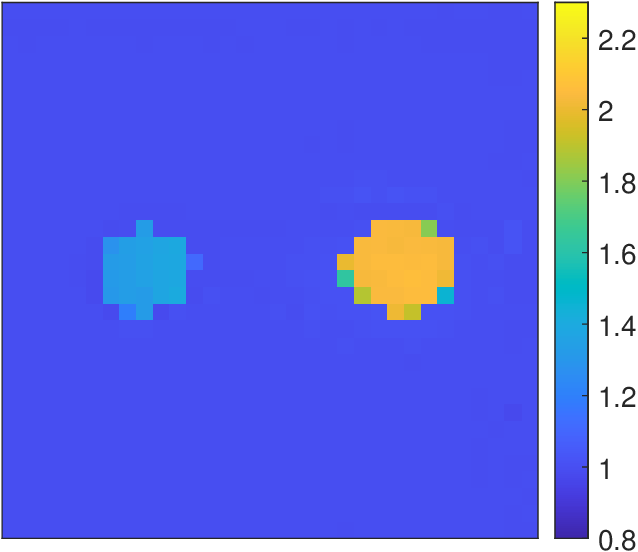} &
		\includegraphics[height=2.2cm]  {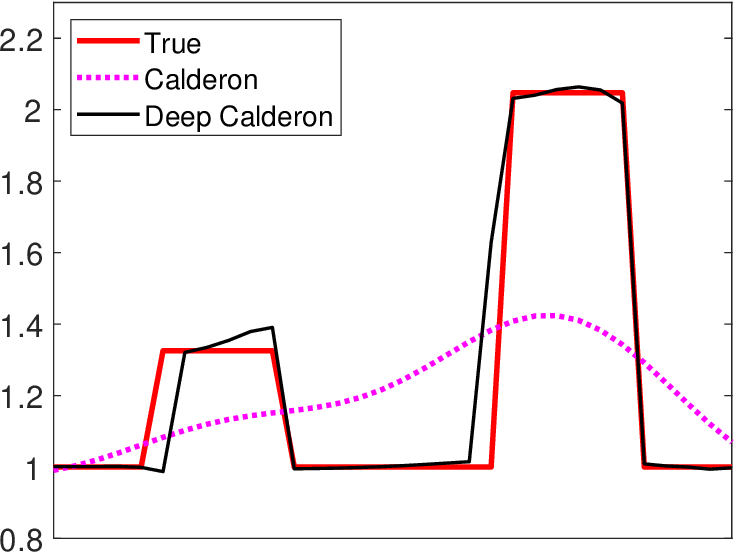}\\  \includegraphics[height=2.2cm]  {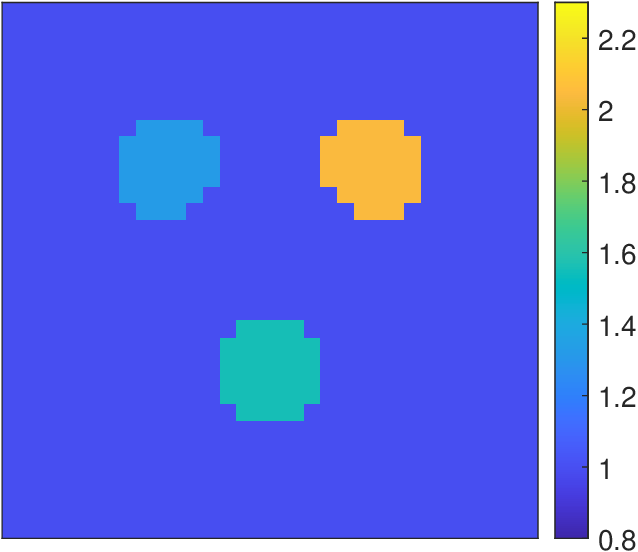} & \includegraphics[height=2.2cm]  {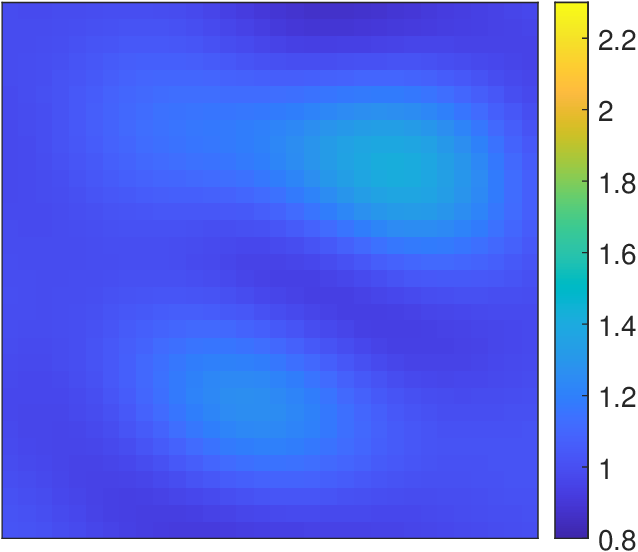} &
		\includegraphics[height=2.2cm]  {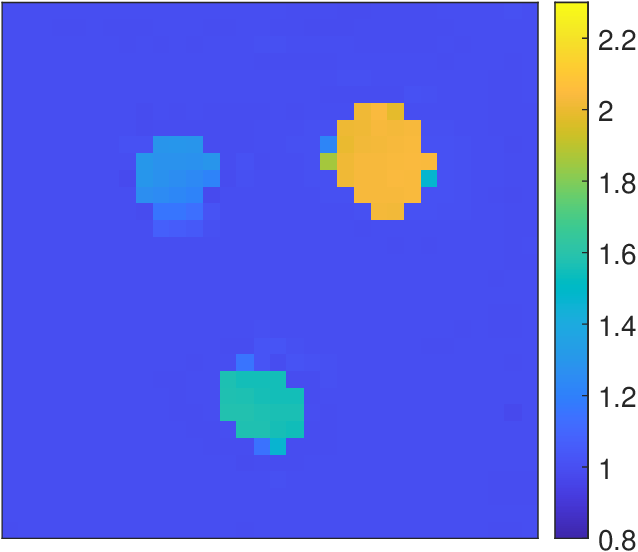} &
		\includegraphics[height=2.2cm]  {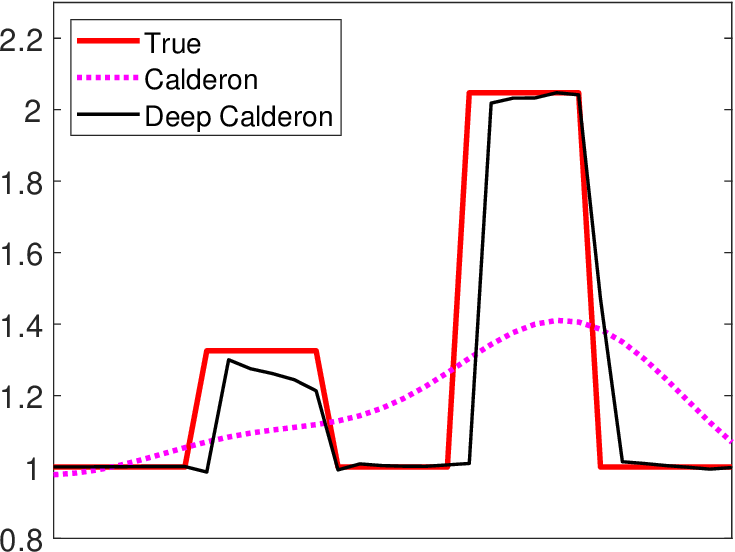}\\
         (a) true & (b) Calder\'on & (c) Deep Calder\'on  & (d) slice
	\end{tabular}
	\caption{\label{fig:exam2.3} The true image, Calder\'on reconstruction, deep reconstructions and slice at $x_2=0$ (2 disks case) and $x_2\approx0.3$ (3 disks case) for Example \ref{exam2} (iii) with multiple inclusions with the centers at: $(-0.3,0)$ and $(0.3,0)$ (top), $(-0.5,0)$ and $(0.5,0)$ (middle) and $(-0.4,0.4)$, $(0.4,0.4)$ and $(0,-0.4)$ (bottom).}
\end{figure}

Finally, we also show the dynamic changes of the loss function for the three cases in Fig.\ \ref{fig:exam2loss}. Note that the training loss decreases steadily throughout the epochs.
\begin{figure}[htbp]
	\centering\setlength{\tabcolsep}{2pt}
	\begin{tabular}{ccc}
		\includegraphics[height=3.7cm]  {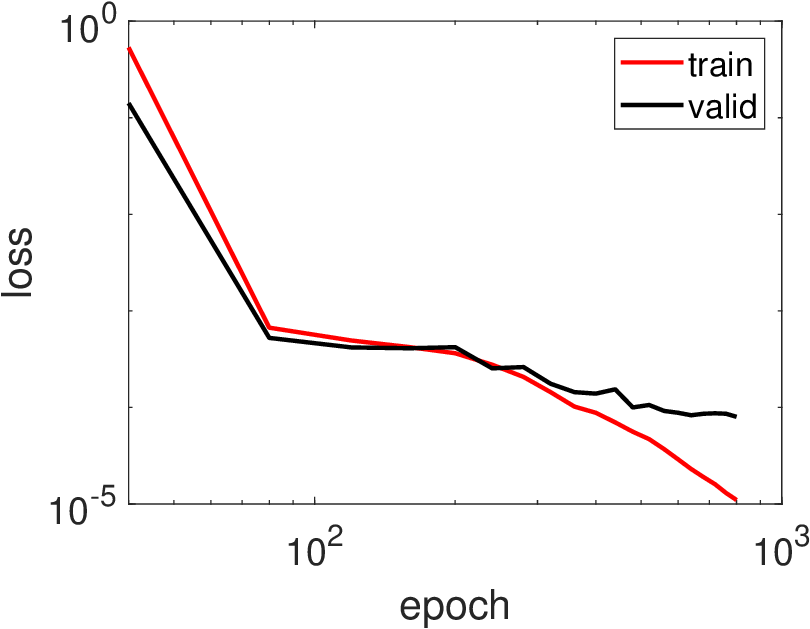} & \includegraphics[height=3.7cm]  {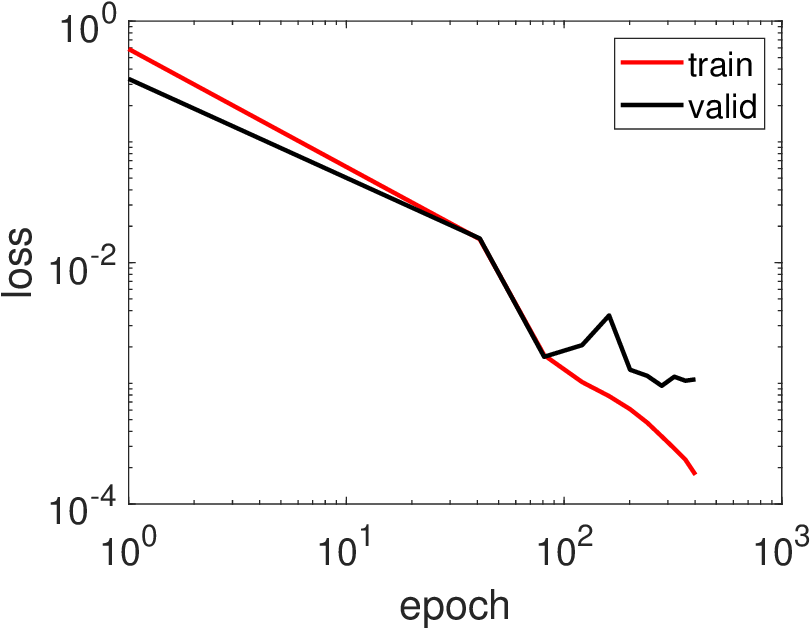} &
		\includegraphics[height=3.7cm]  {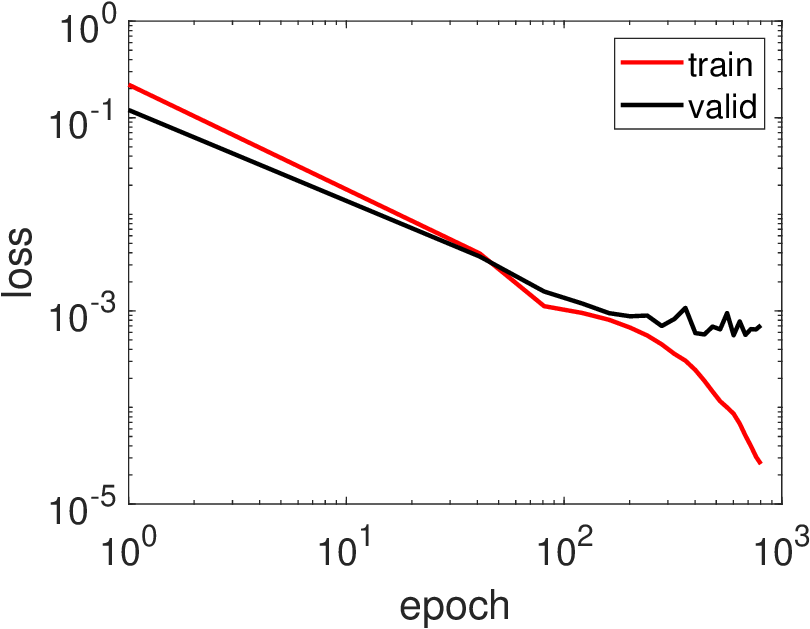} \\
        (a) (i) & (b) (ii) & (c) (iii)
	\end{tabular}
	\caption{\label{fig:exam2loss} The dynamics of the training error and validation loss for Example \ref{exam2}.}
\end{figure}

\section{Conclusions}

In recent years, deep learning methods have garnered significant traction for solving linear and nonlinear inverse problems, including electrical impedance tomography. Postprocessing type methods represent a widely employed class of deep learning based methods for reconstructing EIT images, but may exhibit instability in practice. In this work, we have investigated the deep Calder\'on method for anisotropic conductivity reconstruction. We theoretically established a Lipschitz stability for the inverse problem within a certain admissible class in both two- and high-dimensional cases. Moreover, we relate the stability estimates to the robustness of the deep Calder\'{o}n method on out-of-distribution data, by employing suitable training dataset. We complemented the theoretical findings by empirical evaluations.

\section*{Acknowledgments}
The authors are grateful to the two anonymous referees for their constructive comments and providing reference pointers which have greatly improved the quality of the paper. 

\bibliographystyle{siam}
\bibliography{dcm}

\end{document}